\documentclass{imsart}

\RequirePackage[OT1]{fontenc}
\RequirePackage{amsthm,amsmath,natbib, amssymb,color,amsfonts}
\RequirePackage[colorlinks,citecolor=blue,urlcolor=blue]{hyperref}
\usepackage{latexsym}

\usepackage{graphics}
\usepackage{graphicx}
\usepackage{epsfig}
\usepackage{mathrsfs}

\arxiv{arXiv:1806.09369}

\newcommand{\levy}{L\'evy}
\newcommand{\slln}{strong law of large numbers}
\newcommand{\clt}{central limit theorem}

\newcommand{\spr}{stochastic process}

\newcommand{\ex}{{\rm e}\,}

\newcommand{\asy}{asymptotic}

\newcommand{\ts}{time series}

\numberwithin{equation}{section}
\theoremstyle{plain}

\newtheorem{lemma}{Lemma}[section]

\newtheorem{theorem}[lemma]{Theorem}

\newtheorem{proposition}[lemma]{Proposition}
\newtheorem{definition}[lemma]{Definition}
\newtheorem{corollary}[lemma]{Corollary}
\newtheorem{example}[lemma]{Example}
\newtheorem{exercise}[lemma]{Exercise}
\newtheorem{remark}[lemma]{Remark}
\newtheorem{fig}[lemma]{Figure}
\newtheorem{tab}[lemma]{Table}

\newcommand{\bth}{\begin{theorem}}
\newcommand{\ethe}{\end{theorem}}

\newcommand{\bre}{\begin{remark}\em }
\newcommand{\ere}{\end{remark}}

\newcommand{\ble}{\begin{lemma}}
\newcommand{\ele}{\end{lemma}}

\newcommand{\bde}{\begin{definition}}
\newcommand{\ede}{\end{definition}}
\newcommand{\bco}{\begin{corollary}}
\newcommand{\eco}{\end{corollary}}

\newcommand{\bpr}{\begin{proposition}}
\newcommand{\epr}{\end{proposition}}

\newcommand{\bexer}{\begin{exercise}}
\newcommand{\eexer}{\end{exercise}}

\newcommand{\bexam}{\begin{example}}
\newcommand{\eexam}{\end{example}}

\newcommand{\bfi}{\begin{fig}}
\newcommand{\efi}{\end{fig}}

\newcommand{\btab}{\begin{tab}}
\newcommand{\etab}{\end{tab}}

\newcommand{\rv}{random variable}

\newcommand{\var}{{\rm var}}

\newcommand{\cov}{{\rm cov}}

\newcommand{\rhs}{right-hand side}

\newcommand{\beao}{\begin{eqnarray*}}
\newcommand{\eeao}{\end{eqnarray*}\noindent}

\newcommand{\beam}{\begin{eqnarray}}
\newcommand{\eeam}{\end{eqnarray}\noindent}

\newcommand{\beqq}{\begin{equation}}
\newcommand{\eeqq}{\end{equation}\noindent}

\newcommand{\bce}{\begin{center}}
\newcommand{\ece}{\end{center}}

\newcommand{\barr}{\begin{array}}
\newcommand{\earr}{\end{array}}

\newcommand{\stp}{\stackrel{\P}{\rightarrow}}
\newcommand{\std}{\stackrel{d}{\rightarrow}}
\newcommand{\stas}{\stackrel{\rm a.s.}{\rightarrow}}

\newcommand{\vague}{\stackrel{\lower0.2ex\hbox{$\scriptscriptstyle
                    \it{v} $}}{\rightarrow}}
\newcommand{\weak}{\stackrel{\lower0.2ex\hbox{$\scriptscriptstyle
                    \it{w} $}}{\rightarrow}}
\newcommand{\what}{\stackrel{\lower0.2ex\hbox{$\scriptscriptstyle
                    \it{\hat{w}} $}}{\rightarrow}}

\newcommand{\bdis}{\begin{displaymath}}
\newcommand{\edis}{\end{displaymath}\noindent}

\newcommand{\R}{\mathbb{R}}

\newcommand{\nto}{n\to\infty}

\newcommand{\ov}{\overline}
\newcommand{\wt}{\widetilde}
\newcommand{\wh}{\widehat}
\newcommand{\vep}{\varepsilon}

\newcommand{\la}{\lambda}

\newcommand{\bbr}{{\mathbb R}}

\newcommand{\bbf}{{\mathcal F}}

\newcommand{\BM}{Brownian motion}

\newcommand{\con}{convergence}

\newcommand{\st}{such that}
\newcommand{\fif}{if and only if}
\newcommand{\wrt}{with respect to}
\newcommand{\chf}{characteristic function}
\newcommand{\fct}{function}

\newcommand{\ds}{distribution}

\newcommand{\rep}{representation}

\newcommand{\seq}{sequence}

\newcommand{\pro}{probabilit}

\newcommand{\ms}{measure}

\newcommand{\bfx}{{\bf x}}
\newcommand{\bfX}{{\bf X}}

\newcommand{\bfY}{{\bf Y}}

\newcommand{\bfZ}{{\bf Z}}

\newcommand{\bft}{{\bf t}}

\newcommand{\bfs}{{\bf s}}

\newcommand{\E }{{\mathbb E}}
\renewcommand{\P }{{\mathbb P}}

\newcommand{\1}{{\mathbf 1}}

\allowdisplaybreaks

\begin{document}

\begin{frontmatter}
\title{Distance covariance for discretized stochastic processes}
\runtitle{Distance covariance for discretized stochastic processes}


\begin{aug}
\author{\fnms{Herold} \snm{Dehling}\thanksref{a}\ead[label=e1]{herold.dehling@rub.de}},
\author{\fnms{Muneya} \snm{Matsui}\thanksref{b}\ead[label=e2]{mmuneya@gmail.com}},
\author{\fnms{Thomas} \snm{Mikosch}\thanksref{c}
\ead[label=e3]{mikosch@math.ku.dk}
\ead[label=u1,url]{http://www.foo.com}}, 
\author{\fnms{Gennady} \snm{Samorodnitsky}\thanksref{d} \ead[label=e4]{gs18@cornell.edu}}
\and
\author{\fnms{Laleh} \snm{Tafakori}\thanksref{e} \ead[label=e5]{laleh.tafakori@unimelb.edu.au}}

\address[a]{Department of Mathematics, 
Ruhr-Universit\"at Bochum, 
44780 Bochum, Germany
\printead{e1}} 

\address[b]{Department of Business Administration, Nanzan University, 
18 Yamazato-cho, Showa-ku,  
Nagoya 466-8673, Japan
\printead{e2}}

\address[c]{Department  of Mathematics,
University of Copenhagen,
Universitetsparken 5,
DK-2100 Copenhagen,
Denmark
\printead{e3}}

\address[d]{School of Operations Research  and Information\\ Engineering, 
Cornell University, 
220 Rhodes Hall, 
Ithaca, NY 14853, U.S.A.
\printead{e4}
}

\address[e]{
School of Mathematics and Statistics, 
University of Melbourne,
Richard Berry Building,
Parkville, 
3010, Melbourne, 
Australia 
\printead{e5}}

\runauthor{H. Dehling et al.}


\end{aug}

\begin{abstract}
Given an iid \seq\ of pairs of \spr es on the unit interval we
construct a \ms\ of independence for the components of the pairs.
We define distance covariance and distance correlation based
on approximations of the component processes at finitely many discretization points. Assuming that the mesh of the discretization
converges to zero as a suitable \fct\ of the sample size, we show 
that the sample distance covariance and correlation converge to limits 
which are zero \fif\ the component processes are independent. 
To construct a test for independence of the discretized 
component processes we show consistency of the bootstrap for the 
corresponding sample distance covariance/correlation. 
\end{abstract}


\end{frontmatter}

\section{Introduction}\setcounter{equation}{0}
\subsection{Distance covariance and distance correlation for vectors}
In a series of papers, 
\cite{szekely:rizzo:bakirov:2007,szekely:rizzo:2009,szekely:rizzo:2013,szekely:rizzo:2014} 
introduced {\em distance covariance} and {\em distance correlation}. They are \ms s of the dependence between two vectors $\bfX$ and $\bfY$, 
possibly with different
dimensions. These \ms s have the desirable property that 
they are zero \fif\ $\bfX$ and $\bfY$ are independent. This is in contrast to
many other dependence \ms s  
where one can only make statements about certain aspects of the dependence
between $\bfX$ and $\bfY$. For example, the correlation and covariance
between two real-valued \rv s $X$ and $Y$  allow one to make statements 
about their linear dependence. 
\par
The distance covariance between a $p$-dimensional vector $\bfX$ and a 
$q$-dimensional vector $\bfY$ is a weighted version of the squared 
distance between
the joint \chf\  $\varphi_{\bfX,\bfY}$ of $\bfX$, $\bfY$
and the product of the marginal \chf s
$\varphi_\bfX$, $\varphi_\bfY$  of these vectors. We know that
$\bfX$ and $\bfY$ are independent \fif\  
\beam\label{eq:indch}
\varphi_{\bfX,\bfY}(\bfs,\bft)=\varphi_\bfX(\bfs)\,\varphi_\bfY(\bft)\,,\qquad
\bfs\in\bbr^p\,,\bft\in\bbr^q\,.
\eeam
However, this identity is difficult to check if one has data 
at the disposal; a replacement of the corresponding \chf s 
by empirical versions does not lead to powerful statistical tools for
detecting independence between $\bfX$ and $\bfY$.
First, \cite{feuerverger:1993} in the univariate case and,
later, \cite{szekely:rizzo:bakirov:2007,szekely:rizzo:2009,szekely:rizzo:2013,szekely:rizzo:2014} in the general multivariate case recommended to
use a weighted $L^2$-distance between  $\varphi_{\bfX,\bfY}$ and 
$\varphi_\bfX\,\varphi_\bfY$: for $\beta\in (0,2)$, the {\em distance covariance between $\bfX$ and $\bfY$} is given by
\beao
T_\beta(\bfX,\bfY)= c_pc_q\int_{\bbr^{p+q}} \big|\varphi_{\bfX,\bfY}(\bfs,\bft)-\varphi_\bfX(\bfs)\varphi_\bfY(\bft)\big|^2 
|\bfs|^{-(p+\beta)}|\bft|^{-(q+\beta)}\,d\bfs d\bft\,,
\eeao
where the constants $c_d$ for $d\ge 1$  are chosen \st\
\beao
c_d\,\int_{\bbr^d} (1-\cos(\bfs'\bfx))\,|\bfx|^{-(d+\beta)}d\bfx= |\bfs|^\beta\,.
\eeao
Here and in what follows we suppress the dependence of the 
Euclidean norm $|\cdot|$ on the dimension; it will always be clear
from the context what the dimension is.
The quantity $T_\beta(\bfX,\bfY)$ is finite under suitable moment conditions on $\bfX,\bfY$.
The corresponding {\em distance correlation} is given by
\beao
R_\beta(\bfX,\bfY)= \dfrac{T_\beta(\bfX,\bfY)}{\sqrt{T_\beta(\bfX,\bfX)}\sqrt{T_\beta(\bfY,\bfY)}}\,.
\eeao
\par 
An advantage of choosing the particular weight \fct\ $|\bfs|^{-(p+\beta)}|\bft|^{-(q+\beta)}$ is that the distance covariance has an explicit form: for iid copies
$(\bfX_i,\bfY_i)$, $i=1,2,\ldots,$ of $(\bfX,\bfY)$ we have
\beam\label{eq:1}
T_\beta(\bfX,\bfY)&=& \E [|\bfX_1-\bfX_2|^\beta|\bfY_1-\bfY_2|^\beta]+  
\E [|\bfX_1-\bfX_2|^\beta]\E[|\bfY_1-\bfY_2|^\beta]\nonumber\\&& -
2\,\E [|\bfX_1-\bfX_2|^\beta|\bfY_1-\bfY_3|^\beta]\,.
\eeam 
The weight \fct\ ensures that $T_\beta(c\bfX,c\bfY)=c^{2\beta}T_\beta(\bfX,\bfY)$
for any constant $c$, hence $R_\beta(c\bfX,c\bfY)$ does not depend on $c$,
i.e., the distance correlation is scale invariant. 
A corresponding theory can be built on non-homogeneous kernels as well; see
the discussion and references 
in \cite{dmmw:2016} who consider auto- and cross-distance correlation \fct s for \ts .

It is clear from the construction that $T_\beta(\bfX,\bfY)=
R_\beta(\bfX,\bfY)=0$ \fif\ \eqref{eq:indch} holds. This observation
motivates the construction of sample versions of $T_\beta(\bfX,\bfY)$ 
and $R_\beta(\bfX,\bfY)$ and one hopes that these have properties similar
to their deterministic counterparts. In particular, one would like to  
test independence between $\bfX$ and $\bfY$. 
\par
Replacing the \chf s in 
$T_\beta(\bfX,\bfY)$ and $R_\beta(\bfX,\bfY)$ by their sample analogs 
and taking into account \eqref{eq:1}, we obtain the sample versions of $T_\beta(\bfX,\bfY)$ and $R_\beta(\bfX,\bfY)$:
\beao
T_{n,\beta}(\bfX,\bfY) &=& \dfrac 1 {n^2}\sum_{k,l=1}^n
|\bfX_k-\bfX_l|^\beta |\bfY_k-\bfY_l|^\beta \\
&& + \dfrac 1 {n^2}\sum_{k,l=1}^n |\bfX_k-\bfX_l|^\beta{ \frac 1 {n^2}}
\sum_{k,l=1}^n |\bfY_k-\bfY_l|^\beta \\
&&-2\,\dfrac 1 {n^3}\sum_{k,l,m=1}^n |\bfX_k-\bfX_l|^\beta |\bfY_k-\bfY_m|^\beta\,,\\
R_{n,\beta}(\bfX,\bfY) &=& \dfrac{T_{n,\beta}(\bfX,\bfY)}{\sqrt{T_{n,\beta}(\bfX,\bfX)}\sqrt{T_{n,\beta}(\bfY,\bfY)}}\,.
\eeao
\par
The quantity $T_{n,\beta}(\bfX,\bfY)$ is a $V$-statistic; cf. \cite{szekely:rizzo:bakirov:2007},
\cite{lyons:2013}. Therefore standard theory 
yields a.s. consistency,
\beao
T_{n,\beta}(\bfX,\bfY)\stas T_\beta(\bfX,\bfY)\,,\qquad\nto\,,
\eeao
under suitable moment conditions; see \cite{hoffmann:1994}, \cite{serfling:1980}. If $\bfX$ and $\bfY$
are independent the $V$-statistic  $T_{n,\beta}(\bfX,\bfY)$ is degenerate of order 1.
Under suitable moment conditions, one also has the weak \con\ of $n\,T_{n,\beta}(\bfX,\bfY)$  to a 
weighted sum of iid $\chi^2$-variables; see \cite{serfling:1980}, \cite{lyons:2013}, \cite{arcones:gine:1992}. 
Moreover, $V$-statistics theory
also ensures that $T_{n,\beta}(\bfX,\bfX)\stas T_\beta(\bfX,\bfX)$ and  
$T_{n,\beta}(\bfY,\bfY)\stas T_\beta(\bfY,\bfY)$. Hence $R_{n,\beta}(\bfX,\bfY)$ is an a.s. consistent
estimator of $R_\beta(\bfX,\bfY)$ and, modulo a change of scale, $n R_{n,\beta}(\bfX,\bfY)$ has the same weak limit as $T_{n,\beta}(\bfX,\bfY)$.

\subsection{Distance covariance and distance correlation for \spr es}
\cite{szekely:rizzo:2013} considered the situation 
when $\bfX$ and $\bfY$ are independent and have iid components, $n$ is fixed, $p=q \to\infty$. Under these conditions,  
$R_{n,\beta}(\bfX,\bfY)$ converges to 1. In this way, they justified
the empirical observation that $R_{n,\beta}(\bfX,\bfY)$ is close to 1 
if $p,q$ are large relative to $n$. 
\par
\cite{matsui:mikosch:samorodnitsky:2017} considered a version of the distance covariance 
for stochastic processes $X,Y$ on $[0,1]$, where it was 
assumed that the two processes
are observed at a Poisson number of points in $[0,1]$. Via simulations
the resulting estimator was compared with the distance correlation 
$R_{n,\beta}(\bfX,\bfY)$ where the components of the 
iid vectors $(\bfX_i,\bfY_i)$ consist of a 
Poisson number of the discretizations of $(X_i,Y_i)$, respectively. Both types of estimators exhibited
a similar behavior for independent $X$ and $Y$, approaching zero 
for moderate sizes $n,p,q$. A possible explanation for this phenomenon is 
that \cite{matsui:mikosch:samorodnitsky:2017} and  \cite{szekely:rizzo:2013}
worked under quite distinct conditions.
\cite{szekely:rizzo:2013} considered vectors 
$\bfX$ and $\bfY$ with iid components
whose dimensions increase to infinity for a fixed sample size $n$. In  \cite{matsui:mikosch:samorodnitsky:2017}, $\bfX$ and $\bfY$ can be understood as  
vectors of discretizations 
of genuine \spr es $X,Y$ on $[0,1]$, such as \BM, fractional \BM , \levy\ processes, etc.
In these cases, the components of $\bfX_i$ and $\bfY_i$ are dependent. 
\par
In this paper, we again take up the theme of \cite{szekely:rizzo:2013}
and \cite{matsui:mikosch:samorodnitsky:2017}. We consider two
processes $X$ and $Y$ on $[0,1]$, which we assume to be stochastically
continuous, measurable and bounded. 
In contrast to \cite{matsui:mikosch:samorodnitsky:2017},
\begin{itemize}
\item
we consider discretizations of these processes 
at a partition $0=t_0<t_1<\cdots <t_p=1$ of $[0,1]$, assuming that  $p=p_n\to\infty$ as $\nto$ and the mesh satisfies
\beao
\delta_n=\max_{i=1,\ldots,p} (t_i-t_{i-1})\to 0\,,\quad \nto\,,
\eeao
\item 
we normalize the points $X(t_i)$ and $Y(t_i)$ by $\sqrt{t_i-t_{i-1}}$.
\end{itemize} 
In the sequel, we suppress the dependence of $p$ on $n$.
It will be convenient to write for any partition $(t_i)$ and a 
process $Z$ on $[0,1]$,
\beao
\Delta_i=(t_{i-1},t_i]\,,|\Delta_i|=t_i-t_{i-1}\,,\ i=1,\ldots,p\,,\ \Delta Z(s,t]= Z(t)-Z(s)\,,s<t.
\eeao
We consider a vector of weighted discretizations 
\beam\label{eq:78}
\bfZ_p=\big(|\Delta_1|^{1/2}Z(t_1),\ldots,|\Delta_p|^{1/2} Z(t_p)\big)\,,
\eeam
and define 
\beao
Z^{(p)}(t) = \sum_{i=1}^p Z(t_i)
\1 (t \in \Delta_i)\,,\qquad t\in [0,1]\,.
\eeao   
For stochastically continuous, measurable and bounded processes $Z$ and $Z'$ 
we have 
\beao
|\bfZ_{p}-\bfZ'_{p}|^{2}&=& \sum_{i=1}^p (Z(t_i)-Z'(t_i))^2 |\Delta_i|=
\|Z^{(p)}-(Z')^{(p)}\|_2^2\\
&\to& \int_0^1 (Z(t)-Z'(t))^2\,dt=\|Z-Z'\|_2^2\,,\quad p\to\infty\,,
\eeao
in probability, where $\|\xi\|_2$ denotes the $L^2$-norm of a process $\xi$ on $[0,1]$.
\par
For $\beta\in (0,2]$,
we introduce a \spr\ analog $T_\beta(X,Y)$ of $T_\beta(\bfX,\bfY)$ 
from \eqref{eq:1}. Consider an iid \seq\ $(X_i,Y_i)$, $i=1,2,\ldots,$
of processes $X_i,Y_i$ on 
$[0,1]$ with generic element $(X,Y)$ which is also 
  stochastically continuous, measurable and bounded. Define
\beam\label{eq:t}
T_\beta(X,Y)
&=&\E \big[\|X_1-X_2\|_2^{\beta}\|Y_1-Y_2\|_2^\beta\big] + 
\E\big[\|X_1-X_2\|_2^{\beta}\big]\,\E\big[ \|Y_1-Y_2\|_2^\beta\big]\nonumber\\ 
&&-2\, \E \big[\|X_1-X_2\|_2^{\beta}\,\|Y_1-Y_3\|_2^{\beta}\big]\,,
\eeam
where we assume that all moments involved are finite. Of course, 
$T_\beta(X,Y)=0$ for independent $X,Y$. The converse is not obvious; 
we prove it in Section~\ref{sec:gauss}. 
\par
The sample analog of $T_\beta(X,Y)$ is given by
\beam\label{eq:decompo}
T_{n,\beta}(X,Y)&=&
\dfrac 1 {n^2}\sum_{k,l=1}^n \|X_k-X_l\|_2^\beta \|Y_k-Y_l\|_2^\beta \nonumber\\
&& + \dfrac 1 {n^2}\sum_{k,l=1}^n \|X_k-X_l\|_2^\beta 
\dfrac 1 {n^2} \sum_{k,l=1}^n \|Y_k-Y_l\|_2^\beta\nonumber\\&&
-2\,\dfrac 1 {n^3}\sum_{k,l,m=1}^n \|X_k-X_l\|_2^\beta \|Y_k-Y_m\|_2^\beta\nonumber\\
&=:&\mathrm{I}_1 +\mathrm{I}_3 -2 \mathrm{I}_2\,.
\eeam
Assuming that the moments in $T_\beta(X,Y)$ are finite,
the \slln\ for $V$-statistics yields
\beao
T_{n,\beta}(X,Y)\stas T_\beta(X,Y)\,,\qquad\nto\,.
\eeao
This fact and the observation that $T_\beta(X,Y)$ vanishes for independent $X,Y$
encourage one to call $T_\beta(X,Y)$ the {\em distance covariance between 
$X,Y$}, and $T_{n,\beta}(X,Y)$ its sample version.  The corresponding 
distance and sample distance correlations
$R_\beta(X,Y)$ and 
$R_{n,\beta}(X,Y)$ are defined in the natural way.  

\subsection{Objectives}
We imagine that the coastline of a country (like the 
Netherlands) can be mapped to the interval $[0,1]$ and, at each location
$s\in [0,1]$ and on each day $i$, we have an observation of the height
of sea waves, $X_i$, and the corresponding wind-speed, $Y_i$. 
An interesting question is whether the processes $X_i$ and $Y_i$ are independent. Similarly, we can think of two price processes $X_i$ and $Y_i$ on day $i$ 
given on the interval of the working hours of the stock exchange.
Natural questions are as to whether the two price processes are independent
on the same day and how much serial dependence there is in each of 
the series $(X_i)$ and $(Y_i)$ and between them. In the first case, one is interested in testing
the independence of the processes $X_i$ and $Y_i$. In the second case,
one is interested in testing the independence of $X_i$ and $X_{i+h},Y_{i+h}$ 
for positive lags.
\par
Typically, we will  not have  complete sample paths of $(X_i,Y_i)$
at our disposal. In this paper, we assume that we observe a sample
$\big((X_i^{(p)},Y_i^{(p)})\big)_{i=1,\ldots,n}$ consisting of 
discretizations taken from  an iid 
\seq\ $((X_i,Y_i))_{i=1,2,\ldots}$ on the same partition $(t_i)_{i=0,\ldots,p}$ of 
$[0,1]$. We can define the corresponding sample distance covariance
$T_{n,\beta}(X^{(p)},Y^{(p)})$ and sample distance correlation 
$R_{n,\beta}(X^{(p)},Y^{(p)})$. In view of the discussion above we see that
the latter quantities coincide with the corresponding 
quantities $T_{n,\beta}(\bfX_p,\bfY_p)$ and  $R_{n,\beta}(\bfX_p,\bfY_p)$
where $\bfX_p$ and $\bfY_p$ are defined through \eqref{eq:78}. In the case of an equidistant partition with mesh $\delta_n=1/p$ we also observe  
that $R_{n,\beta}(\bfX_p,\bfY_p)$ is exactly the classical sample distance correlation 
$R_{n,\beta}(\bfX,\bfY)$
of the vectors $\bfX=(X(j/p))_{j=1,\ldots,p}$ and $\bfY=(Y(j/p))_{j=1,\ldots,p}$.  
\par
The main goal of this paper is to show that for independent $X,Y$,
\beam\label{eq;kl}
n\,\big(T_{n,\beta}(X^{(p)},Y^{(p)})-T_{n,\beta}(X,Y)\big)\stp 0\,,\qquad \nto\,,
\eeam 
provided $\delta_n\to 0$ and $p=p_n\to\infty$
sufficiently fast. In turn, we will be able to 
exploit the existing limit theory for the normalized degenerate 
$V$-statistic $n\,T_{n,\beta}(X,Y)$ 
to derive the \ds al limit of 
$n\,T_{n,\beta}(X^{(p)},Y^{(p)})$. This limit has a weighted $\chi^2$-\ds\ 
which is not easily evaluated. We will show that
bootstrap versions of the degenerate $V$-statistics $n\,T_{n,\beta}(X,Y)$ and 
$n\,T_{n,\beta}(X^{(p)},Y^{(p)})$ are close in the sense of Mallows metrics
and have the same \ds al limit as $n\,T_{n,\beta}(X,Y)$.
\par
The paper is organized as follows. In Section~\ref{sec:technical}
we introduce various technical conditions and discuss their applicability
to some  classes of \spr es. 
The main results of Theorem~\ref{thm:1} yield sufficient conditions
for \eqref{eq;kl} and the corresponding versions for the 
distance correlations, assuming independence 
between $X,Y$. The proof is given in Section~\ref{sec:proofthm2.1}.
The bootstrap for $T_{n,\beta}(X^{(p)},Y^{(p)})$
is discussed  in Section~\ref{sec:bootstrap}. There we show that
a suitable bootstrap version of  $T_{n,\beta}(X^{(p)},Y^{(p)})$ is consistent.
The results of Section~\ref{sec:gauss}
may be of independent interest. There we show that $T_\beta(X,Y)=0$ implies
independence of the integrals $\int X dB_1$ and $\int Y dB_2$ conditional on 
$B=(B_1,B_2)$ which has independent \BM\ components on $[0,1]$ and is independent of $(X,Y)$.
In turn, the conditional independence of these integrals implies independence
of $X,Y$.
We give a small simulation study in Section~\ref{sec:simulation} 
which shows that the theoretical results work for small and moderate
values of $n$ and $p$.

\section{Technical conditions}\label{sec:technical}\setcounter{equation}{0}
To derive the results in Section~\ref{sec:main} we assume various conditions 
on the smoothness and moments of the processes $X,Y$ and their relation
with  the parameters of the partition, in particular $p$ and $\delta_n$.
Throughout $\beta\in (0,2)$ is fixed. If any of the processes $X,Y$ 
have finite expectation we assume that they are centered.
\par
We will work under two distinct settings: (1) finite variance of $X,Y$ and 
(2) $X,Y$ have finite $\beta$th moment.

\subsection{The finite variance case}
If $X,Y$ have finite second moments we will work under the set of conditions
(A):
\begin{enumerate}
\item[\rm (A1)] {\em Smoothness of increments.} There exist
$\gamma_X,\gamma_Y>0$ and $c>0$ \st\
\beao
\var\big(\Delta X(s,t]\big)\le c\,|t-s|^{\gamma_X}\quad\mbox{and}\quad \var\big(\Delta Y(s,t]\big)\le c\,|t-s|^{\gamma_Y}\,,\quad s<t\,.
\eeao
\item[\rm (A2)] {\em Growth condition on $p=p_n\to\infty$.} We have
\beao
\delta_n= o\big(n^{ -2/{((\gamma_X\wedge\gamma_Y)(\beta\wedge 1))}}\big)\,,\qquad\nto\,.
\eeao
\item[\rm (A3)] {\em Additional moment conditions.}
If $\beta\in (1,2)$ 
we have
\beao
&&\max_{0\le t\le 1} \E[|X(t)|^{2(2\beta-1)}]+\max_{0\le t\le 1} \E[|Y(t)|^{2(2\beta-1)}]<\infty\,.
\eeao
\end{enumerate}
\subsection{The finite $\beta$th moment case}
If $X,Y$ possibly have infinite 
second moments we will work under the set of conditions
(B):
\begin{enumerate}
\item[\rm (B1)] {\em Finite $\beta$th moment.}
\beao
\E\big[\max_{t\in (0,1]} |X(t)|^{\beta}\big]<\infty\;\;\mbox{and}\;\; 
\E\big[\max_{t\in (0,1]} |Y(t)|^{\beta}\big]<\infty\,,
\eeao
\item[\rm (B2)]
{\em Smoothness of increments.}
There exist $\gamma_X,\gamma_Y>0$ and $c>0$ \st\  
\beao
\max_{i=1,\ldots,p}\E\big[\max_{t\in \Delta_i}|\Delta
	     X(t,t_i]|^{\beta}\big]\le c\, \delta_n^{\gamma_X}\ \mbox{and}\ 
\max_{i=1,\ldots,p}\E\big[\max_{t\in \Delta_i}|\Delta Y(t,t_i]|^{\beta}\big]\le c\, \delta_n^{\gamma_Y}\,.
\eeao
\item[\rm (B3)]  {\em Additional moment and smoothness conditions.}
If $\beta\in (0,1)$ we also have
\beao
\E\big[\max_{0\le t\le 1} |X(t)|^{2\beta}\big]<\infty
	     \quad\mbox{and}\quad
\E\big[\max_{0\le t\le 1} |Y(t)|^{2\beta}\big]<\infty\,,
\eeao
and there exist $\gamma_X',\gamma_Y'>0$ and $c>0$ \st\
\beao
\max_{i=1,\ldots,p}\E\big[\max_{t\in \Delta_i}|\Delta
	     X(t,t_i]|^{2\beta}\big]\le c\, \delta_n^{\gamma_X'}\
	     \mbox{and}\ 
\max_{i=1,\ldots,p}\E\big[\max_{t\in \Delta_i}|\Delta Y(t,t_i]|^{2\beta}\big]\le c\, \delta_n^{\gamma_Y'}\,.
\eeao
\item[\rm (B4)] {\em Growth condition on $p=p_n\to\infty$.} We have
\beao
\delta_n=
o \Big(\big(p\,n^{\beta/(\beta\wedge 1)}\big)^{-\frac {1}{\beta/2+\gamma_X\wedge\gamma_Y}}\Big)\,.
\eeao
\end{enumerate}

\subsection{Discussion of the conditions and examples}\label{subsec:discuss}
 \bre\label{rem:1}
In the proofs we will need the conditions
\beam\label{eq:aug4a}
\mbox{$\E[\|X\|_2^{\beta}]<\infty$ and $\E[\|Y\|_2^{\beta}]<\infty$ for some $\beta\in (0,2)$.}
\eeam 
If (A1) holds (in particular, $\sup_{t\in [0,1]}\big[\var(X(t))+\var(Y(t))\big]<\infty$) 
\eqref{eq:aug4a} is automatic because by Jensen's inequality 
\beao
\E[\|X\|_2^{\beta}]= \E\Big[\Big(\int_0^1 (X(t))^2\,dt\Big)^{\beta/2}\Big]\le \Big(\int_0^1\var(X(t))\,dt\Big)^{\beta/2}<\infty\,.
\eeao
The same argument also shows that $\E [\|X\|_2^2]<\infty$ under (A1).
If (B1) holds then \eqref{eq:aug4a} follows.
\ere
\bre
In the case of an equidistant partition we have $\delta_n=1/p$.
Then the growth condition (A2) reads as
\beam\label{eq:m1}
\dfrac{p}{n^{ \frac 2{(\gamma_X\wedge\gamma_Y)\,(\beta\wedge 1)}}}\to \infty\,,\qquad \nto\,,
\eeam
while (B4) takes on the form 
\beam\label{eq:m2}
\dfrac {p} 
{n^{\frac{\beta}{(\beta/2+\gamma_X\wedge\gamma_Y-1)(\beta\wedge 1)}}}\to\infty\,,\qquad \nto\,,
\eeam
provided one can ensure that $\beta/2+\gamma_X\wedge\gamma_Y>1$.
The message from \eqref{eq:m1}
is that we need to choose $p$ the larger
the smaller $\gamma_X\wedge\gamma_Y$ is, i.e., the rougher the sample paths.
Similarly, for $\beta<1$, $p$ needs to be chosen the larger 
the smaller $\beta$ is.  Similar comments apply to \eqref{eq:m2}. 
\ere
\bexam\label{exam:hurst}\rm 
Assume that $X,Y$ are sample continuous  self-similar processes with  
stationary increments and a finite variance. If the corresponding Hurst
exponents are $H_X,H_Y\in (0,1)$  then for some $c_X>0$,
\beao
\var( \Delta X(s,t])= \var( X(0,t-s])= c_X\,(t-s)^{2H_X}\,,\qquad s<t\,,
\eeao
and similarly for $Y$. That is, we  can choose $\gamma_X= 2 H_X$ and
$\gamma_Y= 2H_Y$ in (A1). Furthermore, 
(A3) holds for $X$ 
 if $\beta\in (1,2)$ and $\E\big[ |X(1)|^{2(2\beta-1)}\big]<\infty$,
 and similarly for $Y$. A special case is that of Gaussian $X$ and $Y$
 which then are  {\em fractional \BM s}, 
and (A3) trivially holds.  A process with the same covariance
structure is the {\em fractional L\'evy process} 
\beao
X(t)= \int_\bbr\big((t-s)_+^{H_X-0.5} -
(-s)_+^{H_X-0.5}\big)\,dL(s)\,,\qquad t\in \bbr\,,H_X\in (0.5,1)\,, 
\eeao 
where $L$ is a two-sided \levy\ process on $\bbr$ with 
mean zero and finite variance, introduced in 
\cite{Marquardt:2006}. This process is not self-similar (unless $L$ is
a Brownian motion) but has stationary increments.  Here (A1) holds with 
$\gamma_X= 2 H_X$ and $\gamma_Y= 2H_Y$. Furthermore, (A3) holds  if
$\E [|L(1)|^{2(2\beta-1)}]<\infty$. 

Notice also that {\em any centered Gaussian processes $X$ and $Y$}
satisfying (A1) have automatically continuous 
sample paths and (A3) is satisfied. 
\eexam

\bexam\rm Assume that $X$ and $Y$ are It\^o integrals, i.e., there are two
Brownian motions $B_X,B_Y$ and predictable processes $Z_X,Z_Y$
\wrt\ the corresponding Brownian filtrations \st\ 
\beao
X(t)= \int_0^t Z_X(s)\, d B_X(s)\,,\qquad Y(t)=\int_0^t Z_Y(s)\, d B_Y(s)\,,\quad
0\le t\le 1\,.
\eeao
Then we have 
\beao
\var\big(\Delta X(s,t]\big)= \int_s^t \E[Z_X^2(x)]\, dx\,,\quad s<t\,.
\eeao
Hence, if $c_X= \sup_{x\in [0,1]}\E[Z_X^2(x)]<\infty$, then
$\var\big(\Delta X(s,t]\big)\le c_X\,(t-s)\,,$
and one can choose $\gamma_X=1$ in (A1). Moreover, (A3) holds for $X$
if $\beta\in (1,2)$ and $\E [|X(1)|^{2(2\beta-1)}]<\infty$. This follows
from an application of Doob's maximal inequality for martingales.
Similar arguments apply to the process $Y$.  A special case is that of zero
drift geometric Brownian motions; a simple computation shows that
nothing changes even when the drift is not zero. 
\par
In the equidistant case we conclude from \eqref{eq:m1} that (A2)
holds if 
\beam\label{eq:m3}
\dfrac{p}{n^{\frac 2{\beta\wedge 1}}}\to\infty\,,\qquad\nto\,.
\eeam 
\eexam
\bexam\rm 
For $\alpha\in (0,2)$ sample continuous self-similar S$\alpha$S processes with stationary
increments provide a family of examples with an infinite second moment. 
For such processes (B1) is satisfied
for $\beta<\alpha$ and (B2) is satisfied with $\gamma_X=\gamma_Y=\beta
H$, where $H$ is the Hurst exponent. This follows from
  continuity, self-similarity and stationarity of the
  increments. Similarly, (B3) holds if $\beta<\alpha/2$ and 
$\gamma_X'=\gamma_Y'=2\beta H$.
Such processes include the {\em fractional harmonizable $\alpha$-stable
motions} and, if $1<\alpha<2$ and 
$1/ \alpha<H<1$, also the {\em linear fractional stable  motions;} 
see Chapter 7 in 
\cite{samorodnitsky:taqqu:1994}. Another example is that 
of the {\em $\gamma$-Mittag Leffler fractional S$\alpha$S motion,} which is
an integral of a $\gamma$-Mittag Leffler 
process \wrt\ a suitable S$\alpha$S random \ms ; see 
\cite{Samorodnitsky:2016}, Section~8.4. Here $H=\gamma+(1-\gamma)/
\alpha$. 
\eexam
\bexam{\rm   \levy\ processes are stochastically continuous and bounded 
by definition.
If $X$ is a \levy\ process with finite second moment (A1) holds
because $\var(\Delta X(s,t))= c\,(t-s)$, for $s<t$ and a constant $c$.
Moreover, (A3) holds for $X$ if $\E[|X(1)|^{2(2\beta-1)}]<\infty$. Indeed, 
an application of \levy 's maximal inequality
yields for $t\in [0,1]$,
\beao
\E[|X(t)|^{2(2\beta-1)}]\le  \E[\max_{0\le t\le 1}|X(t)|^{2(2\beta-1)}]
\le c\,\E[|X(1)|^{2(2\beta-1)}]\,.
\eeao
Similarly, for $X$, (B1) holds if $\E[|X(1)|^\beta]<\infty$,
(B2) is satisfied if $\E[|\Delta X(s,t]|^\beta]\le c (t-s)^{\gamma_X}$,
and (B3) holds if $\E[|\Delta X (s,t]|^{2\beta}]\le c (t-s)^{\gamma_X'}$.
} 
\eexam

\section{Main results}\label{sec:main}\setcounter{equation}{0}

We would like to use the distance covariance to test for independence
of two stochastically continuous bounded stochastic processes $X,Y$ on $[0,1]$. By the \slln\ for
$V$-statistics we have  
\beam\label{eq:pl1}
T_{n,\beta}(X,Y)&\stas& T_\beta(X,Y)\,,
\eeam
where the limit is defined in \eqref{eq:t}.
If $X,Y$ are independent then $T_\beta(X,Y)=0$, and in Section
\ref{sec:gauss} we prove that, conversely,  
$T_\beta(X,Y)=0$ implies independence of $X,Y$. The following theorem
establishes, in particular, that under appropriate conditions, if 
$X,Y$ are independent, then also 
\beam\label{eq:pl}
T_{n,\beta}(X^{(p)},Y^{(p)})-T_{n,\beta}(X,Y)\stp 0
\eeam 
and, hence,
\beam\label{eq:slln}
T_{n,\beta}(X^{(p)},Y^{(p)})\stp 0\,.
\eeam
This relation can be used in testing for independence of $X,Y$. Note
that, if $X,Y$ are dependent the results of 
Section~\ref{sec:gauss} will imply that $T_\beta(X,Y)>0$ and so, by
\eqref{eq:pl1}   and \eqref{eq:pl}, we see that $n\,T_{n,\beta}(X^{(p)},Y^{(p)})\stp\infty$.
 
In fact, the limiting equivalence \eqref{eq:pl} holds 
  for dependent $X,Y$ as well, see the proof of Lemma \ref{lem:1},
as long as one imposes more 
restrictive moment conditions (due to the use of H\"older-type inequalities
for products of dependent \rv s). 

\par
 
In the theorem below we assume, without loss of generality, that $\E
[X(t)]=\E[Y(t)]=0$ for any $t\in [0,1]$, provided the 
expectations are finite. Indeed, $T_{n,\beta}$ contains expressions of the type $X_k-X_l$, $Y_k-Y_l$  or their discrete approximations.
Therefore we can always mean-correct $X_k$ and $Y_k$, without changing the value of $T_{n,\beta}$.  
\bth\label{thm:1} Assume the following conditions:
\begin{enumerate}
\item [\rm 1.]
$X,Y$ are independent stochastically continuous bounded processes on $[0,1]$ defined
on the same \pro y space. 
\item[\rm 2.]
If  $X,Y$ have finite expectations, then these are assumed to be equal
to 0.
\item[\rm 3.] $\delta_n\to0$ as $\nto$.
\item[\rm 4.] $\beta\in (0,2)$.
\end{enumerate} 
Then the  following statements hold.
\begin{enumerate}
\item[\rm (1)]
If either {\rm (A1)}  or $\big[\mbox{{\rm (B1),(B2)} and 
$p\,\delta_n^{\beta/2+\gamma_X\wedge\gamma_Y}\to 0$}\big]$ are
satisfied then \eqref{eq:pl} $($and, hence, \eqref{eq:slln}$)$ hold. 

\item[\rm (2)] If either {\rm (A1),(A2)} or {\rm (B1),(B2),(B4)} hold
then 
\beao
n\,T_{n,\beta}(X^{(p)},Y^{(p)}) \std \sum_{i=1}^\infty \la_i (N_i^2-1)+ c 
\eeao
for an iid \seq\ of standard normal \rv s $(N_i)$, a constant $c$, 
and a square summable
\seq\ $(\la_i)$. 

\item[\rm (3)] If either {\rm (A1),(A3)} or $\big[\mbox{$\beta\in (0,1)$
	     and {\rm (B1)-(B3)}}\ and\ p\,\delta_n^{\beta+\gamma_X' \wedge \gamma_Y'}\to 0
\big]$ hold then 
\beao
R_{n,\beta}(X^{(p)},Y^{(p)})\stp 0\,.
\eeao
\item[\rm (4)] If either {\rm (A1)-(A3)} or $\big[\mbox{$\beta\in
    (0,1)$ and  {\rm (B1)-(B4)}}\ and\ p\,\delta_n^{\beta+\gamma_X' \wedge \gamma_Y'}\to 0\big]$ hold then 
\beao
n\,R_{n,\beta}(X^{(p)},Y^{(p)}) \std \sum_{i=1}^\infty \la_i (N_i^2-1)+ c 
\eeao
for an iid \seq\ of standard normal \rv s $(N_i)$, a constant $c$, 
and a square summable \seq\ $(\la_i)$. 
\end{enumerate}
\ethe
\noindent
The proof is given in Section~\ref{sec:proofthm2.1}.
\bre 
In Appendix \ref{append:c} 
we discuss the \asy\ behavior of 
$T_{n,\beta}(X^{(p)},Y^{(p)})$ and $R_{n,\beta}(X^{(p)},Y^{(p)})$ for dependent processes
$X,Y$. In this case $T_\beta(X,Y)$ is positive. We prove central limit theory
with Gaussian limits for 
\beao
\sqrt{n}\big( T_{n,\beta}(X^{(p)},Y^{(p)})-T_\beta(X,Y), R_{n,\beta}(X^{(p)},Y^{(p)})-R_\beta(X,Y) \big)\,.
\eeao
In particular, if one used the normalization $n$ for the independent case,
one would get $n\, T_{n,\beta}(X^{(p)},Y^{(p)})\stp\infty$ and 
$n\,R_{n,\beta}(X^{(p)},Y^{(p)})\stp\infty$.
This observation allows one to clearly distinguish between the independent
case and the alternative of dependent $X,Y$.
\par
The distinct \asy\  behavior of  $T_{n,\beta}(X^{(p)},Y^{(p)})$ and 
 $R_{n,\beta}(X^{(p)},Y^{(p)})$ in the independent and dependent
cases is explained by the $V$-statistic structure underlying the 
sample distance  covariance  $T_{n,\beta}(X^{(p)},Y^{(p)})$. Indeed,
this quantity is approximated by the non-degenerate $V$-statistic
$T_{n,\beta}(X,Y)$. In view of classical limit theory (see \cite{arcones:gine:1992}) non-degenerate $V$-statistics satisfy the \clt\ with normalization  $\sqrt{n}$.
\ere
\bre\label{rem:x}
The numbers $\la_i$ in parts (2) and (4) of the theorem are the eigenvalues of
certain integral operators.
This follows from limit theory for degenerate $V$-statistics; 
see \cite{serfling:1980}, \cite{lyons:2013}, \cite{arcones:gine:1992}.
Unfortunately, neither the $\lambda_i$ nor the \ds\ of the limit 
are available. \cite{arcones:gine:1992}
proved the consistency of a bootstrap version of degenerate
$U$- and $V$-statistics. These latter results apply to $T_{n,\beta}(X,Y)$ but not 
to  $T_{n,\beta}(X^{(p)},Y^{(p)})$. In Section~\ref{sec:bootstrap} we argue that the
bootstrap also works for a modification of the latter quantity.
\ere

\section{The condition 
$T_{\beta}(X,Y)=0$  and independence of $X$ and $Y$}\label{sec:gauss}\setcounter{equation}{0}
The results in the previous section tell us that
$T_{n,\beta}(X^{(p)},Y^{(p)})\stp T_\beta(X,Y)=0$ for independent
$X,Y$ under various conditions on  
$X,Y$ and the size of the mesh $\delta_n$ of the partition $(t_i)$. An important question is 
whether, conversely, $T_\beta(X,Y)=0$ also implies independence of
$X,Y$. In the case  $\beta\in (0,1]$ an affirmative answer to this
question follows from \cite{lyons:2013}, based on the fact that
the metric obtained by raising the separable Hilbert space distance to the power
$\beta\in (0,1]$ is of the strong negative type. In the sequel we
extend the converse statement to all $\beta\in (0,2)$. Our approach is
based on studying the conditional independence of certain stochastic
integrals. 
\par
Let $B_1$ and $B_2$ be independent Brownian motions on $[0,1]$,
independent of a pair $(X,Y)$ of stochastically continuous bounded stochastic
processes $[0,1]$. The stochastic integrals 
\beao
Z_1=\int_0^1 XdB_1\qquad\mbox{and}\qquad Z_2=\int_0^1 YdB_2
\eeao
are well defined (and are, given $(X,Y)$, independent normal random
variables). 

The next lemma demonstrates a connection between such stochastic integrals
and distance covariances. Let $\bbf_B$ denote the $\sigma$-field
generated by $B=(B_1,B_2)$. 
\ble\label{lem:good}
Let  $\beta\in (0,2)$ and assume that 
$\E[\|X\|_2^\beta]+\E [\|Y\|_2^\beta]<\infty$. Let $Y'$ be a copy of
$Y$ independent of everything else. Then 
\beam
c_0^2\,T_\beta(X,Y)&=&\nonumber \int_{\bbr^2} |st|^{-(1+\beta/2)}
\E \Big| \E \Big[
\ex^{is \int X(u)\,dB_1(u)}\ex^{it \int Y(u)\,dB_2(u)} \\
&&\hspace{0.6cm}- \ex^{is \int X(u)\,dB_1(u)} \ex^{it \int
 Y'(u)\,dB_2(u)}\mid {\mathcal F_B}
\Big]\Big|^2\,ds\,dt,  \label{eq:form}
\eeam
 where
\beao
c_0= \int_{\R}\frac{1-\ex^{-\frac{s^2}{2}}}{|s|^{1+\beta/2}}\,ds\,.
\eeao
\ele
\begin{proof} Consider an independent copy $(X',Y')$ of $(X,Y)$
and let $Y'',Y'''$ be independent copies of $Y$ which are independent of
everything else.
The expectation on the \rhs\ in \eqref{eq:form} can be written as
\beao
&&\E\Big[
\ex^{is\int (X-X')dB_1+it \int(Y-Y')dB_2} + \ex^{is \int
 (X-X')dB_1+it \int(Y''-Y''')dB_2} \\
&&\quad - \ex^{is \int (X-X')dB_1-it \int(Y-Y'')dB_2} - \ex^{-is\int
 (X-X')dB_1+it \int(Y-Y'')dB_2} 
\Big] \\
& = &\E\Big[
\ex^{-\frac{s^2}{2}\int (X(u)-X'(u))^2\,du -\frac{t^2}{2}\int
 (Y(u)-Y'(u))^2\,du}\\
&&\qquad + \ex^{-\frac{s^2}{2}\int (X(u)-X'(u))^2\,du -\frac{t^2}{2}\int
 (Y''(u)-Y'''(u))^2\,du} \\
&&\qquad -2 \ex^{-\frac{s^2}{2}\int (X(u)-X'(u))^2\,du -\frac{t^2}{2}\int
 (Y(u)-Y''(u))^2\,du}
\Big] \\
&=&\E \Big[
\big(1-\ex^{-\frac{s^2}{2}\int (X(u)-X'(u))^2\,du}\big)\big(1-\ex^{-\frac{t^2}{2}\int
 (Y(u)-Y'(u))^2\,du} \big) \\
& &\qquad + \big(1-\ex^{-\frac{s^2}{2}\int (X(u)-X'(u))^2du}\big)\big(1-\ex^{-\frac{t^2}{2}\int
 (Y''(u)-Y'''(u))^2\,du} \big) \\
&& \qquad -2 \big(1-\ex^{-\frac{s^2}{2}\int (X(u)-X'(u))^2\,du}\big)\big(1-\ex^{-\frac{t^2}{2}\int
 (Y(u)-Y''(u))^2\,du} \big)
\Big].
\eeao
By change of variables,
\begin{align*}
 \int_{\R} \frac{1-\ex^{-\frac{s^2}{2}\int (X(u)-X'(u))^2du}}{|s|^{ 1+\beta/2}}ds
 &=  c_0\,\|X-X'\|_2^\beta\,.
\end{align*}
Thus $T_\beta(X,Y)$ coincides with 
\beao
\E \big[ \|X-X'\|_2^\beta \|Y-Y'\|_2^\beta +\|X-X'\|_2^\beta\|Y''-Y'''\|_2^\beta
 -2\|X-X'\|_2^\beta\|Y-Y''\|_2^\beta \big]
\,.
\eeao
\end{proof}

An immediate corollary of  Lemma~\ref{lem:good} is that 
$T_\beta(X,Y)=0$ implies that, for a.e. $s,t$, 
$$
\E \Big[
\ex^{is \int X(u)\,dB_1(u)}\ex^{it \int Y(u)\,dB_2(u)}- \ex^{is \int
  X(u)\,dB_1(u)} \ex^{it \int  Y'(u)\,dB_2(u)}\mid {\mathcal F_B}\Bigr] =0
$$
with probability 1. By Fubini's theorem, on an event of probability 1, 
this equality holds for all rational  $s,t$, hence for all real
$s,t$. We conclude that the stochastic integrals
$Z_1,Z_2$ are conditionally independent given $\bbf_B$. 
\par
The next theorem,
which is the main result of this section, shows that this implies
independence of $X$ and $Y$. 

\begin{theorem} \label{t:indep}
If the stochastic integrals $Z_1$ and $Z_2$  
are a.s. conditionally independent given $\bbf_B$ then  $X,Y$ are
independent. In particular, if $\beta\in (0,2)$ and 
$\E[\|X\|_2^\beta]+\E [\|Y\|_2^\beta]<\infty$, then $T_\beta(X,Y)=0$
if and only if $X,Y$ are independent. 
\end{theorem} 
\begin{proof}
Only the fact that the conditional independence of the integrals
implies independence of $X$ and $Y$ remains to be proved. Let $\bigl(
a(t), \, 0\leq t\leq 1\bigr)$ and $\bigl( 
b(t), \, 0\leq t\leq 1\bigr)$ be functions in $L^2[0,1]$, and 
$$
A_1(t) = \int_0^t a(s)\, ds\quad\mbox{and} \quad A_2(t)=\int_0^t b(s)\, ds,\quad \ 0\leq t\leq
1\,.
$$
Since the law of the bivariate process 
$$
(\tilde B_1(t), \tilde B_2(t),\, 0\leq t\leq 1) = \bigl(
B_1(t)+A_1(t), \, B_2(t)+A_2(t)\bigr)\,,\qquad 
0\leq t\leq 1\,,
$$
is equivalent to the law of the standard bivariate Brownian motion, it follows
that the integrals
$$
\int_0^1 X(t)\, d\tilde B_1(t) = \int_0^1 X(t)\, d B_1(t)
+ \int_0^1 X(t) a(t)\, dt
$$
and
$$
\int_0^1 Y(t)\, d\tilde B_2(t) = \int_0^1 Y(t)\, d B_2(t)
+ \int_0^1 Y(t) b(t)\, dt
$$
are a.s. conditionally independent given $\bbf_B$. 

It is not difficult to construct 
a sequence $ (C_n)$ of events in $\bbf_B$,  
of positive probability,  such that the conditional laws of
the integrals 
$$
\int_0^1 X(t)\, d B_1(t) \ \text{and} \ \int_0^1 Y(t)\, d B_2(t)
$$ 
given $C_n$ converge to the degenerate law at zero 
as $n\to\infty$. One way for producing such a sequence of events is to let
the two independent Brownian motions take values close to zero at the
points $i/n, \, i=0,1,\ldots, n$. Letting $n\to\infty$ we conclude
that the integrals
$$
\int_0^1 X(t)\, a(t)\, dt \ \text{and} \ \int_0^1 Y(t) \,b(t)\, dt
$$
are independent. 
\par
For every fixed realization of the processes $X$ and $Y$,
\begin{equation} \label{e:good.set}
\lim_{\vep\to 0} \frac1\vep\int_{t}^{t+\vep} X(s)\, ds=X(t) \ \
\text{and} \ \ \lim_{\vep\to 0} \frac1\vep\int_{t}^{t+\vep} Y(t)\,
ds=Y(s)
\end{equation} 
for all $t$ in a set of full Lebesgue measure. By Fubini's theorem
there is a  set $M$ of full Lebesgue measure such that, for every $t\in
M$, \eqref{e:good.set} holds a.s. By necessity, the set $M$ is dense in
$[0,1]$. 
\par
To prove our claim it suffices to prove that for any points $0=t_0<
t_1<\cdots<t_k<t_{k+1}=1$, $k\ge 1$, 
the random vectors $(X(t_1),\ldots,X(t_k))$ and 
  $(Y(t_1),\ldots,Y(t_k))$ are independent. By stochastic continuity
  of the processes $X$ and $Y$ it is enough to restrict ourselves to
  the case when every $t_i\in M$. 
Let $0<\vep<\min_{i=1,\ldots,k} (t_{i+1}-t_{i})$. Choosing piece-wise
constant functions $\bigl( a(t), \, 0\leq t\leq 1\bigr)$ and $\bigl(
b(t), \, 0\leq t\leq 1\bigr)$, we conclude that the sums 
$$
\sum_{i=1}^k \theta_i\int_{t_i}^{t_i+\vep} X(t)\, dt \ \ \text{and} \
\ \sum_{i=1}^k \gamma_i\int_{t_i}^{t_i+\vep} Y(t)\, dt 
$$
are independent for any choice of $\theta_1,\ldots, \theta_k$ and
$\gamma_1,\ldots, \gamma_k$. Since all points $(t_i)$ are in the set
$M$, dividing by $\vep$ and letting $\vep\to
0$ we conclude that 
$$
\sum_{i=1}^k \theta_i X(t_i) \ \ \text{and} \
\ \sum_{i=1}^k \gamma_iY(t_i)
$$
are independent for any choice of $\theta_1,\ldots, \theta_k$
and $\gamma_1,\ldots,\gamma_k$. By the
Cram\'er-Wold device this implies that the vectors $(X(t_1),\ldots,X(t_k))$ and 
\\  $(Y(t_1),\ldots,Y(t_k))$ are independent. 
\end{proof}

\section{The bootstrap for the sample distance covariance}\label{sec:bootstrap}\setcounter{equation}{0}
We mentioned in Remark~\ref{rem:x} that the limit \ds\ of $n\,T_{n,\beta}(X,Y)$
is not available. Theorem~\ref{thm:1} states that the discretization
$n\,T_{n,\beta}(X^{(p)},Y^{(p)})$ has the same \asy\ properties as 
$n\,T_{n,\beta}(X,Y)$ under suitable conditions on the smoothness of 
the sample paths, moment conditions and the growth rate of $p=p_n\to\infty$. 
\par
In this section we advocate the use of the bootstrap for approximating 
the \ds\ of  $n\,T_{n,\beta}(X^{(p)},Y^{(p)})$. The bootstrap can be made to work
for the degenerate $V$-statistic $T_{n,\beta}(X,Y)$ as shown in
\cite{arcones:gine:1992}. 
In this case, the naive bootstrap does not
work and one has to modify the degenerate kernel. Since the $V$-statistic
$T_{n,\beta}(X^{(p)},Y^{(p)})$ is degenerate for every fixed $p$
we face the problem of approximating  the \ds\ of the latter
statistic by its bootstrap version. We will show that this approximation
works.
\par
We will make use of a modification of Lemma~2.2 in \cite{dehling:mikosch:1994}, which deals with $U$-statistics
with a kernel defined on the Euclidean space. We work with a separable
metric space $S$. For $m\ge 1$, let $h:S^m\mapsto \bbr$ be a symmetric
\fct.  
Let $(X^{(1)}_i,X^{(2)}_i)$, $i=1,2,\ldots,$ be an $S\times S$-valued 
iid \seq\ with marginal laws ${\mathcal L}(X^{(1)})=F$ and ${\mathcal L}(X^{(2)})=G$,
respectively. On the subset of probability measures on $S$, 
\beao
\Gamma_{2,h}=\big\{H: \E[h^2(Z_1,\ldots,Z_m)]<\infty\quad\mbox{for iid
  $(Z_i)$ with common law $H$}\big\}\,, 
\eeao
we define the semi-metric 
\beao
d_{2,h}(F,G)= \inf\big\{ \big(\E\big[\big(h(X_1^{(1)},\ldots,X_m^{(1)})-
h(X_1^{(2)},\ldots,X_m^{(2)})\big)^2]\big)^{1/2}\big\}\,,
\eeao  
where the infimum is taken over all random elements \\
$\bigl(
X_1^{(1)},\ldots,X_m^{(1)}, X_1^{(2)},\ldots,X_m^{(2)}\bigr)$ in $S^{2m}$ such that 
$(X_i^{(1)},X_i^{(2)})$, $i=1,\ldots,m$, are iid $S^2$-valued 
random elements, $X_i^{(1)}$ has law $F$ and $X_i^{(2)}$ has law $G$.
The fact that $d_{2,h}$ is a semi-metric can be shown using similar 
arguments as in the proof of Lemma~8.1 in \cite{bickel:freedman:1981}
that discusses the properties of the related 
Wasserstein metric $d_2$  on a subset of probability measures on
$\bbr$,  $\Gamma_2=\{H: \E_H[Z^2]<\infty\}$, 
defined by
\beao
d_2(F,G)= \inf\big\{ \big(\E\big[|A-B|^2\big]\big)^{1/2}: {\mathcal L}(A)=F\,,
 {\mathcal L}(B)=G\}\,.
\eeao
\par
Let $m\ge 2$ and choose $H\in\Gamma_{2,h}$. Define a function on $S\times S$ by 
\beam\label{eq:h2}
h_2(x,y;H)&=&\E[h(x,y,Z_3,\ldots,Z_m)]-\E [h(x,Z_2,\ldots,Z_m)]\nonumber\\
&& -\E [h(Z_1,y,Z_3,\ldots,Z_m)]+ \E[h(Z_1,\ldots,Z_m)]\,,
\eeam
where $(Z_i)$ are iid with common law $H$. 
The proof of the following result is completely analogous to that of Lemma~2.2 in
\cite{dehling:mikosch:1994}.
\ble \label{l:dehl.mik}
Let $F,G$ be in $\Gamma_{2,h}$, 
 $\bigl(
X_j^{(1)}\bigr)$  iid with common law $F$, and $\bigl(
X_j^{(2)}\bigr)$  iid with common law $G$. Then for any $n\geq 1$, 
\beam\label{eq:d2h}
&& d_2\Big({\mathcal L}\big( \dfrac 1n \sum_{1\le i\ne j\le n}
h_2(X_i^{(1)},X_j^{(1)};F)\big)\,, {\mathcal L}\big(
\dfrac 1n \sum_{1\le i\ne j\le n} h_2(X_i^{(2)},X_j^{(2)};G)\big)\Big) \nonumber\\
&&\hspace{1cm} \le 2^{5/2}\,d_{2,h}(F,G)\,.
\eeam
\ele
For an $S$-valued iid \seq\  $(Z_i)$ with common law $F\in \Gamma_{2,h}$
and $n\geq 1$ we denote by $F_n$  
the empirical law of $Z_1,\ldots,Z_n$. 
Consider an iid \seq\ $(Z_{ni}^\ast)$ 
with the law $F_n$, that is, given that law, independent of $(Z_i)$. 
The following result is analogous 
to Theorem~2.1 in \cite{dehling:mikosch:1994}.
\bco\label{cor:dehlmik}
Under the aforementioned conditions, and if also \\
$\E[|h(Z_{i_1},\ldots,Z_{i_m})|^2]<\infty$ for all indices $1\leq i_1 \leq \ldots \leq i_m\leq m$,
we have 
\beao
d_2\Big({\mathcal L}\big(\dfrac 1n \sum_{1\le i\ne j\le n} 
h_2(Z_{ni}^\ast,Z_{nj}^\ast;F_n)\big)\;, {\mathcal L}
\big(\dfrac 1n \sum_{1\le i\ne j\le n} 
h_2(Z_{i},Z_{j};F)\big)
\Big)\to 0\,,
\eeao
for almost all realizations of $(Z_i)$.
\eco
\begin{proof} By \eqref{eq:d2h}, it suffices to show that
 $d_{2,h}(F_n,F) \rightarrow 0$, almost surely.  By Varadarajan's
 theorem (see \cite{billingsley:1968}, p.29) the empirical distribution 
$F_n$ converges weakly to the distribution $F$, 
for almost all realizations $(z_i)_{i\geq 1}$ of $(Z_i)_{i\geq 1}$.  
Thus, by Skorokhod's theorem, there exist a sequence of  
random variables $(Z_n^\ast)_{n\geq 1}$ such that $Z_n^\ast$ 
has distribution $F_n$, and an $F$-distributed random variable 
$\tilde{Z}$ such that $Z_{n}^\ast \rightarrow \tilde{Z}$ almost surely.  
We now take $m$ iid copies of the pair $(Z_n^\ast, \tilde{Z})$, 
which we denote by
 $(Z_{n1}^\ast,\tilde{Z}_1),\ldots,(Z_{nm},\tilde{Z}_m)$. Then 
\[
 (Z_{n1}^\ast,\ldots, Z_{nm}^\ast) \rightarrow (\tilde{Z}_1,\ldots,\tilde{Z}_m), \mbox{ almost surely}. 
\]
Moreover, by definition of $d_{2,h}$, we have
\[
  d_{2,h}(F_n,F) \leq \left( \E\left[(h(Z_{n1}^\ast,\ldots,Z_{nm}^\ast) -h(\tilde{Z}_1,\ldots,\tilde{Z}_m) )^2\right] \right)^{1/2}.
\]
It suffices to show that the \rhs\ converges to $0$ as $n\rightarrow \infty$. For any $\epsilon>0$, we can find a bounded continuous function $g:S^m\rightarrow \R$ such that 
\[
  \E \left[ (
 h(\tilde{Z}_1,\ldots,\tilde{Z}_m)-g(\tilde{Z}_1,\ldots,\tilde{Z}_m)
 )^2\right]\leq \epsilon. 
\]
By Lebesgue's dominated convergence theorem, we obtain
\[
  \E \left[ ( g(Z_{n,1}^\ast,\ldots,Z_{n,m}^\ast)-g(\tilde{Z}_1,\ldots,\tilde{Z}_m)   )^2\right]\rightarrow 0.
\]
The strong law of large numbers for $U$-statistics implies that 
\begin{eqnarray*}
&&  \E \left[ (h(Z_{n,1}^\ast,\ldots,Z_{n,m}^\ast)
	-g(Z_{n,1}^\ast,\ldots,Z_{n,m}^\ast ))^2 \right] \\
&& \quad = \frac{1}{n^m} \sum_{1\leq i_1,\ldots,i_m \leq n} (h(z_{i_1},\ldots,z_{i_m}) -g(z_{i_1},\ldots,z_{i_m}))^2 \\
&& \quad \rightarrow \E (h(Z_1,\ldots,Z_m)-g(Z_1,\ldots, Z_m))^2 \leq \epsilon.
\end{eqnarray*}
This finishes the proof.
\end{proof} 

In what follows, $(Z_i)$ will stand for the iid \seq\ of the pairs 
$(X_i,Y_i)$, $i=1,2,\ldots,$ used in the previous sections for defining
the quantities $T_{n,\beta}(X,Y)$. Correspondingly, we write 
$(Z_i^{(p)})$ for the \seq\ of the discretizations $(X_i^{(p)},Y_i^{(p)})$, $i=1,2,\ldots,$ with generic element $Z^{(p)}$.
For the ease of presentation
we focus on the case $\beta=1$ and suppress $\beta$ in the notation.
We consider only the case when $X,Y$ have finite second moments.
A generic element $Z=(X,Y)$ has trajectory
$(x,y)$ assuming values in a \fct\ space 
$S$ where $x,y$ are defined on $[0,1]$ and  are Riemann square-integrable. 
\par
Under the hypothesis that $X,Y$ are 
independent, $T_{n}(X,Y)$ has \rep\  
as a $V$-statistic of order 4 with a 1-degenerate symmetric 
kernel $h_4=h(x_1,x_2,x_3,x_4)$; see Appendix \ref{asymptotic},
 where we also show  that, when scaled by $n$,  
the limits of $T_{n}(X,Y)$ and the corresponding
normalized $U$-statistic (which is obtained by ignoring all summands
$h(Z_{i_1},Z_{i_2},Z_{i_3},Z_{i_4})$ with the property $i_j=i_k$ for $j\ne k$)
differ by an additive constant.
 Applying the Hoeffding decomposition to 
this $U$-statistic, the limiting \ds\ of $nT_{n}(X,Y)$
coincides, up to a scale change,  with the limiting
  distribution of the following normalized  $U$-statistic:
\beao
U_n(Z)= \dfrac 1 {n}\sum_{1\le i\ne j\le n} h_2(Z_i,Z_j;F_Z)
\eeao
where $F_Z=F_X\times F_Y$ and $h_2$ is defined in \eqref{eq:h2}.  
\cite{arcones:gine:1992} proved that the correct
bootstrap version of $n\,T_{n}(X,Y)$ is  
\beao
U_n(Z^\ast)= \dfrac 1 {n}\sum_{1\le i\ne j\le n} h_2(Z_{ni}^\ast,Z_{nj}^\ast;F_{n,Z})\,,
\eeao
where $F_{n,Z}$ is the empirical \ds\ of the iid sample
$Z_1,\ldots,Z_n$. The fact that the limiting distributions of $U_n(Z)$
and $U_n(Z^\ast)$ coincide follows from  Corollary~\ref{cor:dehlmik}.
\par
Our program for the remainder of this section is to show
that we are allowed to replace $Z=(X,Y)$ by the corresponding 
discretizations $Z^{(p)}=(X^{(p)},Y^{(p)})$ in the aforementioned 
$U$- and $V$-statistics, i.e., we will show that suitable bootstrap
versions of $n\,T_{n,\beta}(X,Y)$ and $n\,T_{n,\beta}(X^{(p)
},Y^{(p)})$ have the same  limiting distribution.
We start by showing that $U_n(Z)$ and $U_n(Z^{(p)})$ are close in 
the sense of the $d_2$-metric. 
\ble\label{eq:additonal}
Assume the following conditions:
\begin{enumerate}
\item[\rm 1.] $X,Y$ are independent and have finite second moments. 
\item[\rm 2.] Condition {\rm (A1)} holds.
\item[\rm 3.] $\delta_n\to 0$ as $\nto$.
\end{enumerate}
Then
$d_2\big({\mathcal L}(U_n(Z));{\mathcal L}(U_n(Z^{(p)}))\big)\le c\, \delta_n^{(\gamma_X\wedge\gamma_Y)/2}
\to 0\,.$
\ele
\begin{proof} 
By \eqref{eq:d2h}, with $h$ given by \eqref{e:h4}, we
  have 
\beao
&&  d_{2}\big({\mathcal L}(U_n(Z));{\mathcal L}(U_n(Z^{(p)})\big) \\
&&\quad \le 
c\,\bigl\{
\E\big[\big(h(Z_1,\ldots,Z_4)-h(Z_1^{(p)},\ldots,Z_4^{(p)})\big)^2\big]\bigr\}^{1/2}
\\
&& \quad \le c\,\bigl\{
\E\big[\big(f(Z_1,\ldots,Z_4)-f(Z_1^{(p)},\ldots,Z_4^{(p)})\big)^2\big]\bigr\}^{1/2}
\\
&& \quad \le c\, \bigl( \E I_1^2+\E I_2^2+\E I_3^2\bigr)^{1/2}\,,  
\eeao
where  
\beao
I_1&=&\|X_1-X_2\|_2\|Y_1-Y_2\|_2 -\|X_1^{(p)}-X_2^{(p)}\|_2
\|Y_1^{(p)}-Y_2^{(p)}\|_2 \,,\\
I_2&=&\|X_1-X_2\|_2 \|Y_3-Y_4\|_2 
-\|X_1^{(p)}-X_2^{(p)}\|_2 \|Y_3^{(p)}-Y_4^{(p)}\|_2 \,,\\ 
I_3&=&\|X_1-X_2\|_2 \|Y_1-Y_3\|_2 -
\|X_1^{(p)}-X_2^{(p)}\|_2 \|Y_1^{(p)}-Y_3^{(p)}\|_2  \,.\\
\eeao
The second moments are estimated as in Proposition~\ref{prop:1}
below. We have by \eqref{eq:star}, 
\beao 
&&\E\big[\big(\|X_1-X_2\|_2 -\|X_1^{(p)}-X_2^{(p)}\|_2 \big)^2\,
\|Y_1-Y_2\|_2^{2}\big] \leq c\,\delta_n^{\gamma_X} 
\eeao
and 
\beao
&&\E\big[\|X_1^{(p)}-X_2^{(p)}\|_2^2\,\big(\|Y_1-Y_2\|_2-\|Y_1^{(p)}-Y_2^{(p)}\|_2
\big)^2\big]\ \le c\, \delta_n^{\gamma_Y}\,.
\eeao
That is, $\E[ I_1^2] \leq c\, \delta_n^{\gamma_X\wedge\gamma_Y}$. 
The second moments of $I_2,I_3$ can be bounded by the same
quantities. 
\end{proof}


Our next goal is to show that, under appropriate assumptions,
the difference between the laws of $U_n(Z^\ast)$ and $U_n(Z^{(p)\ast})$ 
asymptotically vanishes. 

\ble\label{lem:boostr} 
Consider  the following conditions:
\begin{enumerate}
\item[\rm 1.] $X,Y$ are independent and have finite second moments. 
\item[\rm 2a.] Condition {\rm (A1)} holds. 
\item[\rm 2b.] $\E[|X(t)-X(s)|^4]\le c\,|t-s|^{\wt \gamma_X}$ and 
 $\E[|Y(t)-Y(s)|^4]\le c\,|t-s|^{\wt \gamma_Y}$ hold.
\item[\rm 3a.] $\sum_{n=1}^\infty\delta_n^{\gamma_X\wedge \gamma_Y}< \infty$.
\item[\rm 3b.]  $\sum_{n=1}^\infty \big( \delta_n^{2( \gamma_X\wedge  \gamma_Y)}+n^{-1}\delta_n^
{\wt\gamma_X\wedge \wt \gamma_Y}\big) < \infty$.
\end{enumerate}
If either {\rm 1, 2a, 3a} or {\rm 1, 2a, 2b, 3b} hold then 
$d_2\big({\mathcal L}(U_n(Z^{\ast})),{\mathcal L}(U_n(Z^{(p)\ast}))\big)  \to 0\,,$
for a.e. realization of $(Z_i)$.

\ele
\begin{proof} 
With $h$ given by \eqref{e:h4}, 
by Lemma \ref{l:dehl.mik} it is enough to prove that 
$d_{2,h}\bigl( {\mathcal L}(Z^{\ast}),{\mathcal L}(Z^{(p)\ast})\big)
\to 0$ for a.e. realization of $(Z_i)$. We have 
\beao
\lefteqn{d_{2,h}(n):=d_{2,h}\bigl( {\mathcal L}(Z^{\ast}),{\mathcal L}(Z^{(p)\ast})\big)}\\
&\le & \left( \E_{F_n}\big[\big( h(Z_1^{\ast},Z_2^{\ast},Z_3^{\ast},Z_4^{\ast})- h(Z_1^{(p)\ast},Z_2^{(p)\ast},Z_3^{(p)\ast},Z_4^{(p)\ast})\big)^2
\big]\right)^{1/2}\\
&=&\dfrac 1 {n^2} \left( \sum_{1\le i_1,  i_2, i_3,  i_4\le n} \big(h(Z_{i_1},Z_{i_2},Z_{i_3},Z_{i_4})-h(Z_{i_1}^{(p)},
Z_{i_2}^{(p)},Z_{i_3}^{(p)},Z_{i_4}^{(p)})\big)^2\right)^{1/2}\\ 
&\le &\dfrac 1 {n^2} \left( \sum_{1\le i_1,  i_2, i_3,  i_4\le n} \big(f(Z_{i_1},Z_{i_2},Z_{i_3},Z_{i_4})-f(Z_{i_1}^{(p)},
Z_{i_2}^{(p)},Z_{i_3}^{(p)},Z_{i_4}^{(p)})\big)^2\right)^{1/2}.
\eeao
We first show that the \rhs\ converges to zero under the assumption that 1, 2a, and 3a hold.
Using (A1), we obtain 
\beao
&& \E\big[d_{2,h}\bigl( {\mathcal L}(Z^{\ast}),{\mathcal
    L}(Z^{(p)\ast})\big)\big]^2 \\ 
&& \le \sum_{1\leq j_1,j_2,j_3,j_4\leq 4} 
\E\big[\big(f(Z_{j_1},\ldots,Z_{j_4})-f(Z_{j_1}^{(p)},\ldots,Z_{j_4}^{(p)})\big)^2\big]
\le c\,\delta_n^{\gamma_X\wedge\gamma_Y}\,.
\eeao
Thus, if $\sum_{n} \delta_n^{\gamma_X\wedge\gamma_Y}<\infty$ applications of Markov's inequality and 
the Borel-Cantelli lemma yield  that $d_{2,h}\bigl( {\mathcal
  L}(Z^{\ast}),{\mathcal L}(Z^{(p)\ast})\big) \to 0$ a.s. 
 as $\nto$. 
\par
Now assume that 1, 2a, 2b and 3b hold.
Using standard calculations for $U$-statistics,
we have 
\begin{align*}
\var(d_{2,h}^2(n))&\le c
\sum_{1\leq j_1,j_2,j_3,j_4\leq 4}
\Big[ n^{-1}
\var\Big(\big(h(Z_{j_1},\ldots,Z_{j_4})-h(Z_{j_1}^{(p)},\ldots,Z_{j_4}^{(p)})\big)^2\Big)\\
&\quad + \Big(\E\big[\big(h(Z_{j_1},\ldots,Z_{j_4})-h(Z_{j_1}^{(p)},\ldots,Z_{j_4}^{(p)})\big)^2
\big]\Big)^2\Big]=J_1+J_2\,.
\end{align*}
We have $J_2=O(\delta_n^{2 (\gamma_X\wedge\gamma_Y)})$. We can handle $J_1$ similarly
to the proof of Lemma~\ref{eq:additonal}. For example, 
\beao
\E\big[\|X_1-X_1^{(p)}\|_2^4]&=& \E\Big[\Big(\int_0^1 (X(u)-X^{(p)}(u))^2\, du\Big)^2\Big]\\
&\le & c\, \int_0^1 
\E\big[(X(u)-X^{(p)}(u))^4\big]\, du \le  c\, \delta_n^{\wt \gamma_X}\,.
\eeao
Now  $d_{2,h}(n)\stas 0$ as $\nto$ follows by an application of Markov's
inequality of order 2, the Borel-Cantelli lemma and since 
$\sum_n \big(n^{-1} \delta_n^{\wt \gamma_X\wedge\wt\gamma_Y}+\delta_n^{2(\gamma_X\wedge\gamma_Y)} \big)<\infty$. We omit further details.
\end{proof} 
Combining the previous arguments, a natural bootstrap version of the degenerate
$V$-statistic $n\,T_n(X^{(p)},Y^{(p)})$ is given by $U_n(Z^{(p)\ast})$.
\bpr
Assume the conditions of Lemma~\ref{lem:boostr}. Then
\beao
d_2\big({\mathcal L}(U_n(Z)),{\mathcal L}(U_n(Z^{(p)\ast}))\big)\to 0
\eeao
for a.e. realization of $(Z_i)$.
\epr
For an application of the bootstrapped sample distance correlation $n R_n(X^{(p)},Y^{(p)})$
we still miss one step in the derivation of the bootstrap 
consistency: we also need to
prove that the denominator quantities converge a.s.
\beao
T_n(X^{(p)},X^{(p)})\stas T(X,X)\quad\mbox{and}\quad T_n(Y^{(p)},Y^{(p)})\stas 
T(Y,Y)\,,\qquad \nto\,.
\eeao
In Lemma \ref{lem:consist} Appendix \ref{appendix:b} we provide sufficient conditions for this to hold.

\section{Simulations}\label{sec:simulation}\setcounter{equation}{0}
In this section we illustrate the theoretical results in a 
small simulation study. Throughout we choose $\beta=1$ and suppress the 
dependence on $\beta$ in the notation.
\par 
We start with
identically distributed fractional Brownian motions
(fBM) $X,Y$ on $[0,1]$ with Hurst coefficient $H$ and correlation $\rho$
where the dependence between $X$ and $Y$ is 
given by the covariance \fct \ 
\beao
 \cov(X(s),Y(t)) = \frac{\rho}{2}\{|s|^{2H} + |t|^{2H}
 -|t-s|^{2H}\},\ s,t\in[0,1]\,.
\eeao
If $X=Y$ we also set $\rho=1$. Note that, for $H=1/2$, the \rhs\
collapses into $\rho (s\wedge t)$, corresponding to \BM s $X,Y$. 
The top graph in Figure~\ref{gaussianprocess} 
nicely illustrates the consistency of the sample correlation
$R_n(X^{(p)},Y^{(p)})$  for independent $X$ and $Y$ $(\rho=0)$.
In the top row we fix $p=100$ and increase $n$ from $100$ to $400$, 
and we choose $H=1/4$, $H=1/2$ (BM) and $H=3/4$. 
Apparently, we can see the influence of the smoothness of the
sample paths: the larger $H$ the larger $\gamma_X=\gamma_Y=2H$
(see Example~\ref{exam:hurst}), the smoother the sample paths and 
the closer $R_n(X^{(p)},Y^{(p)})$ to zero; see also the upper bounds
in Proposition~\ref{prop:1}.
In the bottom row we show $R_n(X^{(p)},Y^{(p)})$ for dependent 
$X$ and $Y$ with $\rho=0.5$. We again choose $H=1/4$,
$H=1/2$ (BM) and $H=3/4$, fix $p=100$ and increase $n$ from $100$
to $300$. In the bottom graphs the sample distance correlation converges 
to some positive constants; we see a clear difference between the 
independent and dependent cases.
\begin{figure}[htbp]
\begin{center}
\includegraphics[width=\textwidth]{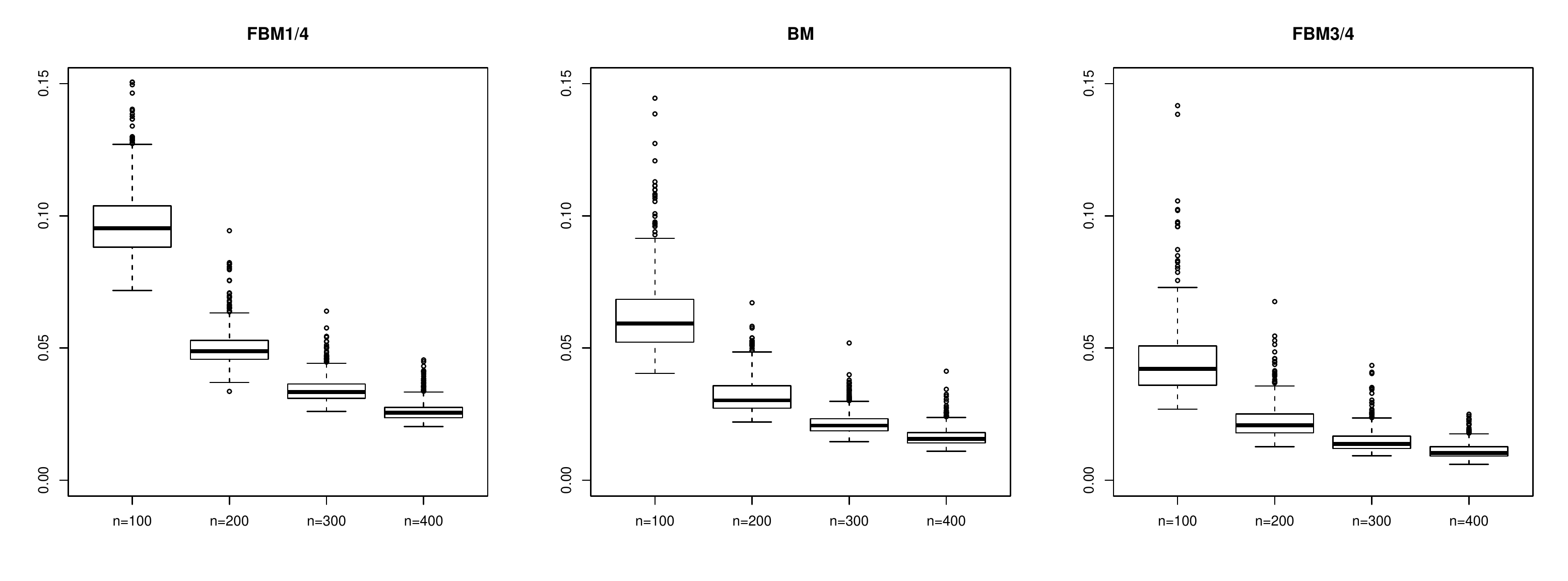}
\includegraphics[width=\textwidth]{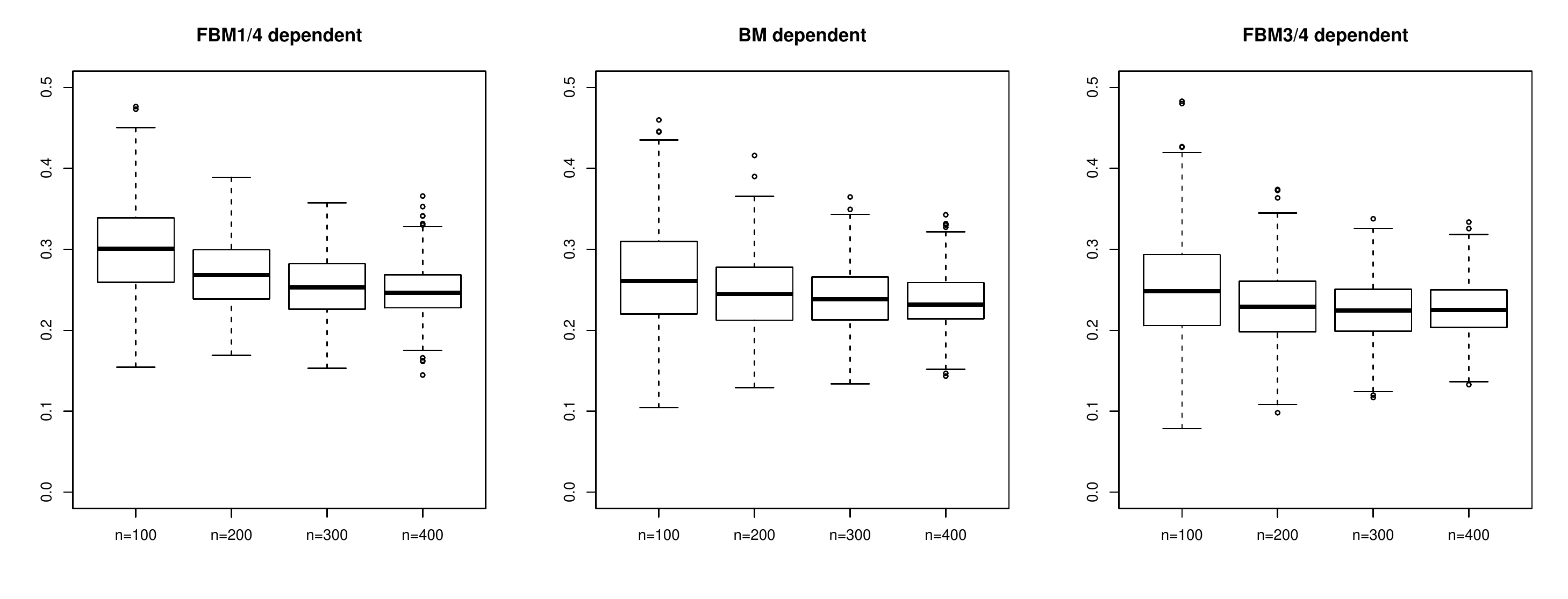}
\end{center}
\caption{Boxplots for $R_n(X^{(p)},Y^{(p)})$ simulated fBMs $X,Y$
with  $H=1/4, 1/2,3/4$ (from left to right), $p=100$ and increasing sample sizes $n$.
 Top: iid fBMs $X,Y$. Each boxplot is based on $500$ replications. 
 Bottom: identically distributed fBMs $X,Y$ with correlation $\rho=0.5$. 
Each boxplot is based on $300$ replications. }
\label{gaussianprocess}
\end{figure}
\begin{figure}[htbp]
	\begin{tabular}{ccc}
	\begin{minipage}{.33\textwidth}
		\includegraphics[width=0.96\textwidth,height=1.1\textwidth]{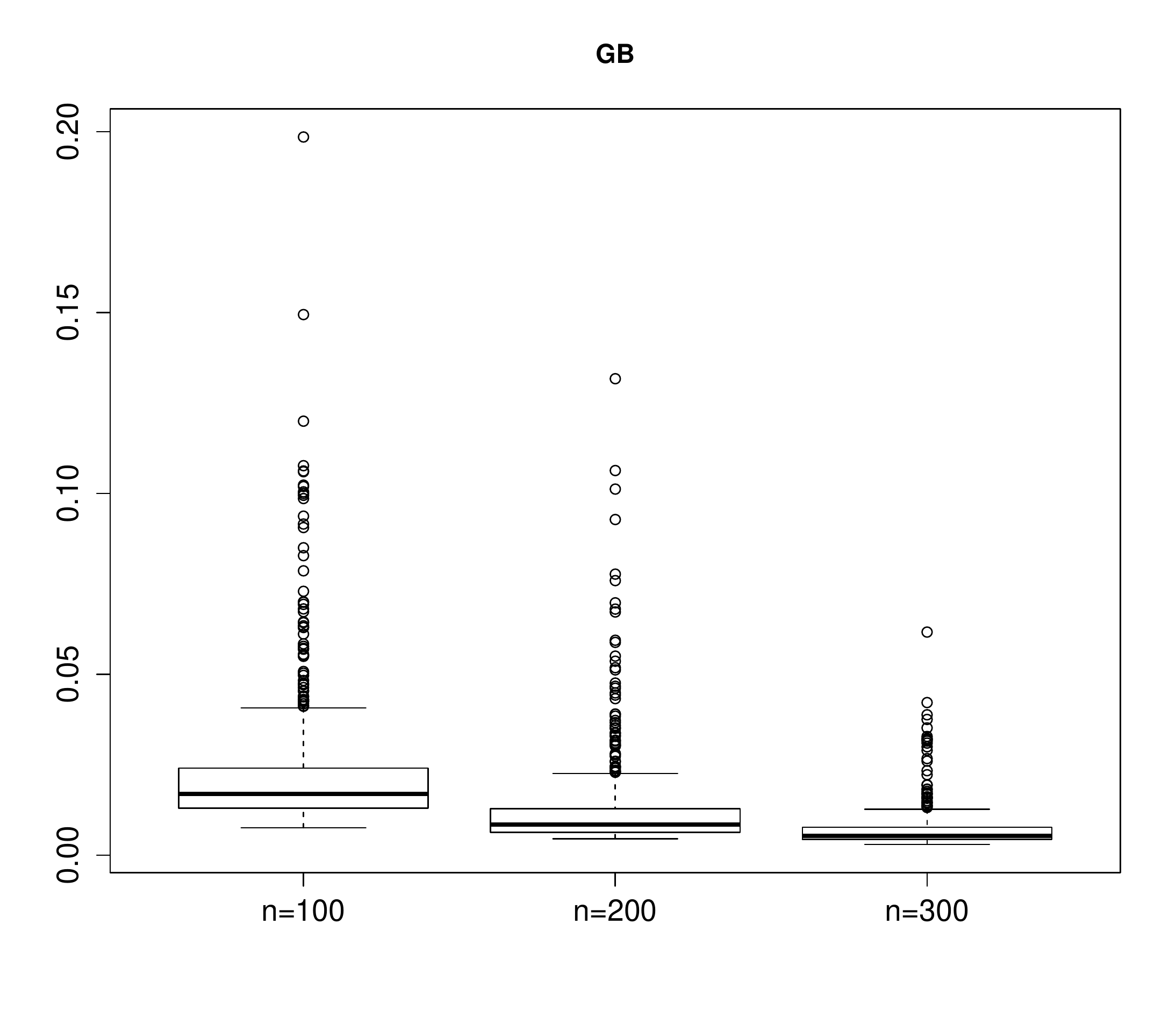}
	\end{minipage}
		\begin{minipage}{.33\textwidth}
		\includegraphics[width=0.96\textwidth,height=1.1\textwidth]{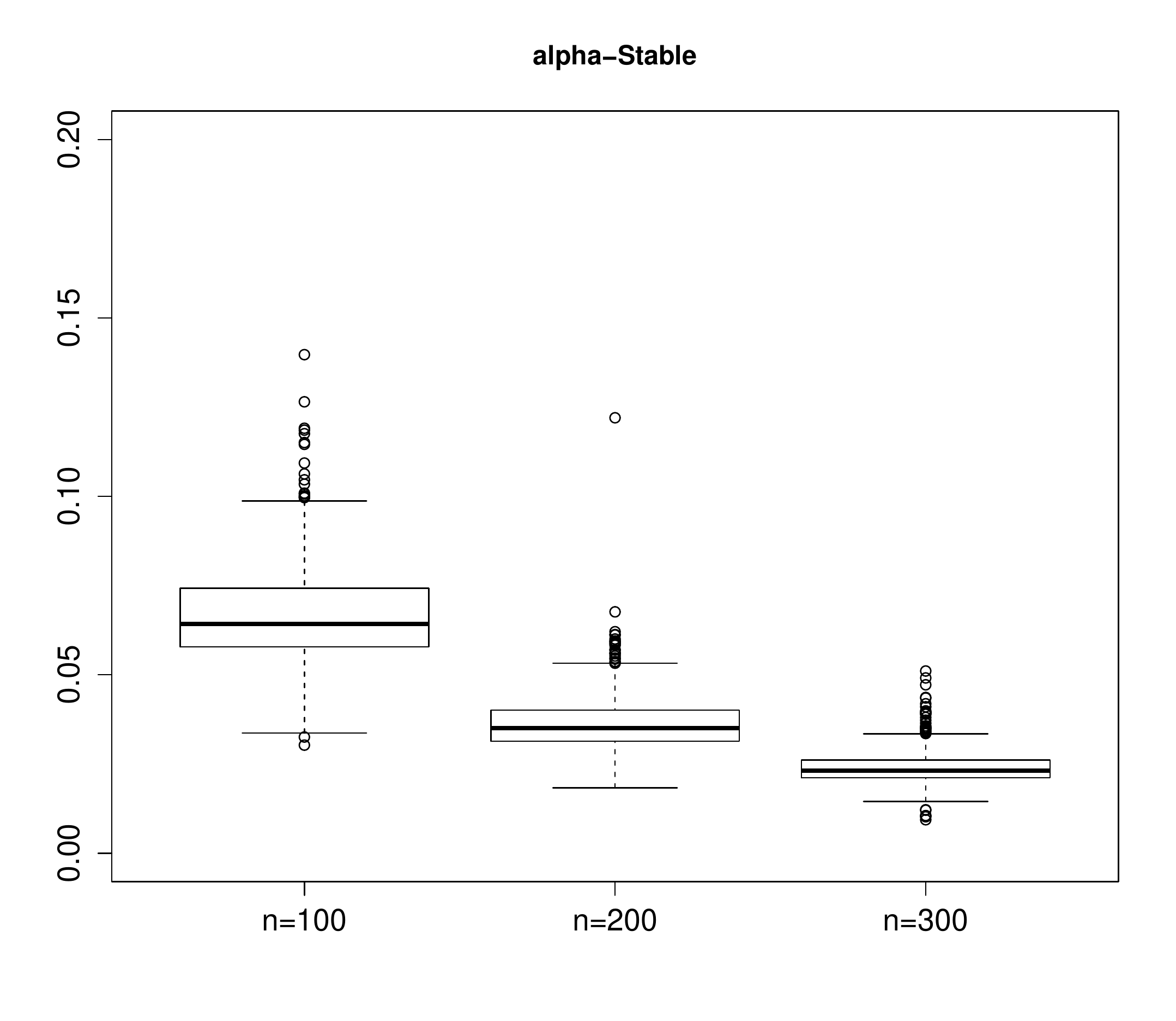}
	\end{minipage}
		\begin{minipage}{.33\textwidth}
		\includegraphics[width=0.96\textwidth,height=1.1\textwidth]{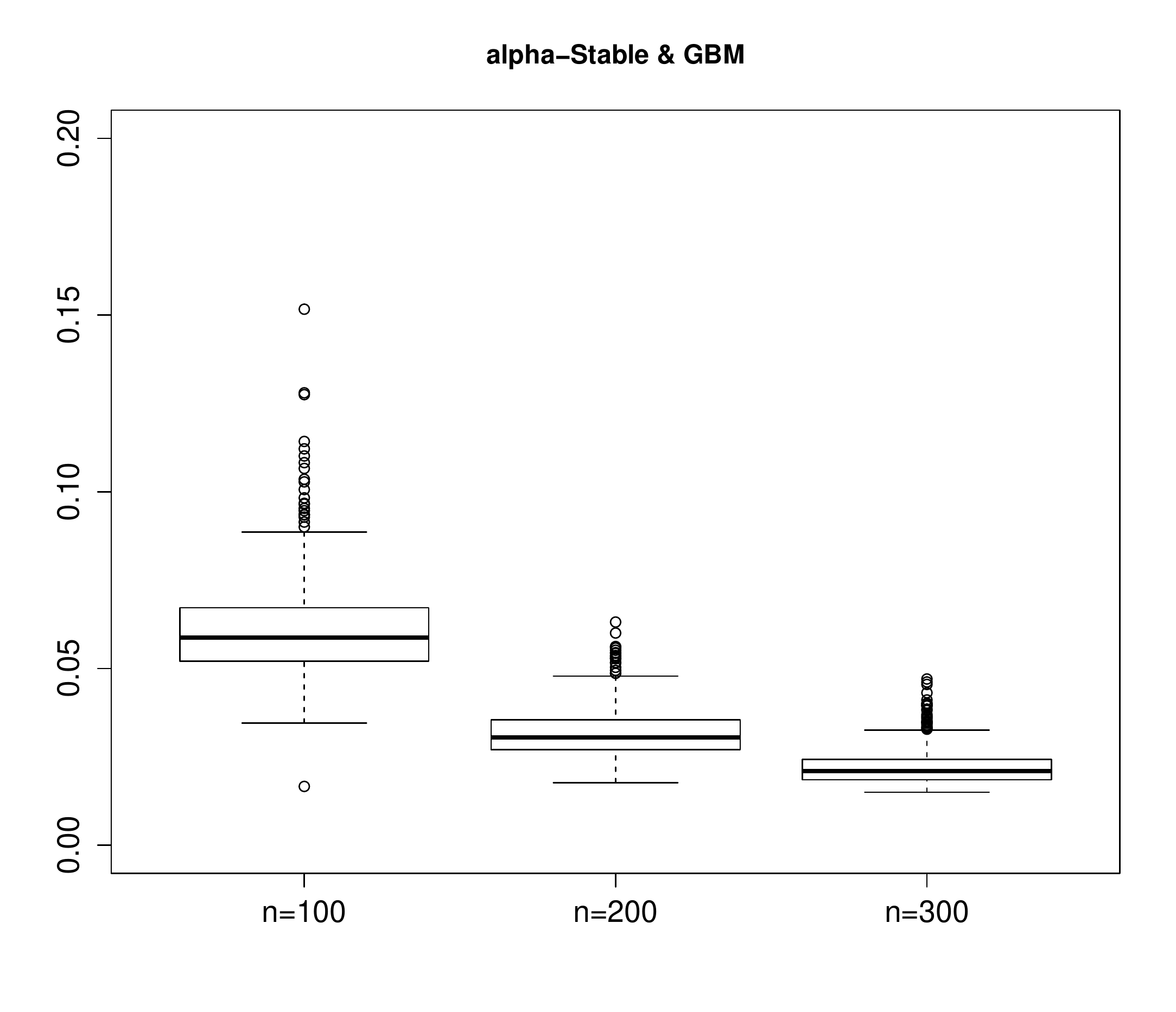}
	\end{minipage}
	\end{tabular}
\caption{Boxplots for $R_n(X^{(p)},Y^{(p)})$ for simulated
independent non-Gaussian processes $X,Y$, $p=100$ and increasing sample
size $n$. Each boxplot is based on $500$ replications.
Left: iid geometric BMs $X,Y$. 
Middle: iid $\alpha$-stable L\'evy motions $X,Y$. 
Right: independent geometric BM $X$ and 
 $\alpha$-stable L\'evy motion $Y$.}
\label{nongaussianprocess}
\end{figure}
\begin{figure}[htbp]
	\begin{tabular}{ccc}
	\begin{minipage}{.33\textwidth}
		\includegraphics[width=1.05\textwidth,height=1.2\textwidth]{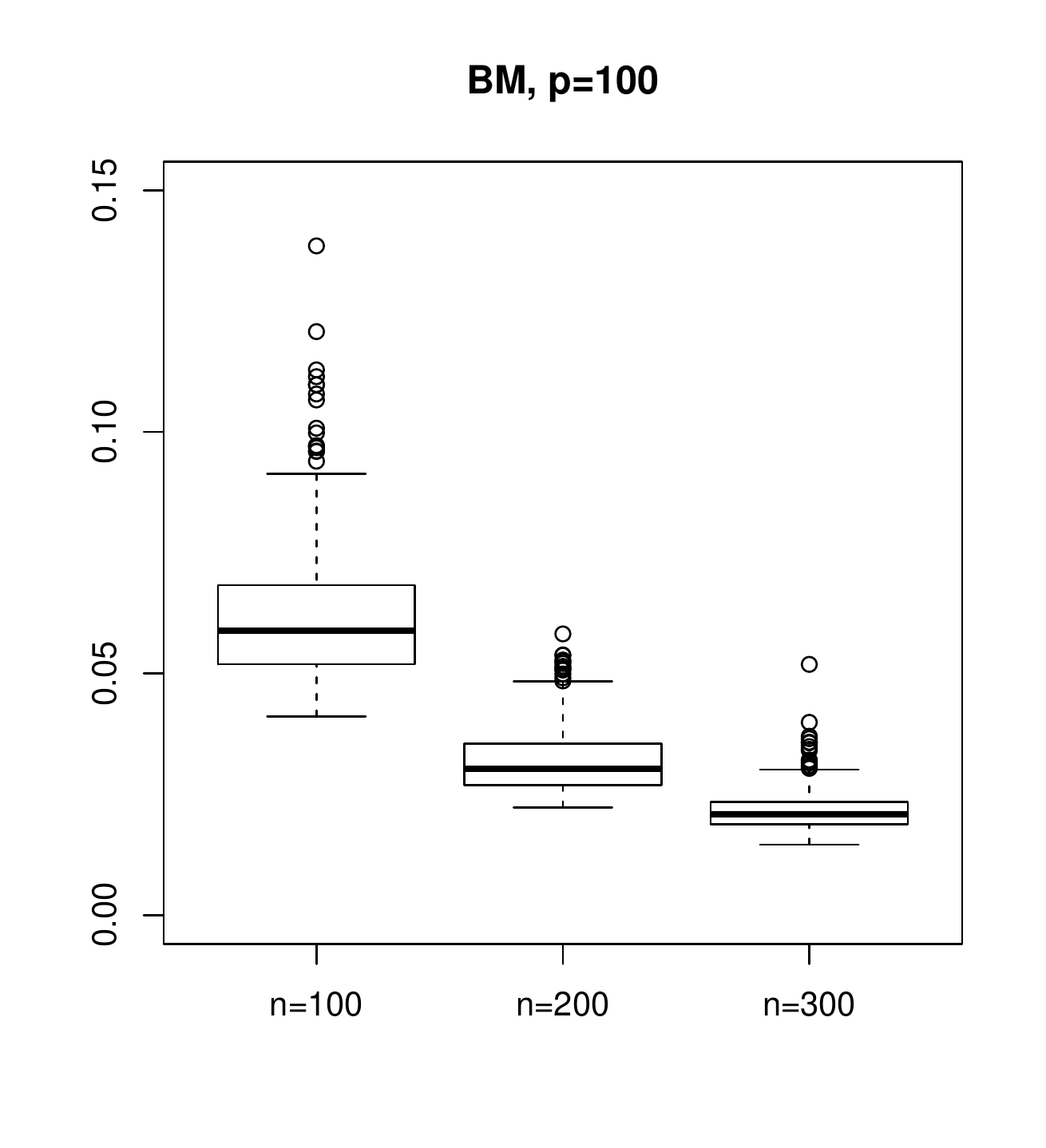}
	\end{minipage}
		\begin{minipage}{.33\textwidth}
		\includegraphics[width=0.95\textwidth,height=1.1\textwidth]
{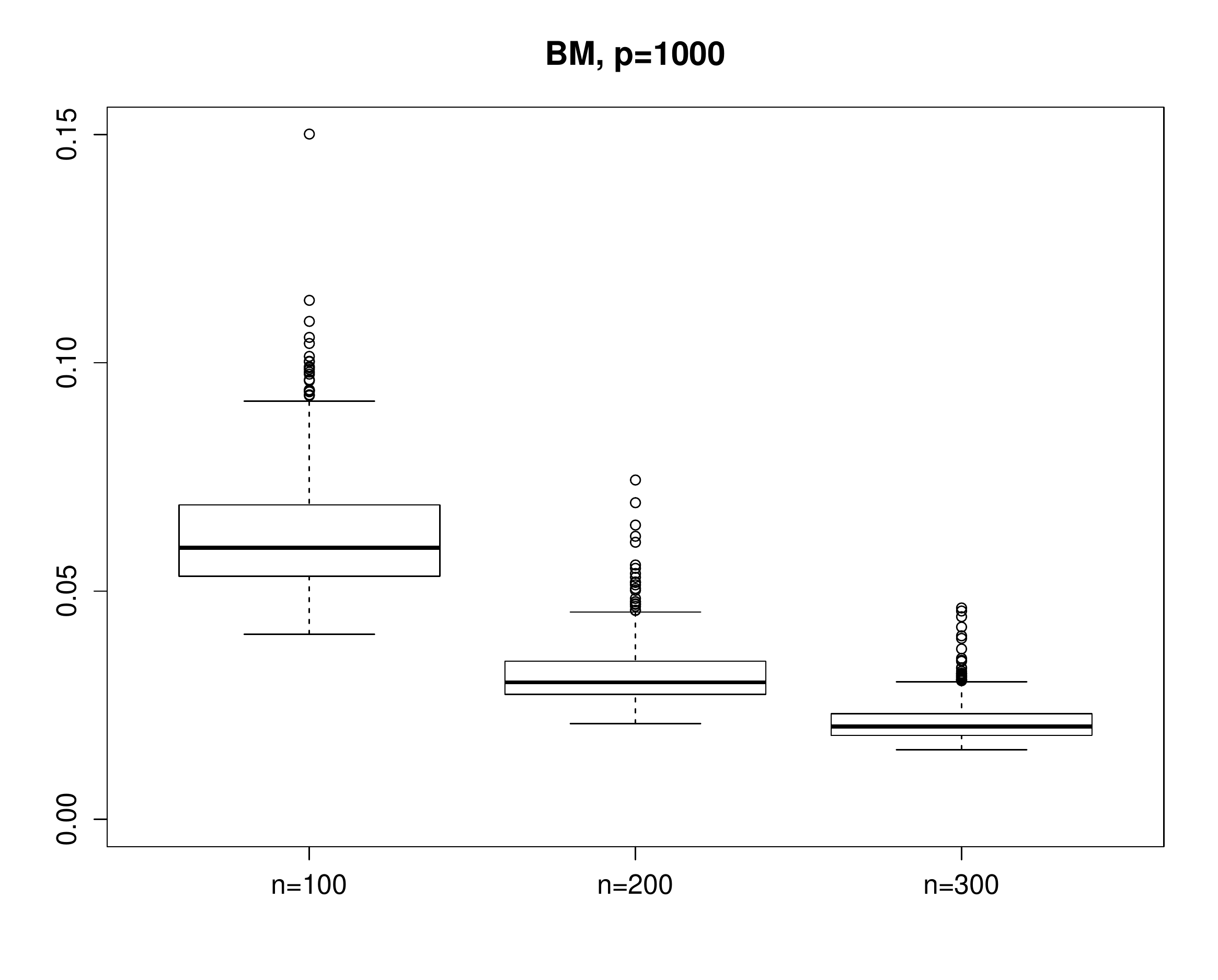}
	\end{minipage}
		\begin{minipage}{.33\textwidth}
		\includegraphics[width=0.96\textwidth,height=1.1\textwidth]{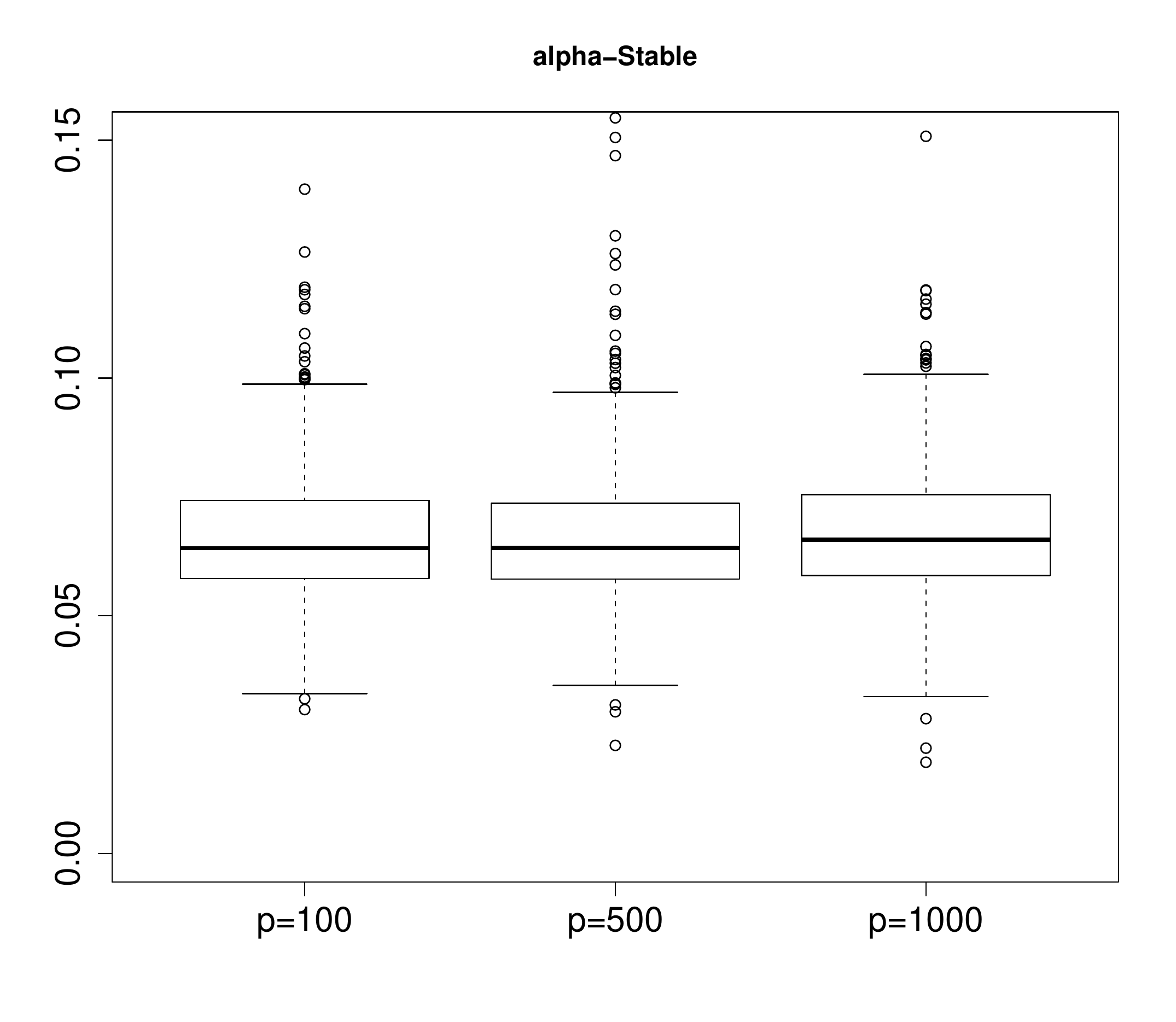}
	\end{minipage}
	\end{tabular}
\caption{Boxplots for $R_n(X^{(p)},Y^{(p)})$ for different $p$.
Left and middle: $X$, $Y$ are iid BMs. For each $p=100$ (left) and $p=1000$
 (middle) we take three distinct sample sizes $n=100,200,300$. 
The boxplots are
based on $300$ replications. Right: 
$X$, $Y$ are iid $\alpha$-stable L\'evy motions, $n=100$ is fixed 
while $p=100,500,1000$.  
The boxplots are based on $500$ replications.
}
\label{boxplot:gridnumber}
\end{figure}
\begin{figure}[htbp]
	\begin{tabular}{ccc}
	\begin{minipage}{.33\textwidth}
		\includegraphics[width=\textwidth,height=1.2\textwidth]{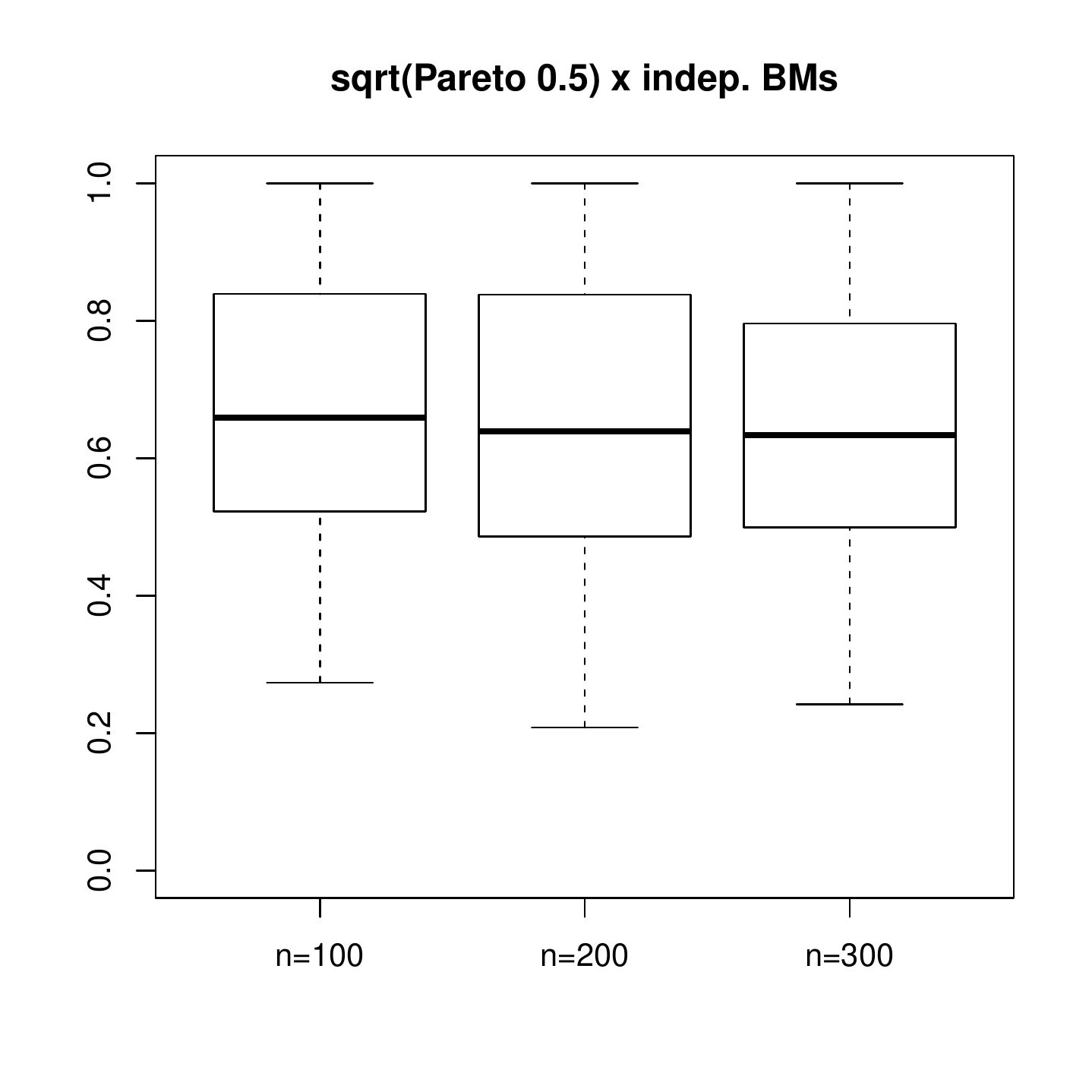}
	\end{minipage}
		\begin{minipage}{.33\textwidth}
		\includegraphics[width=\textwidth,height=1.2\textwidth]{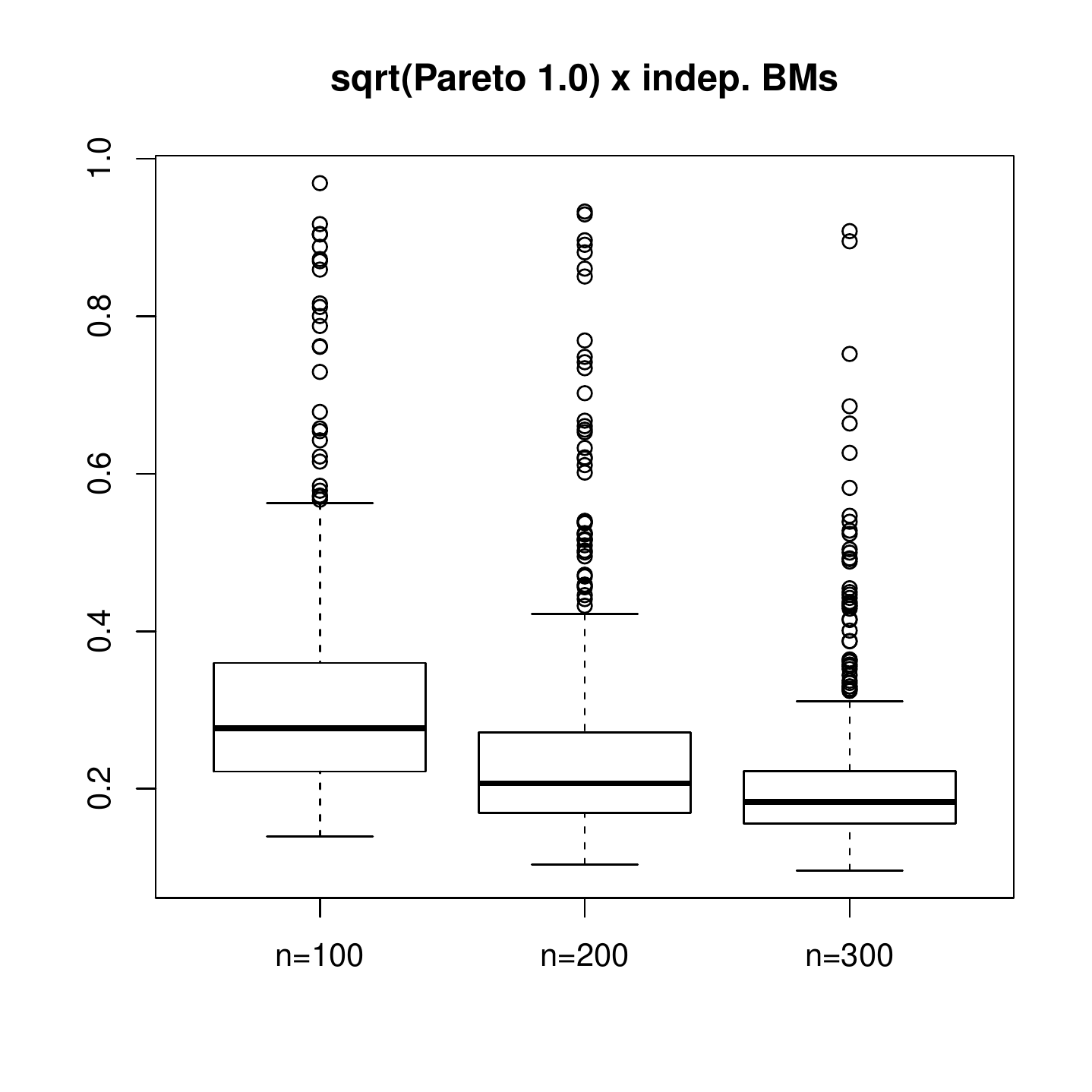}
	\end{minipage}
		\begin{minipage}{.33\textwidth}
		\includegraphics[width=\textwidth,height=1.2\textwidth]{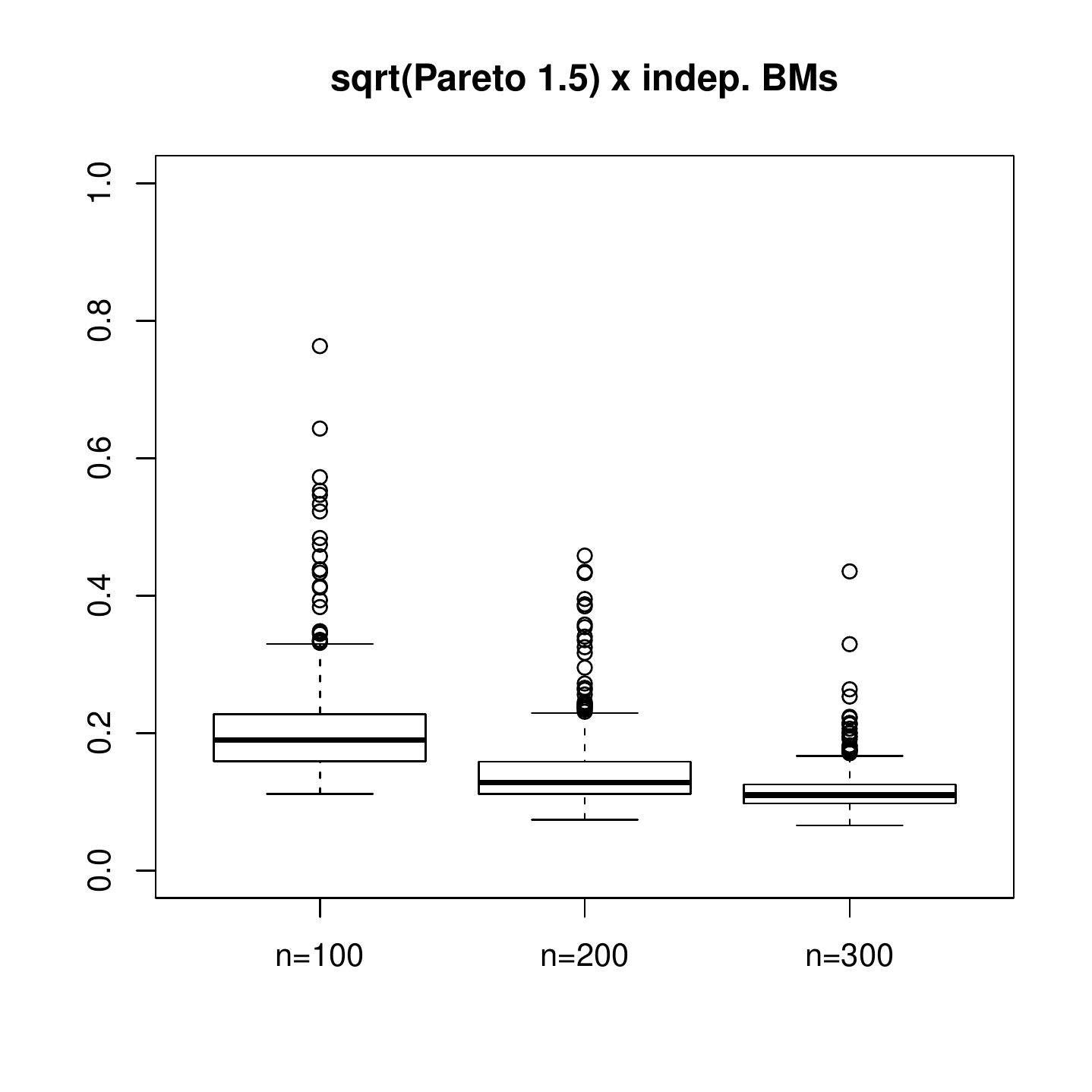}
	\end{minipage}
	\end{tabular}
	\begin{tabular}{rcc}
	\begin{minipage}{.33\textwidth}
		\includegraphics[width=\textwidth,height=1.2\textwidth]{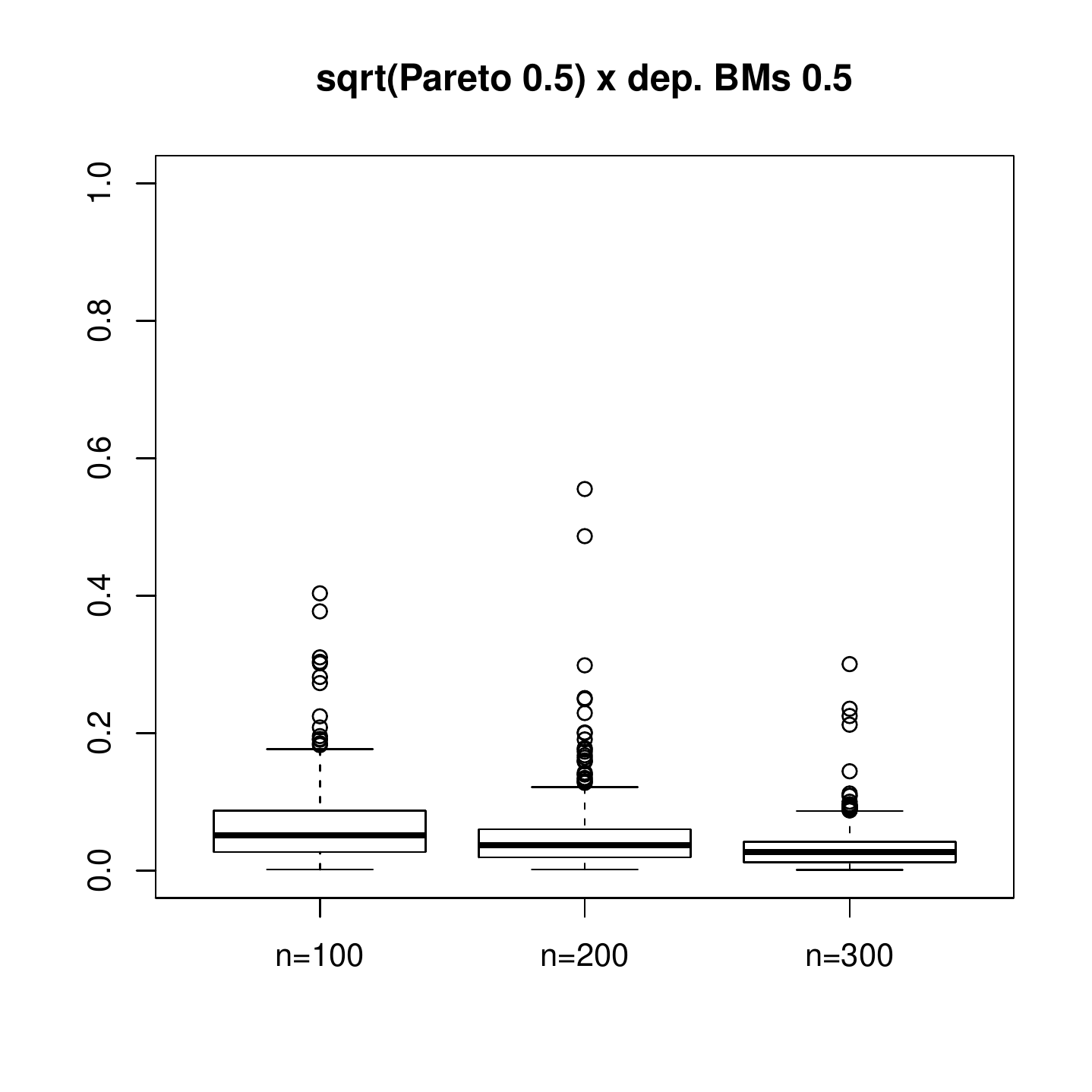}
	\end{minipage}
		\begin{minipage}{.33\textwidth}
	\includegraphics[width=\textwidth,height=1.2\textwidth]{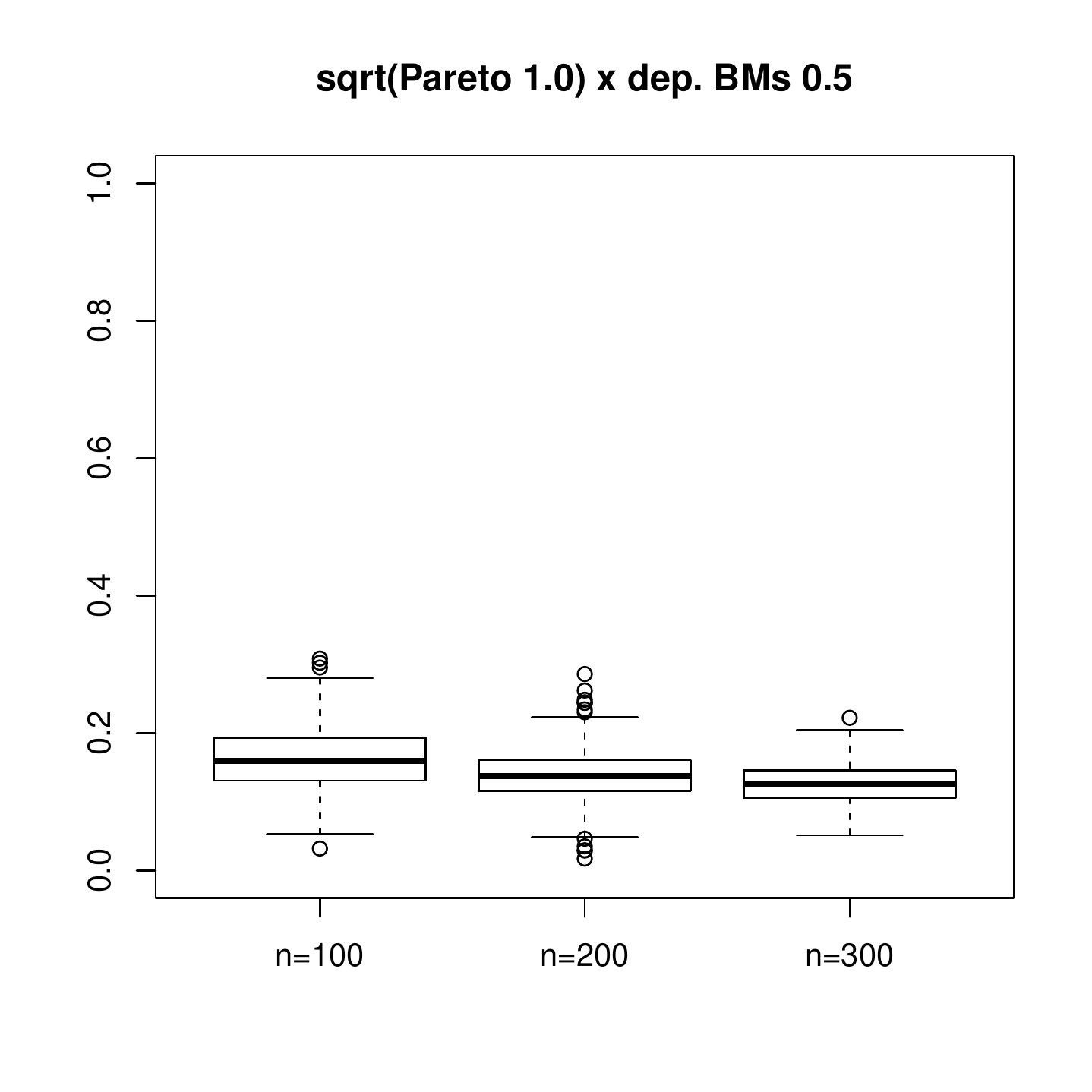}
	\end{minipage}
		\begin{minipage}{.33\textwidth}
	\includegraphics[width=0.95\textwidth,height=1.1\textwidth]{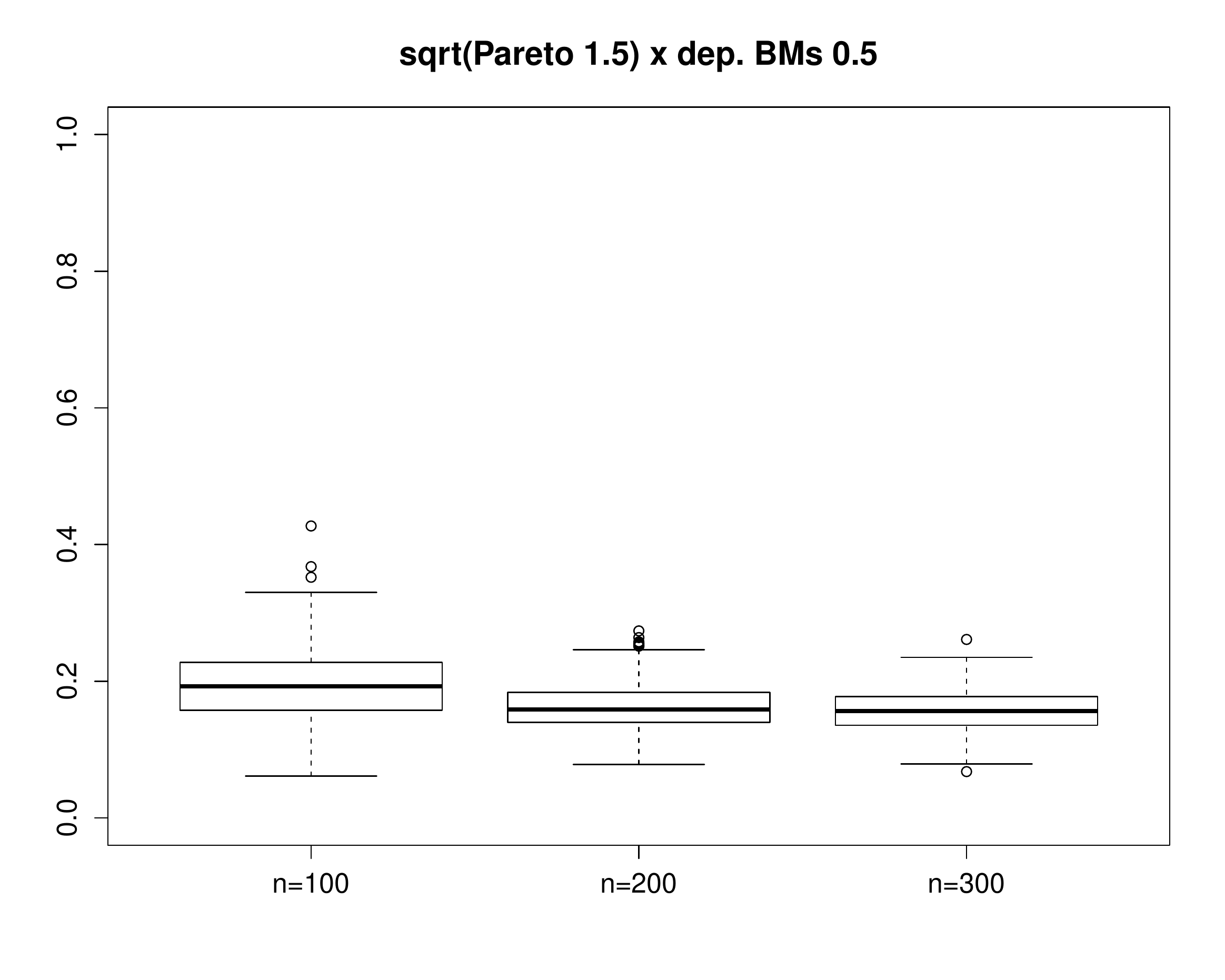}
	\end{minipage}
	\end{tabular}
\caption{Boxplots of $R_n(X^{(p)},Y^{(p)})$ for dependent heavy-tail
 cases. Top: $(X,Y)=A^{1/2}(B_1,B_2)$ for a Pareto$(\alpha)$ variable $A$ independent of iid \BM s $(B_1,B_2)$. Bottom: $(X,Y)= (A_1^{1/2}B_1,A_2^{1/2}B_2)$ for
iid copies $A_1,A_2$ independent of the \BM s $B_1,B_2$ with correlation $\rho=0.5$. From left to right: $\alpha=0.5,\,1.0,\,1.5$. 
Sample sizes $n=100,200,300$, $p=100$, and each plot is based on 
$500$ replications. 
}
\label{boxplot:heavydependent1}
\end{figure}

\begin{figure}[t]
\begin{center}
\includegraphics[width=\textwidth]{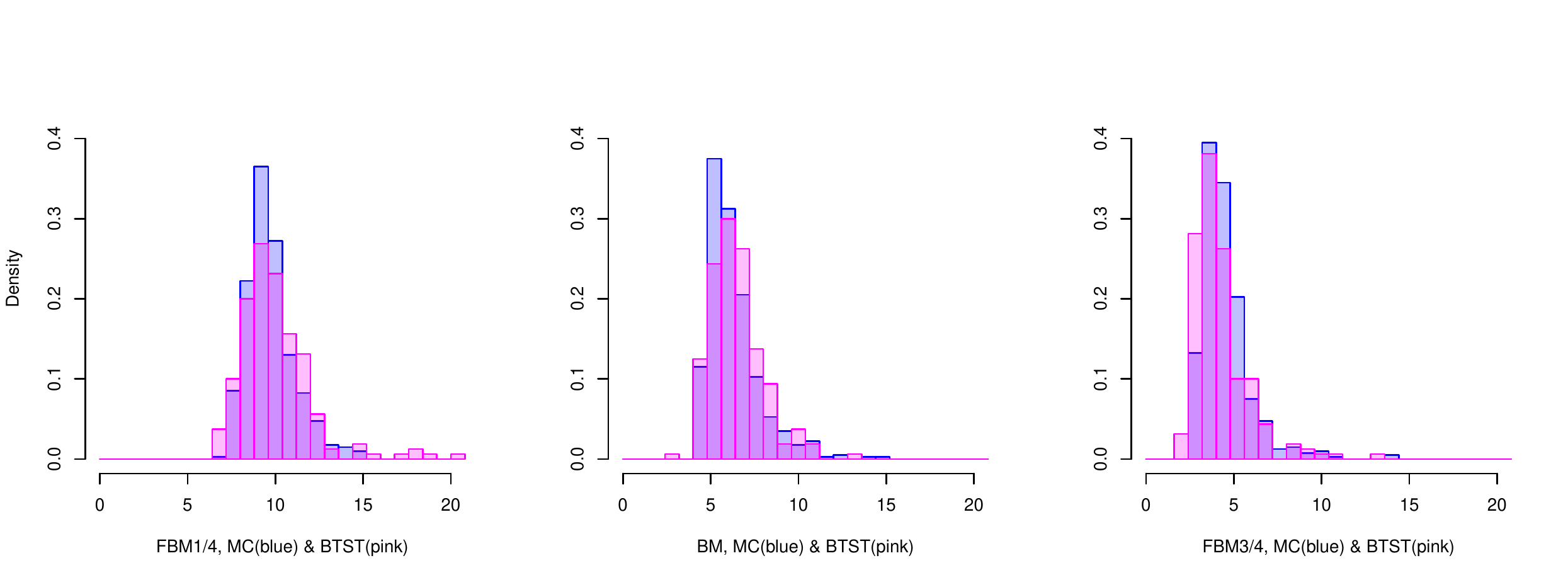}
\includegraphics[width=\textwidth]{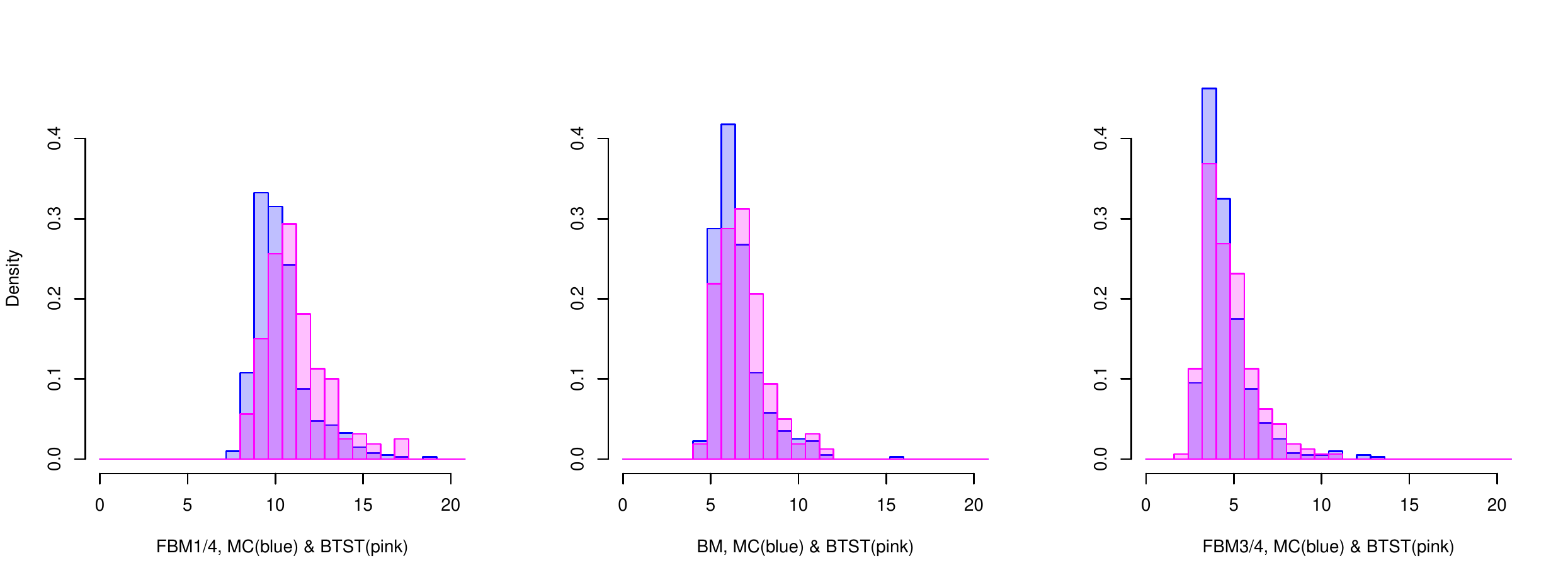}
\end{center}
\caption{Comparison of histograms for $nR_n(X^{(p)},Y^{(p)})$ based on 
Monte Carlo simulation (blue) and bootstrap (pink) for iid fBMs $X$, $Y$
with $H=1/4$, $H=1/2$, $H=3/4$ 
(from left to right). The 
sample size is  $n=100$ (top) and $n=300$ (bottom) and $p=100$. 
The histograms of 
$nR_n(X^{(p)},Y^{(p)})$ and the bootstrap version are based 
on $500$  and $200$ replications, respectively.} 
\label{btstgauss}
\end{figure}
In Figure~\ref{nongaussianprocess} we illustrate the performance of 
the sample distance correlation $R_n(X^{(p)},Y^{(p)})$ when
 $X$ and $Y$ are independent (possibly with distinct \ds s) non-Gaussian
 processes. 
We treat three cases, including heavy-tailed processes: $X,Y$ are 
iid geometric BMs (left), $X,Y$ are iid
$\alpha$-stable L\'evy motions (middle), $X$ is a geometric BM and $Y$
an $\alpha$-stable L\'evy motion (right). 
For geometric BM we choose the parametrization
\[
 X(t)= \exp\big(
(1- 0.7^2/2)t+0.7 B(t),\qquad t\in [0,1]\,,
\]
where $\mu=1$ (drift), $\sigma=0.7$ (volatility) and $B$ is standard
BM. The parameters of the $\alpha$-stable L\'evy motions are
$(\alpha,\beta,\mu,\sigma)=(1.8,0.3,0,1)$; cf. \cite[Ex. 3.1.3]{samorodnitsky:taqqu:1994}. We fix $p=100$ and increase $n$ from $100$ to $300$.
Also in these non-Gaussian settings  
the boxplots nicely illustrate consistency of $R_n(X^{(p)},Y^{(p)})$
even in the heavy-tailed $\alpha$-stable case.  
\par
In Figure~\ref{boxplot:gridnumber}
we study the influence of the size of $p$
on the sample distance correlation for a given $n$.
We choose 
$p=100$ (left) and $p=1000$ (middle) while $X$, $Y$ are independent BMs: 
there is hardly any difference between the left and middle graphs
for a given $n$. 
In the right graph we choose iid $\alpha$-stable L\'evy motions $X,Y$ with
the same parameters as before. We increased $p$ from $100$ to $1000$
and fix $n=100$. Again,
one can hardly see any difference between the boxplots.
These observations are not surprising -- in view of
the definition of the distance correlation and the independence of 
$X^{(p)}$ and $Y^{(p)}$ for any $p$. However,
it is perhaps unexpected that $n$ and $p$ may have similar size and 
still provide good approximations to zero.
In Figure \ref{btstgauss} we visualize how the bootstrap works 
for the normalized sample distance correlations
$nR_n(X^{(p)},Y^{(p)})$ for iid fBMs $X$, $Y$. We show 
histograms based on 500 replications of
$nR_n(X^{(p)},Y^{(p)})$ and compare with the histograms based 
on $200$ replications of the bootstrap version generated from 
a single sample.
We
see that the distributions of $nR_n(X^{(p)},Y^{(p)})$ and 
its bootstrap version are close to each other and get more concentrated.
\par
We also examine some dependent heavy-tailed cases.
We have chosen two simple \spr\ models for $X,Y$ where we can control the tails and the dependence. First, we consider iid standard BMs $B_1,B_2$
which are subject to a joint heavy-tailed shock, 
$(X,Y)=A^{1/2}(B_1,B_2)$, where $A$ is a Pareto$(\alpha)$ variable 
for some $\alpha>0$  with density $f_\alpha(x)=\alpha (1+x)^{-(\alpha+1)},\,x>0$. We also assume that $A$ and $(B_1,B_2)$ are independent.
Notice that $A^{1/2}$ does not have a $2\alpha$th moment.
Second, we consider $(X,Y)=(A_1^{1/2}B_1, A_2^{1/2} B_2)$ where $A_1,A_2$ are iid
copies of $A$ with density $f_\alpha$, independent of $(B_1,B_2)$ 
while $B_1$ and $B_2$ are dependent BMs  with correlation
$\rho=0.5$. We have chosen $2\alpha=1,2,3$. In the case $\alpha=0.5$ the theoretical results of this paper about consistency of $T_n(X^{(p)},Y^{(p)})$ do not apply since $\E[\|X\|_2+ \|Y\|_2]=\infty$ while in the cases $\alpha=1,1.5$, $T_n(X^{(p)},Y^{(p)})\stp T(X,Y)>0$.
\par The first/second model is examined in the top/bottom graphs of Figure~\ref{boxplot:heavydependent1}, respectively.
In the cases $\alpha=1,1.5$ the centers of the boxplots seem to stabilize
with increasing sample size, pointing at the consistency of 
$R_n(X^{(p)},Y^{(p)})$. In the top graphs (first model) we observe
that the \ds s of $R_n(X^{(p)},Y^{(p)})$ have a rather wide range while
the bottom boxplots (second model) are less spread and their center 
is much below those of the first model.
 Moreover, in the $\alpha=0.5$ case the
plot is close to zero. It could be taken as a false indication of independence
between $X$ and $Y$. We do not have a full explanation 
for the phenomena observed in Figure~\ref{boxplot:heavydependent1}; 
in both heavy-tailed dependent models our assumptions for the existence 
of non-degenerate weak limits are not satisfied due to the lack of moments.

\section{Proof of Theorem~\ref{thm:1}}\label{sec:proofthm2.1}\setcounter{equation}{0}
We prove the theorem by a series of auxiliary results.
\bpr\label{prop:1}
Assume the conditions {\rm 1.-4.} of Theorem~\ref{thm:1}.
\begin{enumerate}
\item[\rm 1.]
If also {\rm  (A1)} holds then there is $c$ \st\ for any  $n\ge 1$, 
 \beao
 \E\big[ |T_{n,\beta}(X^{(p)},Y^{(p)})-T_{n,\beta} (X,Y) |\big] \le c\,
     \delta_n^{(\gamma_X\wedge \gamma_Y) \,(\beta\wedge 1)/2}\,. 
\eeao
\item[\rm 2.]
If also {\rm (B1),(B2)} hold then there is $c$ \st
\beao
\E\big[ |T_{n,\beta}(X^{(p)},Y^{(p)})-T_{n,\beta} (X,Y) |\big] \le c\, \big(
p\,\delta_n^{(\beta/2+\gamma_X\wedge \gamma_Y)}\big)^{(\beta\wedge 1)/\beta}\,.
\eeao
\end{enumerate}
\epr

\begin{proof}
We start with the decomposition
\beam\label{eq:decom}
T_{n,\beta} (X^{(p)},Y^{(p)})-T_{n,\beta} (X,Y)= I_1+I_2-2 I_3\,,
\eeam
where
\beam
\label{eq:Is}
I_1&=&\frac{1}{n^2} \sum_{k,l=1}^n \big(\|X_{k}^{(p)}-X_{l}^{(p)}\|_2^\beta \| Y_{k}^{(p)}-Y_{l}^{(p)}\|_2^\beta - \|X_k-X_l \|_2^\beta \|Y_l-Y_k\|_2^\beta\big)\,,
 \nonumber \\
I_2&=&  \frac{1}{n^4} 
\sum_{k,l=1}^n \|X_{k}^{(p)}-X_{l}^{(p)}\|_2^\beta \sum_{k,l=1}^n
 \|Y_{k}^{(p)}-Y_{l}^{(p)}\|_2^\beta \nonumber\\
&& \quad - \frac{1}{n^4} 
 \sum_{k,l=1}^n\|X_k-X_l\|_2^\beta \sum_{k,l=1}^n
 \|Y_k-Y_l\|_2^\beta\,,\nonumber \\
I_3&=& \frac{1}{n^3}
\sum_{k,l,m=1}^n \|X_{k}^{(p)}-X_{l}^{(p)}\|_2^\beta
 \|Y_{k}^{(p)}-Y_{m}^{(p)}\|_2^\beta \nonumber\\
&&\quad- \frac{1}{n^3} \sum_{k,l,m=1}^n 
 \|X_k-X_l\|_2^\beta\|Y_k-Y_m\|_2^\beta\,.
\eeam
We will find bounds for the absolute values of the expectations of these quantities. From now on, $c$ denotes any positive constants
whose values are not of interest. 
\par
{\em First assume that $(X,Y)$ have finite second moment.} Observe that
\begin{align}\label{eq:I11I12}
 |I_1| &\le \frac{1}{n^2} \sum_{k,l=1}^n \big|
 \|X_{k}^{(p)}-X_{l}^{(p)}\|_2^\beta -\|X_k-X_l\|_2^\beta\big|\,
 \|Y_{k}^{(p)}-Y_{l}^{(p)}\|_2^\beta \nonumber \\
& \quad + \frac{1}{n^2} \sum_{k,l=1}^n \big|
 \|Y_{k}^{(p)}-Y_{l}^{(p)}\|_2^\beta -\|Y_k-Y_l\|_2^\beta 
\big| \|X_k-X_l\|_2^\beta  =: I_{11}+I_{12}. 
\end{align}
By a symmetry argument, interchanging the roles of $X$ and $Y$, it suffices to consider $I_{11}$.
Using the independence of $X$ and $Y$, we have
\begin{align*}
 \E [I_{11}]\le  \E 
\big[\big|
 \|X_{1}^{(p)}-X_{2}^{(p)}\|_2^\beta -\|X_1-X_2\|_2^\beta\big|
\big] \,\E [\|Y_{1}^{(p)}-Y_{2}^{(p)}\|_2^\beta]\,. 
\end{align*}
By Lyapunov's inequality,
\beao
\E [\|Y_{1}^{(p)}-Y_{2}^{(p)}\|_2^\beta]&\le& (\E [\|Y_{1}^{(p)}-Y_{2}^{(p)}\|_2^2])^{\beta/2}\\
&\le &c\,\Big(\int_0^1 \var(Y^{(p)}(t))\,dt\Big)^{\beta/2}<\infty\,.
\eeao
Assume $0<\beta \le 1$. Then, by concavity and Jensen's inequality,
\beam\label{eq:star}
\lefteqn{\E \big[\big| \|X_{1}^{(p)}-X_{2}^{(p)}\|_2^\beta -\|X_1-X_2\|_2^\beta \big|\big]} \\
& \le& \E\big[ \|(X_{1}^{(p)}-X_{2}^{(p)})-(X_1-X_2)\|_2^\beta\big] \nonumber \\
& =& \E \Big[\Big(
\sum_{i=1}^{p} \int_{\Delta_i} \big(\Delta X_1(t,t_i]-\Delta X_2(t,t_i]\big)^2\, dt  
\Big)^{\beta/2}\Big]\nonumber\\ 
& \le & \Big(
\sum_{i=1}^{p} \int_{\Delta_i} \var\big(\Delta X_1(t,t_i]-\Delta X_2(t,t_i]\big)\, dt  
\Big)^{\beta/2} \nonumber\\
&  =& 
\Big(
\sum_{i=1}^{p} \int_{\Delta_i} \big(\var(\Delta X_1(t,t_i])+\var(\Delta X_2(t,t_i])\big)
 dt 
\Big)^{\beta/2} 
 \le c\, \delta_n^{\gamma_X\beta/2}. \nonumber
\eeam
The last step follows from (A1).
If $1<\beta<2$, we use the inequality  $|x^\beta-y^\beta|\le \beta 
 (x\vee y)^{\beta-1}|y-x|$ for positive $x,y$ and H\"older's inequality to obtain  
\beam
\lefteqn{ \E \big[\big|
 \|X_{1}^{(p)}-X_{2}^{(p)}\|^\beta_2 - \|X_1-X_2\|^\beta_2
\big|\big]} \nonumber \\
& \le & c\,\E \big[\big(\|X_{1}^{(p)}-X_{2}^{(p)}\|_2^{\beta-1}\vee \|X_1-X_2\|_2^{\beta-1}\big)\,\big| 
\|X_{1}^{(p)}-X_{2}^{(p)}\|_2 - \|X_1-X_2\|_2
\big|\big]\nonumber\\ 
&\le &c\,\E \big[\big(\|X_{1}^{(p)}-X_{2}^{(p)}\|_2^{\beta-1}\vee \|X_1-X_2\|_2^{\beta-1}\big)\,\|
(X_{1}^{(p)}-X_2^{(p)})-(X_{1}-X_2)\|_2 \big]\nonumber \\ 
 & \le& c\, \Big(\E \big[\|X_{1}^{(p)}-X_{2}^{(p)}\|_2^{2}\vee
 \|X_1-X_2\|_2^{2}\big]\Big)^{(\beta-1)/2} \nonumber \\
&&
\hspace{1.0cm} \times
 \Big(\E\big[\|(X_{1}^{(p)}-X_1)-(X_{2}^{(p)}-X_2)\|_2^{2/(3-\beta)}
 \big]\Big)^{(3-\beta)/2} =c\,P_1\,P_2\,.\label{eq:2star} 
\eeam
Since $(3-\beta)^{-1}<1$ the same arguments as in the case $0<\beta<1$ yield $P_2 \le c\, \delta_n^{\gamma_X/2}$.
Moreover, we have
\beao
P_1^{2/(\beta-1)}\le \E \big[\|X_{1}^{(p)}-X_{2}^{(p)}\|_2^2\big]+ \E \big[\|X_{1}-X_{2}\|_2^2\big] =P_{11}+P_{12}\,.
\eeao
It follows from Remark~\ref{rem:1} that $P_{12}<\infty$ and a 
similar argument yields $P_{11}<\infty$.
\par
Summarizing the previous bounds for $0<\beta<2$ under (A1), we have
\beao
 \E [I_{11}] &\le& c\,\delta_n^{(\gamma_X\wedge\gamma_Y)\,(\beta\wedge 1)/2}\,.
\eeao
Now we turn to $I_2$.
Observe that
\beao
|I_2| &\le& \frac{1}{n^2} \sum_{k,l=1}^n \big|
 \|X_{k}^{(p)}-X_{l}^{(p)}\|_2^\beta -\|X_k-X_l\|_2^\beta 
\big| \frac{1}{n^2} \sum_{k,l=1}^n \|Y_{k}^{(p)}-Y_{l}^{(p)}\|_2^\beta  \\
&&\quad + \frac{1}{n^2}\sum_{k,l=1}^n \|X_k-X_l\|_2^\beta  \frac{1}{n^2}
 \sum_{k,l=1}^n  \big|
\|Y_{k}^{(p)}-Y_{l}^{(p)}\|_2^\beta  -\|Y_k-Y_l\|_2^\beta 
\big|\,,
\eeao
and a similar bound exists for $|I_3|$. The same arguments as above yield
\beao
\E[|I_2+I_3|]\le  c\,\delta_n^{(\gamma_X\wedge\gamma_Y)\,(\beta\wedge 1)/2}\,.
\eeao
We omit further details.
\par
{\em Next assume that $(X,Y)$ have finite $\beta$th moment 
for some $\beta\in (0,2)$.} We follow the patterns of the proof in the finite variance
case. We start by bounding $\E[|I_1|]$. First assume $\beta\in (0,1]$. Following \eqref{eq:star}, we have by (B2), 
\beao 
&& \E \Big[\Big(\sum_{i=1}^{p} \int_{\Delta_i} \big(\Delta
 X_1(t,t_i]-\Delta X_2(t,t_i]\big)^2\, dt \Big)^{\beta/2} \Big] \\
&& \le c\,\sum_{i=1}^p |\Delta_i|^{\beta/2}\,\E\big[ \max_{t\in
 \Delta_i}\big|\Delta X(t, t_i]\big|^\beta \big]
\le c\,p\,\delta_n^{\beta/2+\gamma_X}\,.
 \eeao
Now assume $1<\beta<2$. Following \eqref{eq:2star}, we have by  H\"older's inequality,
\begin{align*}
\lefteqn{ \E \big[\big|
 \|X_{1}^{(p)}-X_{2}^{(p)}\|^\beta_2 - \|X_1-X_2\|^\beta_2
\big|\big]} \nonumber \\
&\le c\,\E \big[\big(\|X_{1}^{(p)}-X_{2}^{(p)}\|_2^{\beta-1}\vee \|X_1-X_2\|_2^{\beta-1}\big)\,
\|(X_{1}^{(p)}-X_1)-(X_{2}^{(p)}-X_2)\|_2\big]\nonumber \\
 & \le c\, \Big(
 \E\big[\|X_{1}^{(p)}-X_{2}^{(p)}\|_2^{\beta}\vee \|X_1-X_2\|_2^{\beta}\big]
 \Big)^{(\beta-1)/\beta}\,\nonumber \\
&\qquad \times
\Big( \E\Big[ 
 \|(X_{1}^{(p)}- X_1)-(X_{2}^{(p)}-X_2)\|_2^\beta\Big]
 \Big)^{1/\beta} \nonumber \\
&= c\,\wt P_1\wt P_2\,.\nonumber
\end{align*}
Proceeding as for $0<\beta<1$, we have
\beao
\wt P_2= \Big( \E\Big[ 
 \|(X_{1}^{(p)}- X_1)-(X_{2}^{(p)}-X_2)\|_2^\beta\Big]
 \Big)^{1/\beta}\le c\, \big(p\,\delta_n^{\beta/2+\gamma_X}\big)^{1/\beta}\,.
\eeao
We also have
\beao
\wt P_1^{\beta/(\beta-1)} \le 
\E\big[\|X_{1}^{(p)}-X_{2}^{(p)}\|_2^{\beta}\big]+\E\big[\|X_1-X_2\|_2^{\beta}\big]\,.
\eeao
The \rhs\ is finite by assumption (B1). Collecting bounds for $0<\beta<2$, we arrive at
\beao
\E[|I_1|]\le c\,\big(p\,\delta_n^{\beta/2+\gamma_X\wedge \gamma_Y}\big)^{1\wedge \beta^{-1}}\,.
\eeao
The quantities $\E[|I_i|]$, $i=2,3$, can be bounded in a similar way.
\end{proof}
Now we can finish the proof of the first two parts of Theorem~\ref{thm:1}. We assume that either (A1) or [(B1),(B2)
and $p\,\delta_n^{\beta/2+\gamma_X\wedge \gamma_Y}\to 0$]
are satisfied. Under these assumptions, it follows from Proposition~\ref{prop:1} 
that $T_{n,\beta}(X,Y)-T_{n,\beta}(X^{(p)},Y^{(p)})\stp 0$. 
The quantity $T_{n,\beta} (X,Y)$ can be written as a $V$-statistic of order 4
of the sample $((X_i,Y_i))_{i=1,\ldots,n}$; see Appendix~\ref{asymptotic}. 
(\cite{lyons:2013} used a $V$-statistics of order 6. The higher order
leads to a higher numerical complexity for the calculation of the 
bootstrap quantities.)
Since $X,Y$ are assumed independent and  
$\E [\|X\|_2^{\beta}]+\E [\|Y\|_2^{\beta}]<\infty$ 
(see Remark~\ref{eq:aug4a}) we may apply the \slln\ to the  $V$-statistic 
 $T_{n,\beta} (X,Y)$ implying that 
\beam\label{eq:slln1}
T_{n,\beta}(X,Y)&\stas& 
T_{\beta}(X,Y)=0\,.
\eeam
Hence the first parts of the theorem follow.
\par
Under the corresponding growth conditions (A2) and (B4) on $\delta_n\to 0$,  Proposition~\ref{prop:1} also yields
$n\,(T_{n,\beta}(X,Y)-T_{n,\beta}(X^{(p)},Y^{(p)}))\stp 0$.
Then we can use the fact that the $V$-statistic 
$T_{n,\beta}(X,Y)$ is degenerate of order 1 to conclude that
$n\,T_{n,\beta}(X,Y)$ converges in \ds\ to a series of independent weighted 
$\chi^2$-distributed \rv s, and $n\,T_{n,\beta}(X^{(p)},Y^{(p)})$ has the same
weak limit; we refer to \cite{arcones:gine:1992}, \cite{serfling:1980} for general limit theory on $U$- and $V$-statistics.

Next we prove (3) and (4).
 In view of the first two parts (1), (2) of the theorem they will follow if we can show consistency
of $T_{n,\beta}(X^{(p)},X^{(p)})$ and $T_{n,\beta}(Y^{(p)},Y^{(p)})$. This is the content of the following lemma.
\ble\label{lem:1} Assume the following conditions:
\begin{enumerate}
\item[\rm 1.]
$X$ is defined on $[0,1]$ and has Riemann square-integrable sample paths. 
\item[\rm 2.]
If $X$ has a finite first moment $X$ is centered. 
\item[\rm 3.] $\delta_n\to0$ as $\nto$.
\item[\rm 4.] $\beta\in(0,2)$.
\end{enumerate}
Moreover, consider the following conditions:
\begin{enumerate}
\item[{\rm (1)}]
$X$ has finite second moment and there exist $\gamma_X>0$ and $c>0$ \st
\beam\label{eq:89}
\var \big(X(s,t]\big)\le c\,|t-s|^{\gamma_X}\,,\qquad s<t\,.
\eeam
If $\beta\in(1,2)$
we also assume
\beam\label{e:new1}
\max_{0\le t\le 1} \E [|X(t)|^{2(2\beta-1)}]<\infty\,.
\eeam 
\item[{\rm (2)}] For some $\beta\in (0,1)$, 
\beam\label{eq:ooo}
\E\big[\max_{0\le t\le 1} |X(t)|^{2\beta}\big]<\infty\,,
\eeam
and there exist $\gamma_X'>0$ and $c>0$ \st\ 
\beam\label{eq:77}
\max_{i=1,\ldots,p}\E\big[\max_{t\in \Delta_i}|\Delta X(t,t_i]|^{2\beta}\big]\le c\, \delta_n^{\gamma_X'}\,,
\eeam
and $p\,\delta_n^{\beta+\gamma_X'}\to 0$.
\end{enumerate}
If  either {\rm (1)} or {\rm (2)} hold then
\beao
T_{n,\beta}(X^{(p)},X^{(p)})-T_{n,\beta}(X,X)\stp 0. 
\eeao 
Moreover, we also have  
\beam\label{eq:4a}
T_{n,\beta}(X^{(p)},X^{(p)})\stp T_\beta(X,X),
\eeam
where 
\beao
T_\beta(X,X) =\E \big[\|X_1-X_2\|_2^{2\beta}\big]+ \big(\E \big[\|X_1-X_2\|_2^{\beta}\big]\big)^2
-2\, \E \big[\|X_1-X_2\|_2^{\beta}\,\|X_1-X_3\|_2^{\beta}\big]\,.
\eeao
\ele
Note that \eqref{eq:89} and \eqref{e:new1} are
contained in conditions (A1) and (A3), respectively, 
while \eqref{eq:ooo} and \eqref{eq:77} are contained in  (B3). 
Therefore the conditions of Lemma \ref{lem:1} are satisfied
if those of Theorem 3.1, (3) and (4), hold.

\begin{proof} 
{\em We assume condition {\rm (1)}.} We use the decomposition \eqref{eq:decom}
and follow the lines of the proof of Proposition 7.1 
In this case,
\beam\label{eq:decomvar}
I_1&=& \dfrac 1 {n^2}\sum_{k,l=1}^n\big(
 \|X_k^{(p)}-X_l^{(p)}\|_2^{2\beta}-
 \|X_k-X_l\|_2^{2\beta}\big)\,,\nonumber \\
I_2&=& \Big(\dfrac 1 {n^2}\sum_{k,l=1}^n  \|X_k^{(p)}-X_l^{(p)}\|_2^\beta\Big)^2
-\Big(\dfrac 1 {n^2}\sum_{k,l=1}^n  \|X_k-X_l\|_2^\beta\Big)^2\,,\\
I_3&=& \dfrac 1 {n^3} \sum_{k,l,m=1}^n \big(
 \|X_k^{(p)}-X_l^{(p)}\|_2^\beta \|X_k^{(p)}-X_m^{(p)}\|_2^\beta \nonumber\\
&&\qquad -\|X_k-X_l\|_2^\beta \|X_k-X_m\|_2^\beta\big)\,. \nonumber
\eeam
We start by considering $I_1$. First assume that $\beta\le 1$. Observe
 that 
\begin{align*}
 \E[|I_1|] & \le \E \big[\|(X_1^{(p)}-X_1)-(X_2^{(p)}-X_2)\|^\beta_2\, \big(
\|X_1^{(p)}-X_2^{(p)}\|_2^\beta +\|X_1-X_2\|_2^\beta 
\big)\big] \\
& \le \big(
\E\big[\| (X_1^{(p)}-X_1)-(X_2^{(p)}-X_2) \|_2^{2\beta}\big]
\big)^{1/2} \\
&\quad \times \big(
(\E\big[\|X_1^{(p)}-X_2^{(p)}\|_2^{2\beta})^{1/2}+(\E\|X_1-X_2\|_2^{2\beta}\big] )^{1/2} 
\big)\,. 
\end{align*}
Similarly as in \eqref{eq:star}
the first expectation is bounded by $c\,\delta_n^{\beta \gamma_X}$,
 while the remaining two expectations are bounded, so that as  
in the proof 
of Proposition 7.1
, we have 
that 
\beao
\E [|I_1|]\le c\,\delta_n^{\beta \gamma_X/2}\,.
\eeao
If $1<\beta<2$ we may proceed as for  $\E[I_{11}]$ in the proof 
of Proposition 7.1 
 in the case $1<\beta<2$: 
\beam
\E[|I_1|]&\le& 
\E \big[\big|
 \|X_{1}^{(p)}-X_{2}^{(p)}\|^{2\beta}_2 - \|X_1-X_2\|^{2\beta}_2
\big|\big] \nonumber \\
& \le & c\,\E \big[\|X_{1}^{(p)}-X_{2}^{(p)}\|_2^{2\beta-1}\vee
 \|X_1-X_2\|_2^{2\beta-1}\nonumber \\
&&\qquad \times \big| 
\|X_{1}^{(p)}-X_{2}^{(p)}\|_2 - \|X_1-X_2\|_2
\big|\big]\nonumber \\ 
 & \le& c \Big(\E \big[\|X_{1}^{(p)}-X_{2}^{(p)}\|_2^{2(2\beta-1)}\vee
 \|X_1-X_2\|_2^{2(2\beta-1)}\big]\Big)^{1/2} \nonumber \\
&& \times \Big(\E\big[\|(X_{1}^{(p)}-X_1)-(X_{2}^{(p)}-X_2)\|_2^2
 \big]\Big)^{1/2} \nonumber \\
&=&c\,P_1\,P_2\,. \label{labelforP_2}
\eeam
We have $P_2\le c\,\delta_n^{\gamma_X/2}$ and
\beao
P_1^2\le \E \big[\|X_{1}^{(p)}-X_{2}^{(p)}\|_2^{2(2\beta-1)}\big]+\E\big[ \|X_1-X_2\|_2^{2(2\beta-1)}\big]=P_{11}+P_{12}\,.
\eeao
We deal only with $P_{12}$; $P_{11}$ can be bounded in a similar way. For $1<2\beta\le 2$, the \fct\ $f(x)=|x|^{2\beta-1}$ is concave.
Therefore 
\beao
P_{12}&=& \E\Big[\Big( \int_0^1 (X_1(t)-X_2(t))^2\,dt \Big)^{2\beta-1} \Big]\\
&\le &\Big(\E\Big[ \int_0^1 (X_1(t)-X_2(t))^{2}\,dt \Big]\Big)^{2\beta-1}<\infty\,. 
\eeao
In the last step we used \eqref{eq:89}.
\par
If  $2<2\beta<4$ we have by Lyapunov's inequality and \eqref{e:new1},
\beam
\label{bound:P12}
P_{12}&=& \E\Big[\Big( \int_0^1 (X_1(t)-X_2(t))^2\,dt \Big)^{2\beta-1}
 \Big] \\
&\le& \E\Big[\int_0^1 |X_1(t)-X_2(t)|^{2(2\beta-1)}\,dt
 \Big]<\infty\,. \nonumber
\eeam
Thus we proved that
\beao
\E[|I_1|]\le c\,\delta_n^{\gamma_X (\beta \wedge 1)/2}\,.
\eeao
We can deal with $I_2$ in the same way by observing that
\beam\label{eq:=}
I_2&=& \dfrac 1 {n^2}\sum_{k,l=1}^n \big( \|X_k^{(p)}-X_l^{(p)}\|_2^\beta- \|X_k-X_l\|_2^\beta\big)\nonumber\\\
&&\quad \times
\dfrac 1 {n^2}\sum_{k,l=1}^n \big( \|X_k^{(p)}-X_l^{(p)}\|_2^\beta+\|X_k-X_l\|_2^\beta\big)\nonumber\\
&=&\wt P_1\wt P_2\,.
\eeam
The expected value of $\wt P_2$ is bounded and hence $\wt P_2$ is stochastically bounded while similar calculations as
 for $I_1$ show that $\E[|\wt P_1|]\to 0$. Hence $I_2\stp 0$.
We have
\beam\label{eq:I3xx}
I_3&=& \dfrac 1 {n^3} \sum_{k,l,m=1}^n \big( \|X_k^{(p)}-X_l^{(p)}\|_2^\beta -\|X_k-X_l\|_2^\beta \big)
\|X_k^{(p)}-X_m^{(p)}\|_2^\beta \nonumber \\&&+
\dfrac 1 {n^3} \sum_{k,l,m=1}^n \|X_k-X_l\|_2^\beta \big(\|X_k^{(p)}-X_m^{(p)}\|_2^\beta-\|X_k-X_m\|_2^\beta\big) \nonumber\\
&=&I_{31}+I_{32}\,.
\eeam
We will deal only with $I_{32}$; the other case is similar. 
Assume $0<\beta\le 1$. By the Cauchy-Schwarz inequality and using similar 
bounds as above,
\beam\label{eq:!}
\E[|I_{32}|] &\le& \Big(
 \E\big[\|X_1-X_2\|_2^{2\beta}\big]\Big)^{1/2}\, \Big(
 \E\big[\big|
 \|X_1^{(p)}-X_3^{(p)}\|_2^\beta - \|X_1-X_3\|_2^\beta
 \big|^2\big]
 \Big)^{1/2} \nonumber\\
&\le &\Big(
\E\big[\|X_1-X_2\|_2^{2\beta}\big]
\Big)^{1/2} \,\Big(
 \E\big[
 \|(X_1^{(p)}-X_1) - (X_3^{(p)}-X_3)\|_2^{2\beta}\big]
 \Big)^{1/2} \nonumber \\
&& \to 0\,.
\eeam
Now assume $1<\beta<2$. Then 
\beao
\E[|I_{32}|]& 
\le& c\, \E \Big[\| (X_1^{(p)}-X_1)-(X_3^{(p)}-X_3)\|_2 \\ 
&& \hspace{1cm} \times \Big(\|X_1^{(p)}-X_3^{(p)}\|_2^{\beta-1} \vee
  \|X_1-X_3\|_2^{\beta-1} \Big) \,\|X_1-X_2\|^\beta_2\Big] \\
& \le& c\, \Big(
\E\big[ 
 \|(X_1^{(p)}-X_1)-(X_3^{(p)}-X_3)\|_2^{2}\big]
 \Big)^{1/2} \\
& &\times \Big(
\E\big[\big(\|X_1^{(p)}-X_3^{(p)}\|_2^{2(\beta-1)} \vee
 \|X_1-X_3\|_2^{2(\beta-1)}\big)\,
 \|X_1-X_2\|_2^{2\beta}\big]\Big)^{1/2} \\
& \le& c \Big(
\E\big[
 \|(X_1^{(p)}-X_1)-(X_3^{(p)}-X_3)\|_2^{2}\big]
 \Big)^{1/2} \\
&&  \times \Big\{ \Big(
\E\big[\|X_1^{(p)}-X_3^{(p)}\|_2^{2(\beta-1)} \|X_1-X_2\|_2^{2\beta}\big]\Big)^{1/2}\\
&&+ \Big(
\E\big[
 \|X_1-X_3\|_2^{2(\beta-1)}
 \|X_1-X_2\|^{2\beta}_2\big]\Big)^{1/2} \Big\}\,.
\eeao
The first factor is $P_2$ from above which is
bounded by $c \delta_n^{\gamma_X/2}$. 
 For the second term, we only consider
 $\E[\|X_1^{(p)}-X_3^{(p)}\|_2^{2(\beta-1)} \|X_1-X_2\|_2^{2\beta}\big]$ by a symmetry argument. An  
 application of  H\"older's inequality to 
 this quantity yields the bounds 
\[
 \big(
 \E \big[\|X_1^{(p)}-X_3^{(p)}\|^{2(2\beta-1)}_2\big]
\big)^{\frac{\beta-1}{2\beta-1}} \big(
\E\big[\|X_1-X_2\|_2^{2(2\beta-1)}\big]
\big)^{\frac{\beta}{2\beta-1}} =
 P_{11}^{\frac{\beta-1}{2\beta-1}} P_{12}^{\frac{\beta}{2\beta-1}},
\] 
where $P_{11},\,P_{12}$ are defined above and shown to be bounded.
This concludes the proof under condition (1).\\[2mm]
{\em We assume condition {\rm (2)}}. Now we prove the lemma under the condition that the moments of $X(t)$ of the order $2\beta\in (0,2)$ are finite.
We have for $2\beta\le 1$ by concavity and in view of condition \eqref{eq:77},
\beam\label{eq:*}
\E[|I_1|]&\le &\E \big[\|(X_1-X_1^{(p)})-(X_2-X_2^{(p)})\|_2^{2\beta}\big]\nonumber \\
&\le& c \,\delta_n^\beta \sum_{i=1}^p \E\big[\max_{t\in \Delta_i}|\Delta
 X(t,t_i]|^{2\beta}\big] \le c\, p\,\delta_n^{\beta+\gamma_X'}\,.
\eeam
The \rhs\ goes to zero by assumption.
For $2\beta\in (1,2)$ we have by H\"older's inequality,
\beam
\E[|I_1|]&\le& 
\E \big[\big|
 \|X_{1}^{(p)}-X_{2}^{(p)}\|^{2\beta}_2 - \|X_1-X_2\|^{2\beta}_2
\big|\big] \nonumber \\
& \le & c\,\E \big[\big(\|X_{1}^{(p)}-X_{2}^{(p)}\|_2^{2\beta-1}\vee
 \|X_1-X_2\|_2^{2\beta-1}\big) \nonumber\\
&&\hspace{0.8cm} \times \|(X_{1}^{(p)}-X_1)-(X_{2}^{(p)}-X_2)\|_2\big]\nonumber\\ 
 & \le& c \Big(\E \big[\|X_{1}^{(p)}-X_{2}^{(p)}\|_2^{2\beta}\vee \|X_1-X_2\|_2^{2\beta}\big]\Big)^{(2\beta-1)/(2\beta)}\nonumber \\
&& \times
\Big(\E\big[\|(X_{1}^{(p)}-X_1)-(X_{2}^{(p)}-X_2)\|_2^{2\beta}
 \big]\Big)^{1/(2\beta)} \nonumber\\
&=&c\,\wh P_1\,\wh P_2\,. \label{def:P_1:P_2}
\eeam
The quantity $\wh P_1$ is finite in view of \eqref{eq:ooo} and $\wh P_2\to 0$ by the argument of \eqref{eq:*}.
\par
For $I_2=\wt P_1\wt P_2$ we use \eqref{eq:=}. Since $\E[\|X_1-X_2\|_2^\beta]$ and  $\E[\|X_1^{(p)}-X_2^{(p)}\|_2^\beta]$ are finite the expectation of
$\wt P_2$ is bounded while
\beao
\E[|\wt P_1|]\le 2\,\E[\|X-X^{(p)}\|_2^\beta]\le 2\,\big(\E[\|X-X^{(p)}\|_2^{2\beta} ]\big)^{1/2}\,.
\eeao
The argument of \eqref{eq:*} shows that the \rhs\ converges to zero.
\par
Finally, we use the decomposition $I_3=I_{31}+I_{32}$. Inequality \eqref{eq:!} and the
bounds above show that $\E [|I_{32}|] \to 0$; the case $\E [|I_{31}|] \to 0$ follows in a similar way. 
\par
Collecting all bounds above, we proved $T_{n,\beta}(X^{(p)},X^{(p)})-T_{n,\beta}(X,X)\stp 0$ both under the conditions of (1) and (2).
Then relation \eqref{eq:4a} is immediate. Indeed, 
under the assumption $\E[\|X\|_2^{2\beta}]<\infty$ the \slln\ 
for $U$- and $V$-statistics yields $T_{n,\beta}(X,X)\stas T_\beta(X,X)$. 
\end{proof}


\appendix
\section{The sample distance covariance as a degenerate V-statistic}\label{asymptotic}\setcounter{equation}{0}
We assume that $Z_i=(X_i,Y_i)$, $i=1,2,\ldots,$ is an iid \seq\ with generic element $(X,Y)$ whose components are
Riemann square-integrable on $[0,1]$, and 
$\E[\|X\|_2^\beta+\|Y\|_2^\beta + \|X\|_2^\beta\|Y\|_2^\beta]<\infty$ and for some $\beta\in (0,2)$. Under the assumption of independence on $X,Y$
\cite{lyons:2013,lyons:2018} proved that $T_{n,\beta}(X,Y)$
has \rep\ as a $V$-statistic of order 6 with degenerate kernel of
order 1. In what follows, we will indicate that it can be written 
as a $V$-statistic of order 4 with symmetric degenerate kernel of order 1.
This fact is useful for improving upon the complexity of the numerical approximation of  the sample distance correlation and its bootstrap version.
\par
We start with the kernel
\begin{align*}
& f((x_1,y_1),(x_2,y_2),(x_3,y_3),(x_4,y_4))\ (=: f(z_1,z_2,z_3.z_4)) \\
& \quad = \|x_1-x_2\|_2^\beta\|y_1-y_2\|_2^\beta 
  +\|x_1-x_2 \|_2^\beta \|y_3-y_4 \|_2^\beta  -2\|x_1-x_2\|_2^\beta \|y_1-y_3\|_2^\beta.
\end{align*}
From this representation, it is obvious that 
\[
  T_{n,\beta}(X,Y) =\frac{1}{n^4} \sum_{1\leq i,j,k,l \leq n} f(Z_i,Z_j,Z_k,Z_l). 
\]
Then one can define the corresponding symmetric kernel via 
the usual symmetrization as
\begin{equation} \label{e:h4}
 h(z_1,z_2,z_3,z_4)=\frac{1}{24} \sum_{ (l_1,l_2,l_3,l_4) \mbox{ permutation of  }(1,2,3,4) }
 f(z_{l_1},z_{l_2},z_{l_3},z_{l_4}).
\end{equation}
It is not difficult to see that the kernel $h$ is at least $1$-degenerate,
by showing that, under the null hypothesis of
  independence of $X$ and $Y$,
\begin{align*}
&  \E [f(z_1,Z_2,Z_3,Z_4)] + \E [f(Z_2,z_1,Z_3,Z_4)] +  \E
 [f(Z_2,Z_3,z_1,Z_4)] \\
&\qquad + 
\E [f(Z_2,Z_3,Z_4,z_1)]=0\,. 
\end{align*}
Still under the null  hypothesis of   independence of $X$ and $Y$, 
\beao
\lefteqn{\E [h(z_1,z_2,(X_3,Y_3),(X_4,Y_4))]}\\
 &=& \frac{1}{6} \Bigl( \| x_1-x_2\|_2^\beta + \E[ \|
X_1-X_2\|_2^\beta] - \E [\| x_1-X\|_2^\beta] -  \E[ \| x_2-X\|_2^\beta]\Bigr) \\
&&\ \ \times\Bigl( \| y_1-y_2\|_2^\beta + \E [\|
Y_1-Y_2\|_2^\beta] - \E[ \| y_1-Y\|_2^\beta] -  \E[ \|
y_2-Y\|_2^\beta]\Bigr)\,, 
\eeao
and the \rhs\ is not constant. Hence, the kernel $h$ is
precisely $1$-degenerate. In summary: 
\ble
If $X,Y$ are independent and $\E[\|X\|_2^\beta+\|Y\|_2^\beta]<\infty$  for some $\beta \in (0,2)$ then $T_{n,\beta}(X,Y)$ has \rep\ as a
$V$-statistic with a symmetric kernel $h$ of order 4 which is 1-degenerate.
Moreover, the corresponding $U$-statistic $\wt T_{n,\beta}(X,Y)$, 
which is obtained from  $T_{n,\beta}(X,Y)$ by restricting the summation to indices 
$(i_1,i_2,i_3,i_4)$ with mutually distinct components, satisfies the
relation that as $\nto$
\beam\label{eq:serf}
\qquad n\,\big(T_{n,\beta}(X,Y)-\wt T_{n,\beta}(X,Y)\big)\stp
\E[\|X_1-X_2\|_2^\beta]\E[\|Y_1-Y_2\|_2^\beta]. 
\eeam 
\ele
Indeed, observe that
$\Delta_n=T_{n,\beta}-\wt T_{n,\beta}$ is based on summation of the kernel $h$
over indices $(i_1,i_2,i_3,i_4)$ for which at least two components coincide.
If more than 2 indices coincide the number of these summands
in  $\Delta_n$ is of the order $O(n^2)$. However,
the normalization in $n\Delta_n$ is of the order $n^3$. 
Therefore the sum of these terms
is negligible as  $\nto$. Finally,  the part of the sum corresponding
to the case when exactly two indices coincide 
and the other indices   are different,  can be written 
as a $U$-statistic of order 3. By the law of large numbers, this $U$-statistic
converges a.s. to   
$\E[\|X_1-X_2\|_2^\beta]\E[\|Y_1-Y_2\|_2^\beta]$. 
\bre 
The additional moment assumption on $h(Z_{i_1},Z_{i_2},Z_{i_3},Z_{i_4})$, $1\leq i_1\leq i_2\leq i_3\leq i_4\leq 4$, required in Corollary~\ref{cor:dehlmik} 
is satisfied for our kernel. Note that it suffices to consider the non-symmetric kernel $f$, and to show that
$ \E [(f(Z_{i_1},Z_{i_2},Z_{i_3},Z_{i_4}))^2]<\infty$, for all indices $1\leq i_1,\ldots,i_4\leq 4$. For our specific kernel, this condition reads 
\[
 \E \left[\left(\|X_{i_1}-X_{i_2}\|^\beta \left[\|Y_{i_1}-Y_{i_2}\|^\beta +\| Y_{i_3}-Y_{i_4} \|^\beta
 -2 \|Y_{i_1}-Y_{i_3}\|^\beta    \right] \right)^2\right] <\infty,
\]
and this holds under the moment conditions made in this paper.
\ere

\section{Bootstrap consistency for Section 5}
\label{appendix:b}
For the proof of the bootstrap consistency in Section 5 
we need a.s. \con\ of $T_{n,1}(X^{(p)},Y^{(p)})=:T_n(X^{(p)},Y^{(p)})$. We give 
some sufficient conditions.
\ble\label{lem:consist}
Assume the following conditions on the Riemann square-integrable process
$X$ on $[0,1]$.
\begin{enumerate} 
\item[\rm 1.] $\E[\|X\|_2^2]<\infty$ and $\E[X(u)]=0$ for $u\in [0,1]$.
\item[\rm 2.] {\rm (A.1)} holds.
\item[\rm 3.] $E[|X(t)-X(s)|^4] \le c|t-s|^{\wt \gamma_X}$ holds for some 
$\wt \gamma_X>0$.
\item[\rm 4.] $\sum_{n=1}^\infty n^{-1} \big(\delta_n^{\gamma_X}+\delta_n^{\wt\gamma_X}\big)<\infty$.
\end{enumerate}
Then $T_n(X^{(p)},X^{(p)})\stas T(X,X)$
holds as $\nto$.
\ele
\begin{proof} From \eqref{eq:decom} recall the decomposition 
$T_n(X^{(p)},X^{(p)})-T_n(X,X)=I_1+I_2-2I_3$; see also  \eqref{eq:decomvar}.
Since $\E[\|X\|_2^2]<\infty$, by the \slln\ for $V$-statistics, 
$T_n(X,X)\stas T(X,X)$. Therefore it suffices to show that
\beao
I_i&\stas& 0\,,\qquad i=1,3\,,\\
I_2'&:=&\dfrac 1 {n^2}\sum_{k,l=1}^n  \|X_k^{(p)}-X_l^{(p)}\|_2-\dfrac 1 {n^2}\sum_{k,l=1}^n  \|X_k-X_l\|_2\stas 0\,.
\eeao
We have 
\beao
|I_2'|&\le& \dfrac 1n \sum_{k}^n  \|X_k^{(p)}-X_k\|_2\\
&=& \dfrac 1n \sum_{k}^n  \big(\|X_k^{(p)}-X_k\|_2-\E[ \|X^{(p)}-X\|_2]\big)
+ \E[ \|X^{(p)}-X\|_2]\,.
\eeao
By Jensen's inequality,
\beao
\E[ \|X^{(p)}-X\|_2]&\le & \Big(\int_0^1 \var(X^{(p)}(u)-X(u))\,du\Big)^{1/2}\le 
\delta_n^{\gamma_X/2}\to 0.
\eeao
Moreover,
\beao
\var\Big( \dfrac 1n \sum_{k}^n  \big(\|X_k^{(p)}-X_k\|_2\big)\Big)\le  n^{-1} \E[\|X^{(p)}-X\|_2^2]\le n^{-1} \delta_n^{\gamma_X}\,. 
\eeao
Using Markov's inequality and the Borel-Cantelli lemma, we conclude that
$I_2'\stas 0$ if  $\sum_nn^{-1} \delta_n^{\gamma_X}<\infty$.
\par
The proof of $I_1\stas 0$ is similar. We have by the Cauchy-Schwarz inequality,
\beao
|I_1|&\le & \Big(\dfrac 1 {n^2}\sum_{k,l=1}^n \big(\|X_k^{(p)}-X_l^{(p)}\|_2-\|X_k-X_l\|_2\big)^2\Big)^{1/2}\\&&\times \Big(\dfrac 1 {n^2}\sum_{k,l=1}^n \big(\|X_k^{(p)}-X_l^{(p)}\|_2+\|X_k-X_l\|_2\big)^2\Big)^{1/2}\\
&\le &c\,\dfrac 1 {n^2}\sum_{k,l=1}^n \big(\|X_k^{(p)}-X_k\|_2+\|X_l-X_l^{(p)}\|_2\big)^2\,\\&&+c\,\Big(\dfrac 1 {n^2}\sum_{k,l=1}^n \big(\|X_k^{(p)}-X_k\|_2+
\|X_l-X_l^{(p)}\|_2\big)^2\Big)^{1/2} \\
&&\hspace{1cm} \times
 \Big(\dfrac 1 {n^2}\sum_{k,l=1}^n \|X_k-X_l\|_2^2\Big)^{1/2}\,.
\eeao
Therefore it remains to show that
\beao
\dfrac 1 {n}\sum_{k=1}^n \big(\|X_k^{(p)}-X_k\|_2^2-\E[\|X^{(p)}-X\|_2^2]\big)
+  \E[\|X^{(p)}-X\|_2^2] \stas 0\,.
\eeao
But we have $\E[\|X^{(p)}-X\|_2^2]=O(\delta_n^{\gamma_X})$ and
\beao
\var\Big(
\dfrac 1 {n}\sum_{k=1}^n \big(\|X_k^{(p)}-X_k\|_2^2\big)\Big)
&\le& n^{-1} \E[\|X_k^{(p)}-X_k\|_2^4]\\&\le & n^{-1}
\int_0^1 \E [(X^{(p)}(u)-X(u))^4]\,du\le
n^{-1}\delta_n^{\wt \gamma_X}\,.
\eeao
Since we assume $\sum_n n^{-1}\delta_n^{\wt \gamma_X}<\infty$ applications of
Markov's inequality and the Borel-Cantelli lemma show that $I_1\stas 0$.
\par
Finally, we show $I_3\stas 0$. We have
\beao
I_3&=& \dfrac 1 {n^3} \sum_{k,l,m=1}^n \big( \|X_k^{(p)}-X_l^{(p)}\|_2
 -\|X_k-X_l\|_2 \big) \nonumber \\
&& \hspace{2cm} \times \big(\|X_k^{(p)}-X_m^{(p)}\|_2-\|X_k-X_m\|_2\big) \nonumber \\&&+
 \dfrac 2 {n^3} \sum_{k,l,m=1}^n \big( \|X_k^{(p)}-X_l^{(p)}\|_2 -\|X_k-X_l\|_2 \big)\,\|X_k-X_m\|_2\\
&=&I_{31}+I_{32}\,.
\eeao
The Cauchy-Schwarz inequality yields
\beao
|I_{31}|&\le & \dfrac 1 {n^2} \sum_{k,l=1}^n 
\big( \|X_k^{(p)}-X_l^{(p)}\|_2 -\|X_k-X_l\|_2 \big)^2\\&\le &c\,
\dfrac 1 {n} \sum_{k=1}^n \|X_k^{(p)}-X_k\|_2^2\stas 0\,,\\
|I_{32}|&\le & c \Big(\dfrac 1{n^2} \sum_{k,l=1}^n 
\big( \|X_k^{(p)}-X_l^{(p)}\|_2 -\|X_k-X_l\|_2 \big)^2\Big)^{1/2}  \times \Big(\dfrac 1 {n^2}\sum_{k,l=1}^n\|X_k-X_l\|_2^2\Big)^{1/2}\\
&\le &c\, \Big(\dfrac 1{n} \sum_{k=1}^n 
\big( \|X_k^{(p)}-X_k\|_2^2 \Big)^{1/2}\,
\Big(\dfrac 1 {n^2}\sum_{k,l=1}^n\|X_k-X_l\|_2^2\Big)^{1/2}\stas 0\,.
\eeao
This proves the lemma.
\end{proof}

\section{Asymptotic behavior under the alternative hypothesis}
\label{append:c}
In this section we obtain analogs of the previous results
under the alternative hypothesis when $X,Y$ are dependent. 
In this case we need conditions on $X,Y$ which are more restrictive than
in the independent case. 
We investigate the asymptotic behavior of $T_{n,\beta}(X^{(p)},Y^{(p)})$
and $R_{n,\beta}(X^{(p)},Y^{(p)})$ under the alternative.
\par
In view of (1.4), $T_\beta(X,Y),\,T_\beta(X,X)$ and
$T_\beta(Y,Y)$, hence  
$R_{\beta}(X,Y)$, are
finite if 
\beao
\E[\|X\|_2^{2\beta}+\|Y\|^{2\beta}_2]
<\infty\,.
\eeao 
\begin{proposition}
\label{prop:alternative}
 Assume the following conditions:
\begin{enumerate}
\item [\rm 1.]
$X,Y$ are (possibly dependent) stochastically continuous bounded processes on $[0,1]$ defined
on the same \pro y space. 
\item[\rm 2.]
If  $X,Y$ have finite expectations, then these are assumed to be equal
to 0.
\item[\rm 3.] $\delta_n\to0$ as $\nto$.
\item[\rm 4.] $\beta\in (0,2)$.
\end{enumerate} 
Then the following statements hold.
\begin{enumerate}
\item[\rm (1)]
If either 
{\rm (A1),(A3)}
or $\big[\mbox{{\rm (B3)} and 
$p\,\delta_n^{\beta+\gamma_X'\wedge\gamma_Y'}\to 0$}\big]$ hold. 
Then 
\beam\label{consist:alternative}
  T_{n,\beta}(X^{(p)},Y^{(p)})-T_{n,\beta}(X,Y) \stackrel{\P}{\to} 0, 
\eeam
and 
\begin{align}
 \label{consistR:alternative}
 R_{n,\beta}(X^{(p)},Y^{(p)})-R_{n,\beta}(X,Y) \stackrel{\P}{\to} 0. 
\end{align}
\item[\rm (2)] If either {\rm (A1), (A3)} 
and 
\beam\label{eq:ccca}
 \delta_n = o(n^{-\frac{1}{(\beta \wedge 1)(\gamma_X\wedge \gamma_Y)}}
),\quad n\to\infty,
\eeam
or 
{\rm (B3)}
 and 
\beam\label{eq:cccb}
 \delta_n = o\big((pn)^{-\frac{1}{\beta+\gamma_X'\wedge \gamma_Y'}}\big),\quad n\to\infty
\eeam
hold, then 
\beam\label{eq:cltT_dep}
\sqrt{n}\,(T_{n,\beta}(X^{(p)},Y^{(p)}) -T_{n,\beta}(X,Y)) \stackrel{\P}{\to} 0,
\eeam 
and 
\beam\label{eq:cltR_dep}
\sqrt{n}\,(R_{n,\beta}(X^{(p)},Y^{(p)}) -R_{n,\beta}(X,Y))
	     \stackrel{\P}{\to} 0. 
\eeam 
\end{enumerate}
\end{proposition}

\begin{proof} {\bf Part (1).}
{\em First assume that $(X,Y)$ have finite second moment.} We follow the lines
of proof of Theorem 3.1 from the beginning until inequality (7.3). 
Again using a symmetry argument, it suffices to consider $I_{11}$. 

Assume $\beta\in(0,1]$. An application of the 
Cauchy-Schwarz inequality yields
\begin{align*}
 \E [I_{11}]\le  \big( \E 
\big[\big|
 \|X_{1}^{(p)}-X_{2}^{(p)}\|_2^\beta -\|X_1-X_2\|_2^\beta\big|^2
\big] \big)^{1/2} \,\big( \E [\|Y_{1}^{(p)}-Y_{2}^{(p)}\|_2^{2\beta}] \big)^{1/2}\,. 
\end{align*}
By Lyapunov's inequality,
\beao
\E [\|Y_{1}^{(p)}-Y_{2}^{(p)}\|_2^{2\beta}]&\le& (\E [\|Y_{1}^{(p)}-Y_{2}^{(p)}\|_2^2])^{\beta}\\
&\le &c\,\Big(\int_0^1 \var(Y^{(p)}(t))\,dt\Big)^{\beta}<\infty\,.
\eeao
Proceeding as for (7.4) with $\beta/2$ replaced by $\beta$, we have 
\beao
\E \big[\big| \|X_{1}^{(p)}-X_{2}^{(p)}\|_2^\beta
 -\|X_1-X_2\|_2^\beta \big|^2\big] 
\le \E\big[ \|(X_{1}^{(p)}-X_{2}^{(p)})-(X_1-X_2)\|_2^{2\beta}\big] 
 \le c\, \delta_n^{\gamma_X\beta}, 
\eeao
where the condition (A.1) is used. 

If $\beta\in(1,2)$, we use the inequality $|x^\beta-y^\beta|\le
 \beta(x\wedge y)^{\beta-1}|y-x|$ for positive $x,y$ and then the three-function H\"older
inequality with conjugates
 $(2(2\beta-1)/(\beta-1),2(2\beta-1)/\beta,2)$. This procedure yields 
\begin{align*}
& \E \big[\big|
 \|X_{1}^{(p)}-X_{2}^{(p)}\|^\beta_2 - \|X_1-X_2\|^\beta_2
\big|
\|Y_1^{(p)}-Y_2^{(p)}\|_2^\beta 
\big]   \\
& \le  c\,\E \big[\big(\|X_{1}^{(p)}-X_{2}^{(p)}\|_2^{\beta-1}\vee
 \|X_1-X_2\|_2^{\beta-1}\big)\,\|Y_1^{(p)}-Y_2^{(p)}\|_2^\beta \\
& \hspace{1cm} \times 
\big| 
\|(X_{1}^{(p)}-X_{2}^{(p)})-(X_1-X_2)\|_2
\big| \big] \\ 
& \le  c\,\big( \E \big[\big(\|X_{1}^{(p)}-X_{2}^{(p)}\|_2^{2(2\beta-1)}\vee
 \|X_1-X_2\|_2^{2(2\beta-1)}\big) \big]
 \big)^{\frac{\beta-1}{2(2\beta-1)}} \\
& \hspace{1cm} \times 
\big(\E\big[
\|Y_1^{(p)}-Y_2^{(p)}\|_2^{2(2\beta-1)}\big]\big)^{\frac{\beta}{2(2\beta-1)}}
 \Big(\E\big[\|(X_{1}^{(p)}-X_1)-(X_{2}^{(p)}-X_2)\|_2^{2}
 \big]\Big)^{1/2} 
 \\
& = c \ov P_1\cdot \ov P_2\cdot \ov P_3. 
\end{align*}
Similarly to the bound for $P_2$ in \eqref{labelforP_2},
we have $\ov P_3 \le c \delta_n^{\gamma_X/2}$
 and 
\[
 \ov P_1^{\frac{2(2\beta-1)}{\beta-1}} \le \E
 [\|X_1^{(p)}-X_2^{(p)}\|_{2}^{2(2\beta-1)} ]+ \E [\|X_1-X_2\|_2^{2(2\beta-1)}]    
 = \ov P_{11}+ \ov P_{12}.
\]
However, $\ov P_{11},\ov P_{12}$ are bounded similarly as $P_{11}, P_{12}$ 
in \eqref{bound:P12}. By a symmetry argument, $\ov P_2$ is bounded.
Thus we arrive at $\E[| I_1|] \le
 c\delta_n^{(\gamma_X\wedge\gamma_Y) (\beta \wedge 1)/2}$ for $\beta\in (0,2)$, and similar arguments prove 
\[
 \E[|I_2+I_3|] \le c \delta_n^{(\gamma_X\wedge \gamma_Y)(\beta
 \wedge 1)/2}. 
\]
We omit further details. 

{\em Next assume that $X,Y$ have finite $(2\beta)$th moment for some
 $\beta\in (0,1)$.} We follow the strategy of the proof in the finite
 variance case. We only bound $\E[|I_1|]$ since the quantities
 $\E[|I_i|]$\,,$i=2,3$ can be bounded in a similar manner. Again by 
the Cauchy-Schwarz inequality, 
 \[
   \E [I_{11}]\le  \big( \E 
\big[\big|
 \|X_{1}^{(p)}-X_{2}^{(p)}\|_2^\beta -\|X_1-X_2\|_2^\beta\big|^2
\big] \big)^{1/2} \,\big( \E [\|Y_{1}^{(p)}-Y_{2}^{(p)}\|_2^{2\beta}] \big)^{1/2}\,. 
 \] 
Then direct calculation together with (B3) yields 
\begin{align*}
 \E [\|
Y_1^{(p)}-Y_2^{(p)}
\|_2^{2\beta}] &= \E\Big[ \Big(
\int_0^1 (Y_1^{(p)}(t)-Y_2^{(p)}(t))^2 dt
\Big)^\beta\Big] \\
 & \le c\, \E\Big[\max_{0\le t\le 1}|Y(t)|^{2\beta}\Big] <\infty. 
\end{align*}
By concavity and (B3) we have  
\beao
\lefteqn{\E 
\big[\big|
 \|X_{1}^{(p)}-X_{2}^{(p)}\|_2^\beta -\|X_1-X_2\|_2^\beta\big|^2
\big]}\\ 
 &\le& \E \big[\|(X_{1}^{(p)}-X_1)-(X_{2}^{(p)}-X_2)\|_2^{2\beta}\big] \\
 &\le& c\,\E \Big[
 \Big(
\sum_{i=1}^p \int_{\Delta_i}(\Delta X_1(t,t_i] - \Delta X_2(t,t_i])^2 dt 
\Big)^\beta
\Big] \\
& \le& c\, p\, \delta_n^{\beta+\gamma_X'}. 
\eeao
A symmetry argument yields the corresponding result for $I_{12}$,
leading to  $\E[|I_{1}|]\le c\,p^{1/2}\,\delta_n^{(\beta+\gamma_X'\wedge
 \gamma_Y')/2}$, and the \rhs\ converges to 0 as $\nto$ by assumption.
\par
Thus we proved, under the assumption of a finite 
second moment for $X,Y$, that
\beam\label{eq:kha}
D_n=\E \Big[\Big|T_{n,\beta}(X^{(p)},Y^{(p)})-T_{n,\beta}(X,Y)\Big|\Big]
\le c\, \delta_n^{(\gamma_X\wedge \gamma_Y)(\beta
 \wedge 1)/2}\,,
\eeam 
and, under the assumption of a finite $(2\beta)$th moment of $X,Y$ for some 
$\beta\in(0,1)$, that
\beam\label{eq:khb}
D_n\le  c\,p^{1/2}\,\delta_n^{(\beta+\gamma_X'\wedge
 \gamma_Y')/2}\,.
\eeam
Since the \rhs s in \eqref{eq:kha} and  \eqref{eq:khb}
converge to zero by assumption we proved \eqref{consist:alternative}.
The conditions of Lemma \ref{lem:1}  are satisfied, implying 
$T_{n,\beta}(X^{(p)},X^{(p)})\stp T_\beta(X,X)$, $T_{n,\beta}(Y^{(p)},Y^{(p)})\stp T_\beta(Y,Y)$, and the \slln\ for $V$-statistics yields $T_{n,\beta}(X,X)\stp T_\beta(X,X)$, $T_{n,\beta}(Y,Y)\stp T_\beta(Y,Y)$. Then \eqref{consistR:alternative} follows.\\
{\bf Part (2).}
Under the growth conditions on $\delta_n\to 0$ we have in both cases,
see \eqref{eq:kha} and \eqref{eq:khb}, that
$\sqrt{n} D_n\to 0$, implying \eqref{eq:cltT_dep}.
\par
For the convergence in \eqref{eq:cltR_dep}, we observe that 
\beam\label{eq:decompo}
\lefteqn{\sqrt{n}\, \big(
R_{n,\beta}(X^{(p)},Y^{(p)})-R_{n,\beta}(X,Y)
\big)} \nonumber\\
&=& \dfrac{ \sqrt{n}\,\big(T_{n,\beta}(X^{(p)},Y^{(p)})- T_{n,\beta}(X,Y)\big)}{
A_n}\nonumber\\
&&+\sqrt{n} \,T_{n,\beta}(X,Y)\,
\dfrac{T_{n,\beta}(X,X)\,T_{n,\beta}(Y,Y)-T_{n,\beta}(X^{(p)},X^{(p)})\,
T_{n,\beta}(Y^{(p)},Y^{(p)})}{A_n\,B_n}\,,\nonumber\\
\eeam 
where 
\beao
A_n&=&  \sqrt{T_{n,\beta}(X^{(p)},X^{(p)})\,T_{n,\beta}(Y^{(p)},Y^{(p)})\, T_{n,\beta}(X,X)\,T_{n,\beta}(Y,Y)}\,,\\
B_n&=&\sqrt{T_{n,\beta}(X^{(p)},X^{(p)})\,T_{n,\beta}(Y^{(p)},Y^{(p)})}+
\sqrt{T_{n,\beta}(X,X)\,T_{n,\beta}(Y,Y)}\,.
\eeao
The quantities $A_n,B_n$  converge in \pro y to positive constants. Therefore
it suffices to show that
\beao
\sqrt{n} \big(T_{n,\beta}(X^{(p)},X^{(p)})- T_{n,\beta}(X,X)\big)&\stp& 0\,,\\
 \sqrt{n} \big(T_{n,\beta}(Y^{(p)},Y^{(p)})- T_{n,\beta}(Y,Y)\big)&\stp & 0\,,
\eeao
but these relations follow from \eqref{eq:cltT_dep} 
applied to $(X,X)$ and $(Y,Y)$, respectively.
\end{proof}
\bco
Assume the conditions of Proposition~\ref{prop:alternative}. Then $T_\beta(X,Y)>0$.
Moreover, if $\E[\|X\|_2^{2\beta}\|Y\|_2^{2\beta}]<\infty$ then 
the \seq\
\beao
\sqrt{n}(T_{n,\beta}(X^{(p)},Y^{(p)})-T_{\beta}(X,Y))
\eeao
has a mean-zero Gaussian limit.
If also $\E[\|X\|_2^{4\beta}+\|Y\|_2^{4\beta}]<\infty$ then
the \seq\ 
\beao
\sqrt{n}(R_{n,\beta}(X^{(p)},Y^{(p)})-R_{\beta}(X,Y))
\eeao
has a mean-zero Gaussian limit.
\eco
\begin{proof}
We proved in Theorem~4.2 
that $T_\beta(X,Y)>0$
\fif\ $X,Y$ are dependent. In view of Proposition~\ref{prop:alternative}
the statements will follow if we can show that
\beao
\sqrt{n}(T_{n,\beta}(X,Y)-T_{\beta}(X,Y))\quad\mbox{and}\quad 
\sqrt{n}(R_{n,\beta}(X,Y)-R_{\beta}(X,Y))
\eeao
have Gaussian limits. However, the \clt\ for $T_{n,\beta}(X,Y)$ follows
from the fact that  it is a non-degenerate $V$-statistic
(see the end of this proof) provided it has finite variance;
see \cite{arcones:gine:1992}.
This condition is ensured by $\E[\|X\|_2^{2\beta}\|Y\|_2^{2\beta}]<\infty$.
It is satisfied due to the assumptions.
\par
As regards the \clt\ for $R_{n,\beta}(X,Y)$, we can
follow an argument similar to the decomposition \eqref{eq:decompo}. We need to prove joint \asy\ Gaussianity of the vector \seq
\beam\label{eq:coco}
\sqrt{n}\big(T_{n,\beta}(X,Y)-T_\beta(X,Y), T_{n,\beta}(X,X)-T_\beta(X,X),
 T_{n,\beta}(Y,Y)-T_\beta(Y,Y)\big)\,,\;\;n\ge 1\,. \nonumber\\
\eeam
This \con\ follows if any linear combination of its components has a mean-zero 
Gaussian limit. By virtue of the moment condition
$\E[\|X\|_2^{4\beta}+\|Y\|_2^{4\beta}]<\infty$
each of the components in \eqref{eq:coco} is a non-degenerate $V$-statistic
with finite positive variance, hence they have Gaussian mean-zero limits,
and if there is joint \con\ the limit is non-degenerate.
However, any linear combination of these components is again a non-degenerate 
$V$-statistic and therefore the \clt\ for non-degenerate $V$-statistics
with a Gaussian limit applies to them as well. 
\par
Finally, we show that the kernel $h$ introduced in Lemma A.1
is non-degenerate, i.e., the conditional expectation
$\E [h(z_1,Z_2,Z_3,Z_4)]$ with deterministic $z_1=(x_1,y_1)$ and iid 
random vectors $Z_i=(X_i,Y_i)$, $i=2,3,4,$ is not a constant. 
By the symmetry of the kernel $h$ in (A.1) we have 
\begin{align*}
& \E[h(z_1,Z_2,Z_3,Z_4)] \\
&=
 \frac{1}{4} \E [f(z_1,Z_2,Z_3,Z_4)+f(Z_2,z_1,Z_3,Z_4) +f(Z_2,Z_3,z_1,Z_4)+f(Z_2,Z_3,Z_4,z_1) ] \\
&= \frac 14\E \big[ (\|x_1-X_1\|_2^\beta \|y_1-Y_1\|_2^\beta+ \|x_1-X_1\|_2^\beta
 \|Y_2-Y_3\|_2^\beta-2 \|x_1-X_1\|_2^\beta \|y_1-Y_2\|_2^\beta)\\
&\qquad + (\|x_1-X_1\|_2^\beta \|y_1-Y_1\|_2^\beta+ \|x_1-X_1\|_2^\beta
 \|Y_2-Y_3\|_2^\beta-2 \|x_1-X_1\|_2^\beta \|Y_1-Y_2\|_2^\beta)\\
&\qquad + (\|X_1-X_2\|_2^\beta \|Y_1-Y_2\|_2^\beta+ \|X_1-X_2\|_2^\beta
 \|y_1-Y_3\|_2^\beta-2 \|X_1-X_2\|_2^\beta \|Y_1-y_1\|_2^\beta)\\
&\qquad + (\|X_1-X_2\|_2^\beta \|Y_1-Y_2\|_2^\beta+ \|X_1-X_2\|_2^\beta
 \|y_1-Y_3\|_2^\beta-2 \|X_1-X_2\|_2^\beta \|Y_1-Y_3\|_2^\beta) \big]\\
&=  \frac{1}{2} 
\E\big[\big(\|x_1-X_1\|_2^\beta-\|X_1-X_2\|_2^\beta \big)\,\big( \|y_1-Y_1\|_2^\beta-\E[ \|y_1-Y_1\|_2^\beta]\big)
\big] \\
&\;\;\; -  \frac{1}{2} \E \big[\|X_1-x_1\|_2^\beta\big( \|Y_1-Y_2\|_2^\beta-\E[ \|Y_1-Y_2\|_2^\beta]
\big)\big]
+{\rm const.} \\
 \\
\end{align*}
We observe that the kernel 
\begin{equation} \label{e:kernel}
f(x,y)= \|x-y\|_2^\beta, \quad  0<\beta<2\,,
\end{equation}
is strongly negative definite on $L^2[0,1]$ in the sense of 
\cite{klebanov:2005}; see also \cite{lyons:2013}, Remark 3.19 and
Corollary 3.20. If
$$
\E\bigl[ \bigl( \| x_1-X_1\|_2^\beta - \| X_2-X_1\|_2^\beta\bigr)
\bigl(  \| y_1-Y_1\|_2^\beta - \E [\| y_1-Y_1\|_2^\beta]\bigr)\bigr] 
$$
is independent of $y_1$ for any fixed $x_1$, then for any $y_1$,
\beao
\lefteqn{\E\bigl[ \bigl( \| x_1-X_1\|_2^\beta - \| X_2-X_1\|_2^\beta\bigr)
 \| y_1-Y_1\|_2^\beta \bigr]}\\
&=& \E\bigl[ \bigl( \| x_1-X_3\|_2^\beta - \| X_2-X_3\|_2^\beta\bigr)
 \| y_1-Y_1\|_2^\beta \bigr]\,.
\eeao
We will apply Theorem 4.1 in \cite{klebanov:2005}.
Note, first of all, that this 
theorem extends immediately to signed measures.  Using this version of the theorem, we have for any Borel set $A\subset L^2[0,1]$ and $x_1\in L^2[0,1]$,
\beam\label{e:step.1}
\lefteqn{\P(Y_1\in A)\, \E\bigl[ \bigl( \| x_1-X_1\|_2^\beta - \|
X_2-X_1\|_2^\beta\bigr)\bigr]}\nonumber\\
&=& \E\bigl[ \1(Y_1\in A)\bigl( \| x_1-X_1\|_2^\beta - \|
X_2-X_1\|_2^\beta\bigr)\bigr]\,.
\eeam
In the light of Klebanov's
Theorem 4.1 we view \eqref{e:step.1}, once  again,  but 
this time as a function of $x_1\in L^2[0,1]$. There is a
difficulty, though, since there is a ``free term''. However, we can rewrite
\eqref{e:step.1} as
\begin{equation} \label{e:free.t}
\E\bigl[ \1(Y_1\in A)\,\| x_1-X_1\|_2^\beta\bigr]
= \E\bigl[ \1(Y_1\in A)\,  \| x_1-X_3\|_2^\beta \bigr] +
W\,,
\end{equation}
where
$$
W = \E\bigl[ \1(Y_1\in A)\,\|
X_2-X_1\|_2^\beta\bigr] - \P(Y_1\in A)\,  \E\big[\|
X_2-X_1\|_2^\beta\big] \,.
$$
Choosing $x_1=0$, we see that
$$
W =  \E\bigl[ \1(Y_1\in A)\,\|
X_1\|_2^\beta\bigr] - \P(Y_1\in A)\,  \E\big[\|
X_3\|_2^\beta\big] \,,
$$
so \eqref{e:free.t} reduces to 
\begin{equation} \label{e:step.2}
\E\bigl[ \1(Y_1\in A)\,\bigl( \|
x_1-X_1\|_2^\beta-\|X_1\|_2^\beta\bigr)\bigr] 
=  \E\bigl[ \1(Y_1\in A)\,\bigl( \|
x_1-X_3\|_2^\beta-\|X_3\|_2^\beta\bigr)\bigr] \,.
\end{equation} 
If the function $f$ in \eqref{e:kernel} is strongly negative definite
on $L^2[0,1]$, then so is the function
\begin{equation} \label{e:kernel.1}
\tilde f(x,y)= \|x-y\|_2^\beta -  \|y\|_2^\beta, \quad \ 0<\beta<2\,.
\end{equation}
Applying Klebanov's theorem to \eqref{e:step.2}, we obtain 
for any Borel set $B\subset L^2[0,1]$,
$$
\P(Y_1\in A,X_1\in B) = \P(Y_1\in A,X_3\in B)\,,
$$
so $X_1$ and $Y_1$ must be independent, contradicting our assumptions.
Therefore the \fct\ $\E[h(z_1,Z_2,Z_3,Z_4)]$ cannot be constant. This concludes the proof. 
\end{proof}

\section*{Acknowledgements}
The paper was finished when Thomas Mikosch
visited Ruhruniversit\"at Bochum (RUB) supported by an Alexander von Humboldt
Research Award. He would like to thank his colleagues in Bochum
for their hospitality. Munyea Matsui visited the University of Copenhagen
and RUB, and Laleh Tafakori the University of Copenhagen in 2017/2018
when major parts of this research were developed. They would like to thank
their host institutions for hospitality.

Herold Dehling's research was partially supported by the DFG through the
Collaborative Research Grant SFB 823. Muneya Matsui's research is partly supported by the JSPS Grant-in-Aid
for Young Scientists B (16k16023).
Gennady Samorodnitsky's research was partially supported by
 the ARO grants W911NF-12-10385 and W911NF-18 -10318 at Cornell University.
Laleh Tafakori would like to thank the Australian Research Council for support through Laureate Fellowship FL130100039.


\begin{thebibliography}{99}

\bibitem[Arcones and Gin\'e(1992)]{arcones:gine:1992}
{\sc Arcones, M.A. and Gin\'e, E.}\ (1992)
On the bootstrap of $U$ and $V$statistics. {\it Ann. Statist.}
{\bf 20}, 655--674.

\bibitem[Bickel and Freedman(1981)]{bickel:freedman:1981}
{\sc Bickel, P.J. and Freedman, D.A.}\ (1981)
Some \asy\ theory for the bootstrap. {\it Ann. Statist.} {\bf 9}, 1196--1217.

\bibitem[Billingsley(1968)]{billingsley:1968}
{\sc Billingsley, P.} (1968)
{\it Convergence of Probability Measures.} Wiley, New York.

\bibitem[Davis et al.(2018)]{dmmw:2016}
{\sc Davis, R.A., Matsui, M., Mikosch, T. and  Wan, P.}\ (2018)
Applications of distance correlation to time series.
{\em Bernoulli} {\bf 24}, 3087--3116.


\bibitem[Dehling and Mikosch(1994)]{dehling:mikosch:1994}
{\sc Dehling, H. and Mikosch, T.}\ (1994)
Random quadratic forms and the bootstrap for $U$-statistics.
{\em J. Multivar. Anal.} {\bf 51}, 392--413.

\bibitem[Feuerverger(1993)]{feuerverger:1993}
{\sc Feuerverger, A.}\ (1993) A consistent test for bivariate dependence. {\em Int. Stat. Rev.} {\bf 61}, 419--433.


\bibitem[Hoffmann-J\o rgensen(1994)]{hoffmann:1994}
{\sc Hoffmann-J\o rgensen, J.}\ (1994)
{\em Probability with a View Towards Statistics.} Chapman \& Hall, New
	York.

\bibitem[Klebanov(2005)]{klebanov:2005}
{\sc Klebanov, L.B.} \ (2005)
{\em $\mathfrak{N}$-Distances and their Applications.} Charles University Press, Prague.


\bibitem[Lyons(2013)]{lyons:2013}
{\sc Lyons, R.}\ (2013)
Distance covariance in metric spaces. 
{\em Ann. Probab.} {\bf 41}, 3284--3305. 

\bibitem[Lyons(2018)]{lyons:2018}
{\sc Lyons, R.}\ (2018) 
Errata to ``Distance covariance in metric spaces''. {\em Ann. Probab.}
{\bf 46}, 2400--2405.
 \bibitem[Marquardt(2006)]{Marquardt:2006}
{\sc Marquardt, T.}\ (2006). Fractional L$\acute{e}$vy processes with an application to long memory moving average processes.
{\em Bernoulli} {\bf 12}, 1099--1126.

\bibitem[Matsui et al.(2017)]{matsui:mikosch:samorodnitsky:2017}
{\sc Matsui, M., Mikosch, T. and Samorodnitsky G.}\ (2017) 
Distance covariance for stochastic processes. 
{\em Probab. Math. Statist.} {\bf 37}, 355--372. 


\bibitem[Samorodnitsky(2016)]{Samorodnitsky:2016}
{\sc Samorodnitsky, G.}\ (2016)
{\em Stochastic Processes and Long Range Dependence.}
Springer, Berlin.

\bibitem[Samorodnitsky and Taqqu(1994)]{samorodnitsky:taqqu:1994}
{\sc Samorodnitsky, G. and Taqqu, M.S.}\ (1994)
{\em Stable Non-Gaussian Random Processes.
Stochastic Models with Infinite Variance.} Chapman \& Hall, London. 

\bibitem[Sz\'ekely et al.(2007)]{szekely:rizzo:bakirov:2007}
{\sc Sz\'ekely, G.J., Rizzo, M.L. and Bakirov, N.K.}\ (2007)
Measuring and testing dependence by correlation of distances. {\em Ann. Statist.} {\bf 35}, 2769--2794.

\bibitem[Sz\' ekely and Rizzo(2009)]{szekely:rizzo:2009}
{\sc Sz\' ekely, G.J. and  Rizzo, M.L.}\ (2009)
Brownian distance covariance. {\em Ann. Appl. Stat.}  {\bf 3}, 1236--1265.

\bibitem[Sz\' ekely and Rizzo(2013)]{szekely:rizzo:2013}
{\sc Sz\'ekely, G.J. and Rizzo, M.L.} (2013)
The distance correlation $t$-test of independence in high dimension.
{\em J. Multivar. Anal.} {\bf 117}, 193--213.

\bibitem[Sz\' ekely and Rizzo(2014)]{szekely:rizzo:2014}
{\sc Sz\' ekely, G.J. and  Rizzo, M.L.}\ (2014)
Partial distance correlation with methods for dissimilarities. {\em
	Ann. Statist.} {\bf 42}, 2382--2412.

\bibitem[Serfling(1980)]{serfling:1980}
{\sc Serfling, R.J.}\ (1980)
{\em Approximation Theorems of Mathematical Statistics.} 
John Wiley \& Sons, New York. 





\end{thebibliography}
\end{document}